\documentclass[12pt,reqno]{amsart}
\ifx\diagram\isundefined\else\expandafter\endinput
\fi

\edef\cdrestoreat{
\noexpand\catcode`\noexpand\@=\the\catcode`\@
\noexpand\catcode`\noexpand\#=\the\catcode`\#
\noexpand\catcode`\noexpand\$=\the\catcode`\$
\noexpand\catcode`\noexpand\<=\the\catcode`\<
\noexpand\catcode`\noexpand\>=\the\catcode`\>
\noexpand\catcode`\noexpand\:=\the\catcode`\:
\noexpand\catcode`\noexpand\;=\the\catcode`\;
\noexpand\catcode`\noexpand\!=\the\catcode`\!
\noexpand\catcode`\noexpand\?=\the\catcode`\?
\noexpand\catcode`\noexpand\+=\the\catcode'53
}\catcode`\@=11 \catcode`\#=6 \catcode`\<=12 \catcode`\>=12 \catcode'53=12
\catcode`\:=12 \catcode`\;=12 \catcode`\!=12 \catcode`\?=12

\ifx\diagram@help@messages\CD@qK\let\diagram@help@messages y\fi

\def\cdps@Rokicki#1{\special{ps:#1}}\let\cdps@dvips\cdps@Rokicki\let
\cdps@RadicalEye\cdps@Rokicki\let\CD@HB\cdps@Rokicki\let\CD@IK\cdps@Rokicki
\let\CD@HB\cdps@Rokicki
\def\cdps@Bechtolsheim#1{\special{dvitps: Literal "#1"}}%
\let\cdps@dvitps\cdps@Bechtolsheim\let\cdps@IntegratedComputerSystems
\cdps@Bechtolsheim
\def\cdps@Clark#1{\special{dvitops: inline #1}}
\let\cdps@dvitops\cdps@Clark
\let\cdps@OzTeX\empty\let\cdps@oztex\empty\let\cdps@Trevorrow\empty
\def\cdps@Coombes#1{\special{ps-string #1}}

\count@=\year\multiply\count@12 \advance\count@\month
\ifnum\count@>24382 
\ifnum\count@>24385 
\errhelp{You may press RETURN and carry on for the time being.}\errmessage{! remain
compatible with your files. Please get it NOW.}\fi\fi

\def\CD@DE{\global\let}\def\CD@RH{\outer\def}

{\escapechar\m@ne\xdef\CD@o{\string\{}\xdef\CD@yC{\string\}}
\catcode`\&=4 \CD@DE\CD@Q=&\xdef\CD@S{\string\&}
\catcode`\$=3 \CD@DE\CD@k=$\CD@DE\CD@ND=$
\xdef\CD@nC{\string\$}\gdef\CD@LG{$$}
\catcode`\_=8 \CD@DE\CD@lJ=_
\obeylines\catcode`\^=7 \CD@DE\@super=^
\ifnum\newlinechar=10 \gdef\CD@uG{^^J}
\else\ifnum\newlinechar=13 \gdef\CD@uG{^^M}
\else\ifnum\newlinechar=-1 \gdef\CD@uG{^^J}
\else\CD@DE\CD@uG\space\expandafter\message{! input error: \noexpand
\newlinechar\space is ASCII \the\newlinechar, not LF=10 or CR=13.}
\fi\fi\fi}

\mathchardef\lessthan='30474 \mathchardef\greaterthan='30476

\ifx\tenln\CD@qK
\font\tenln=line10\relax
\fi\ifx\tenlnw\CD@qK\ifx\tenln\nullfont\let\tenlnw\nullfont\else
\font\tenlnw=linew10\relax
\fi\fi

\ifx\inputlineno\CD@qK\csname newcount\endcsname\inputlineno\inputlineno\m@ne
\fi

\def\cd@shouldnt#1{\CD@KB{* THIS (#1) SHOULD NEVER HAPPEN! *}}

\def\get@round@pair#1(#2,#3){#1{#2}{#3}}
\def\get@square@arg#1[#2]{#1{#2}}
\def\CD@AE#1{\CD@PK\let\CD@DH\CD@@E\CD@@E#1,],}
\def\CD@m{[}\def\CD@RD{]}\def\commdiag#1{{\let\enddiagram\relax\diagram[]#1%
\enddiagram}}

\def\CD@BF{{\ifx\CD@EH[\aftergroup\get@square@arg\aftergroup\CD@YH\else
\aftergroup\CD@JH\fi}}
\def\CD@CF#1#2{\def\CD@YH{#1}\def\CD@JH{#2}\futurelet\CD@EH\CD@BF}

\def\CD@KK{|}

\def\CD@PB{
\tokcase\CD@DD:\CD@y\break@args;\catcase\@super:\upper@label;\catcase\CD@lJ:%
\lower@label;\tokcase{~}:\middle@label;
\tokcase<:\CD@iF;
\tokcase>:\CD@iI;
\tokcase(:\CD@BC;
\tokcase[:\optional@;
\tokcase.:\CD@JJ;
\catcase\space:\eat@space;\catcase\bgroup:\positional@;\default:\CD@@A
\break@args;\endswitch}

\def\switch@arg{
\catcase\@super:\upper@label;\catcase\CD@lJ:\lower@label;\tokcase[:\optional@
;
\tokcase.:\CD@JJ;
\catcase\space:\eat@space;\catcase\bgroup:\positional@;\tokcase{~}:%
\middle@label;
\default:\CD@y\break@args;\endswitch}

\let\CD@tJ\relax\ifx\protect\CD@qK\let\protect\relax\fi\ifx\AtEndDocument
\CD@qK\def\CD@PG{\CD@gB}\def\CD@GF#1#2{}\else\def\CD@PG#1{\edef\CD@CH{#1}%
\expandafter\CD@oC\CD@CH\CD@OD}\def\CD@oC#1\CD@OD{\AtEndDocument{\typeout{%
\CD@tA: #1}}}\def\CD@GF#1#2{\gdef#1{#2}\AtEndDocument{#1}}\fi\def\CD@ZA#1#2{%
\def#1{\CD@PG{#2\CD@mD\CD@W}\CD@DE#1\relax}}\def\CD@uF#1\repeat{\def\CD@p{#1}%
\CD@OF}\def\CD@OF{\CD@p\relax\expandafter\CD@OF\fi}\def\CD@sF#1\repeat{\def
\CD@q{#1}\CD@PF}\def\CD@PF{\CD@q\relax\expandafter\CD@PF\fi}\def\CD@tF#1%
\repeat{\def\CD@r{#1}\CD@QF}\def\CD@QF{\CD@r\relax\expandafter\CD@QF\fi}\def
\CD@tG#1#2#3{\def#2{\let#1\iftrue}\def#3{\let#1\iffalse}#3}\if y%
\diagram@help@messages\def\CD@rG#1#2{\csname newtoks\endcsname#1#1=%
\expandafter{\csname#2\endcsname}}\else\csname newtoks\endcsname\no@cd@help
\no@cd@help{See the manual}\def\CD@rG#1#2{\let#1\no@cd@help}\fi\chardef\CD@lF
=1 \chardef\CD@lI=2 \chardef\CD@MH=5 \chardef\CD@tH=6 \chardef\CD@sH=7
\chardef\CD@PC=9 \dimendef\CD@hI=2 \dimendef\CD@hF=3 \dimendef\CD@mF=4
\dimendef\CD@mI=5 \dimendef\CD@wJ=6 \dimendef\CD@tI=8 \dimendef\CD@sI=9
\skipdef\CD@uB=1 \skipdef\CD@NF=2 \skipdef\CD@tB=3 \skipdef\CD@ZE=4 \skipdef
\countdef\CD@JC=9 \countdef\CD@gD=8 \countdef\CD@A=7 \def\sdef#1#2{\def#1{#2}%
}\def\CD@L#1{\expandafter\aftergroup\csname#1\endcsname}\def\CD@RC#1{%
\expandafter\def\csname#1\endcsname}\def\CD@sD#1{\expandafter\gdef\csname#1%
\endcsname}\def\CD@vC#1{\expandafter\edef\csname#1\endcsname}\def\CD@nF#1#2{%
\expandafter\let\csname#1\expandafter\endcsname\csname#2\endcsname}\def\CD@EE
#1#2{\expandafter\CD@DE\csname#1\expandafter\endcsname\csname#2\endcsname}%
\def\CD@AK#1{\csname#1\endcsname}\def\CD@XJ#1{\expandafter\show\csname#1%
\endcsname}\def\CD@ZJ#1{\expandafter\showthe\csname#1\endcsname}\def\CD@WJ#1{%
\expandafter\showbox\csname#1\endcsname}\def\CD@tA{Commutative Diagram}\edef
\CD@kH{\string\par}\edef\CD@dC{\string\diagram}\edef\CD@HD{\string\enddiagram
}\edef\CD@EC{\string\\}\def\CD@eF{LaTeX}\ifx\@ignoretrue\CD@qK\expandafter
\CD@tG\csname if@ignore\endcsname\ignore@true\ignore@false\def\@ignoretrue{%
\global\ignore@true}\def\@ignorefalse{\global\ignore@false}\fi

\def\CD@g{{\ifnum0=`}\fi}\def\CD@wC{\ifnum0=`{\fi}}\def\catcase#1:{\ifcat
\noexpand\CD@EH#1\CD@tJ\expandafter\CD@kC\else\expandafter\CD@dJ\fi}\def
\tokcase#1:{\ifx\CD@EH#1\CD@tJ\expandafter\CD@kC\else\expandafter\CD@dJ\fi}%
\def\CD@kC#1;#2\endswitch{#1}\def\CD@dJ#1;{}\let\endswitch\relax\def\default:%
#1;#2\endswitch{#1}\ifx\at@\CD@qK\def\at@{@}\fi\edef\CD@P{\CD@o pt\CD@yC}%
\CD@RC{\CD@P>}#1>#2>{\CD@z\rTo\sp{#1}\sb{#2}\CD@z}\CD@RC{\CD@P<}#1<#2<{\CD@z
\lTo\sp{#1}\sb{#2}\CD@z}\CD@RC{\CD@P)}#1)#2){\CD@z\rTo\sp{#1}\sb{#2}\CD@z}%
\CD@RC{\CD@P(}#1(#2({\CD@z\lTo\sp{#1}\sb{#2}\CD@z}
\def\CD@O{\def\endCD{\enddiagram}\CD@RC{\CD@P A}##1A##2A{\uTo<{##1}>{##2}%
\CD@z\CD@z}\CD@RC{\CD@P V}##1V##2V{\dTo<{##1}>{##2}\CD@z\CD@z}\CD@RC{\CD@P=}{%
\CD@z\hEq\CD@z}\CD@RC{\CD@P\CD@KK}{\vEq\CD@z\CD@z}\CD@RC{\CD@P\string\vert}{%
\vEq\CD@z\CD@z}\CD@RC{\CD@P.}{\CD@z\CD@z}\let\CD@z\CD@Q}\def\CD@IE{\let\tmp
\CD@JE\ifcat A\noexpand\CD@CH\else\ifcat=\noexpand\CD@CH\else\ifcat\relax
\noexpand\CD@CH\else\let\tmp\at@\fi\fi\fi\tmp}\def\CD@JE#1{\CD@nF{tmp}{\CD@P
\string#1}\ifx\tmp\relax\def\tmp{\at@#1}\fi\tmp}\def\CD@z{}\begingroup
\aftergroup\def\aftergroup\CD@T\aftergroup{\aftergroup\def\catcode`\@\active
\aftergroup @\endgroup{\futurelet\CD@CH\CD@IE}}\def\CD@uK#1{\CD@nF{x}{tex_#1:%
D}\ifx\x\relax\else\CD@nF{#1}{x}\fi} \def\CD@vK{\CD@uK{par}\CD@uK{everypar}%
\CD@uK{noindent}\CD@uK{parskip}\CD@uK{unskip}\CD@uK{hskip}\CD@uK{indent}}

\newcount\CD@uA\newcount\CD@vA\newcount\CD@wA\newcount\CD@xA\newdimen\CD@OA
\newdimen\CD@PA\CD@tG\CD@gE\CD@@A\CD@y\CD@tG\CD@hE\CD@EA\CD@BA\newdimen\CD@RA
\newdimen\CD@SA\newcount\CD@yA\newcount\CD@zA\newdimen\CD@QA\newbox\CD@DA
\CD@tG\CD@lE\CD@dA\CD@bA\newcount\CD@LH\newcount\CD@TC\def\CD@V#1#2{\ifdim#1<%
#2\relax#1=#2\relax\fi}\def\CD@X#1#2{\ifdim#1>#2\relax#1=#2\relax\fi}%
\newdimen\CD@XH\CD@XH=1sp \newdimen\CD@zC\CD@zC\z@\def\CD@cJ{\ifdim\CD@zC=1em%
\else\CD@nJ\fi}\def\CD@nJ{\CD@zC1em\def\CD@NC{\fontdimen8\textfont3 }\CD@@J
\CD@NJ\setbox0=\vbox{\CD@t\noindent\CD@k\null\penalty-9993\null\CD@ND\null
\endgraf\setbox0=\lastbox\unskip\unpenalty\setbox1=\lastbox\global\setbox
\CD@IG=\hbox{\unhbox0\unskip\unskip\unpenalty\setbox0=\lastbox}\global\setbox
\CD@KG=\hbox{\unhbox1\unskip\unpenalty\setbox1=\lastbox}}}\newdimen\CD@@I
\CD@@I=1true in \divide\CD@@I300 \def\CD@zH#1{\multiply#1\tw@\advance#1\ifnum
#1<\z@-\else+\fi\CD@@I\divide#1\tw@\divide#1\CD@@I\multiply#1\CD@@I}\def
\MapBreadth{\afterassignment\CD@gI\CD@LF}\newdimen\CD@LF\newdimen\CD@oI\def
\CD@gI{\CD@oI\CD@LF\CD@V\CD@@I{4\CD@XH}\CD@X\CD@@I\p@\CD@zH\CD@oI\ifdim\CD@LF
>\z@\CD@V\CD@oI\CD@@I\fi\CD@cJ}\def\CD@RJ#1{\CD@zD\count@\CD@@I#1\ifnum
\count@>\z@\divide\CD@@I\count@\fi\CD@gI\CD@NJ}\def\CD@NJ{\dimen@\CD@QC
\count@\dimen@\divide\count@5\divide\count@\CD@@I\edef\CD@OC{\the\count@}}%
\def\CD@AJ{\CD@QJ\z@}\def\CD@QJ#1{\CD@tI\axisheight\advance\CD@tI#1\relax
\advance\CD@tI-.5\CD@oI\CD@zH\CD@tI\CD@sI-\CD@tI\advance\CD@tI\CD@LF}%
\newdimen\CD@DC\CD@DC\z@\newdimen\CD@eJ\CD@eJ\z@\def\CD@CJ#1{\CD@sI#1\relax
\CD@tI\CD@sI\advance\CD@tI\CD@LF\relax}\def\horizhtdp{height\CD@tI depth%
\CD@sI}\def\axisheight{\fontdimen22\the\textfont\tw@}\def\script@axisheight{%
\fontdimen22\the\scriptfont\tw@}\def\ss@axisheight{\fontdimen22\the
\scriptscriptfont\tw@}\def\CD@NC{0.4pt}\def\CD@VK{\fontdimen3\textfont\z@}%
\def\CD@UK{\fontdimen3\textfont\z@}\newdimen\PileSpacing\newdimen\CD@nA\CD@nA
\z@\def\CD@RG{\ifincommdiag1.3em\else2em\fi}\newdimen\CD@YB\def\CellSize{%
\afterassignment\CD@kB\DiagramCellHeight}\newdimen\DiagramCellHeight
\DiagramCellHeight-\maxdimen\newdimen\DiagramCellWidth\DiagramCellWidth-%
\maxdimen\def\CD@kB{\DiagramCellWidth\DiagramCellHeight}\def\CD@QC{3em}%
\newdimen\MapShortFall\def\MapsAbut{\MapShortFall\z@\objectheight\z@
\objectwidth\z@}\newdimen\CD@iA\CD@iA\z@\CD@tG\CD@vE\CD@aB\CD@ZB\expandafter
\ifx\expandafter\iftrue\csname ifUglyObsoleteDiagrams\endcsname\CD@ZB\else
\CD@aB\fi\CD@nF{ifUglyObsoleteDiagrams}{relax}\newif\ifUglyObsoleteDiagrams
\def\CD@nK{\CD@aB\UglyObsoleteDiagramsfalse}\def\CD@oK{\CD@ZB
\UglyObsoleteDiagramstrue}\CD@vE\CD@nK\else\CD@oK\fi\CD@tG\CD@hK\CD@dK\CD@cK
\CD@cK\def\CD@sK{\ifx\pdfoutput\CD@qK\else\ifx\pdfoutput\relax\else\ifnum
\pdfoutput>\z@\CD@pK\fi\fi\fi} \def\CD@pK{\global\CD@dK\global\CD@aB\global
\UglyObsoleteDiagramsfalse\global\let\CD@n\empty\global\let\CD@oK\relax
\global\let\CD@pK\relax\global\let\CD@sK\relax}\def\CD@tK#1{}\ifx\pdfliteral\CD@qK\else\ifx\pdfliteral\relax\else\let\CD@tK
\pdfliteral\fi\fi\ifx\XeTeXrevision\CD@qK\else\ifx\XeTeXrevision\relax\else
\ifdim\XeTeXrevision pt<.996pt \expandafter\message{! XeTeX version
\XeTeXrevision\space does not support PDF literals, so diagonals will not work%
!}\else\expandafter\message{RUNNING UNDER XETEX \XeTeXrevision}\CD@pK\fi\fi
\fi\CD@sK\def\newarrowhead{\CD@mG h\CD@BG\CD@GG>}\def\newarrowtail{\CD@mG t%
\CD@BG\CD@GG>}\def\newarrowmiddle{\CD@mG m\CD@BG\hbox@maths\empty}\def
\newarrowfiller{\CD@mG f\CD@bE\CD@MK-}\def\CD@mG#1#2#3#4#5#6#7#8#9{\CD@RC{r#1%
:#5}{#2{#6}}\CD@RC{l#1:#5}{#2{#7}}\CD@RC{d#1:#5}{#3{#8}}\CD@RC{u#1:#5}{#3{#9}%
}\CD@vC{-#1:#5}{\expandafter\noexpand\csname-#1:#4\endcsname\noexpand\CD@MC}%
\CD@vC{+#1:#5}{\expandafter\noexpand\csname+#1:#4\endcsname\noexpand\CD@MC}}%
\CD@ZA\CD@MC{\CD@eF\space diagonals are used unless PostScript is set}\def
\defaultarrowhead#1{\edef\CD@sJ{#1}\CD@@J}\def\CD@@J{\CD@IJ\CD@sJ<>ht\CD@IJ
\CD@sJ<>th}\def\CD@IJ#1#2#3#4#5{\CD@HJ{r#4}{#3}{l#5}{#2}{r#4:#1}\CD@HJ{r#5}{#%
2}{l#4}{#3}{l#4:#1}\CD@HJ{d#4}{#3}{u#5}{#2}{d#4:#1}\CD@HJ{d#5}{#2}{u#4}{#3}{u%
#4:#1}}\def\CD@HJ#1#2#3#4#5{\begingroup\aftergroup\CD@GJ\CD@L{#1+:#2}\CD@L{#1%
:#2}\CD@L{#3:#4}\CD@L{#5}\endgroup}\def\CD@GJ#1#2#3#4{\csname newbox%
\endcsname#1\def#2{\copy#1}\def#3{\copy#1}\setbox#1=\box\voidb@x}\def\CD@sJ{}%
\CD@@J\def\CD@GJ#1#2#3#4{\setbox#1=#4}\ifx\tenln\nullfont\def\CD@sJ{vee}\else
\let\CD@sJ\CD@eF\fi\def\CD@xF#1#2#3{\begingroup\aftergroup\CD@wF\CD@L{#1#2:#3%
#3}\CD@L{#1#2:#3}\aftergroup\CD@yF\CD@L{#1#2:#3-#3}\CD@L{#1#2:#3}\endgroup}%
\def\CD@wF#1#2{\def#1{\hbox{\rlap{#2}\kern.4\CD@zC#2}}}\def\CD@yF#1#2{\def#1{%
\hbox{\rlap{#2}\kern.4\CD@zC#2\kern-.4\CD@zC}}}\CD@xF lh>\CD@xF rt>\CD@xF rh<%
\CD@xF rt<\def\CD@yF#1#2{\def#1{\hbox{\kern-.4\CD@zC\rlap{#2}\kern.4\CD@zC#2}%
}}\CD@xF rh>\CD@xF lh<\CD@xF lt>\CD@xF lt<\def\CD@wF#1#2{\def#1{\vbox{\vbox to%
\z@{#2\vss}\nointerlineskip\kern.4\CD@zC#2}}}\def\CD@yF#1#2{\def#1{\vbox{%
\vbox to\z@{#2\vss}\nointerlineskip\kern.4\CD@zC#2\kern-.4\CD@zC}}}\CD@xF uh>%
\CD@xF dt>\CD@xF dh<\CD@xF dt<\def\CD@yF#1#2{\def#1{\vbox{\kern-.4\CD@zC\vbox
to\z@{#2\vss}\nointerlineskip\kern.4\CD@zC#2}}}\CD@xF dh>\CD@xF ut>\CD@xF uh<%
\CD@xF ut<\def\CD@BG#1{\hbox{\mathsurround\z@\offinterlineskip\CD@k\mkern-1.5%
mu{#1}\mkern-1.5mu\CD@ND}}\def\hbox@maths#1{\hbox{\CD@k#1\CD@ND}}\def\CD@GG#1%
{\hbox to\CD@LF{\setbox0=\hbox{\offinterlineskip\mathsurround\z@\CD@k{#1}%
\CD@ND}\dimen0.5\wd0\advance\dimen0-.5\CD@oI\CD@zH{\dimen0}\kern-\dimen0%
\unhbox0\hss}}\def\CD@sB#1{\hbox to2\CD@LF{\hss\offinterlineskip\mathsurround
\z@\CD@k{#1}\CD@ND\hss}}\def\CD@vF#1{\hbox{\mathsurround\z@\CD@k{#1}\CD@ND}}%
\def\CD@bE#1{\hbox{\kern-.15\CD@zC\CD@k{#1}\CD@ND\kern-.15\CD@zC}}\def\CD@MK#%
1{\vbox{\offinterlineskip\kern-.2ex\CD@GG{#1}\kern-.2ex}}\def\@fillh{%
\xleaders\vrule\horizhtdp}\def\@fillv{\xleaders\hrule width\CD@LF}\CD@nF{rf:-%
}{@fillh}\CD@nF{lf:-}{@fillh}\CD@nF{df:-}{@fillv}\CD@nF{uf:-}{@fillv}\CD@nF{%
rh:}{null}\CD@nF{rm:}{null}\CD@nF{rt:}{null}\CD@nF{lh:}{null}\CD@nF{lm:}{null%
}\CD@nF{lt:}{null}\CD@nF{dh:}{null}\CD@nF{dm:}{null}\CD@nF{dt:}{null}\CD@nF{%
uh:}{null}\CD@nF{um:}{null}\CD@nF{ut:}{null}\CD@nF{+h:}{null}\CD@nF{+m:}{null%
}\CD@nF{+t:}{null}\CD@nF{-h:}{null}\CD@nF{-m:}{null}\CD@nF{-t:}{null}\def
\CD@@D{\hbox{\vrule height 1pt depth-1pt width 1pt}}\CD@RC{rf:}{\CD@@D}\CD@nF
{lf:}{rf:}\CD@nF{+f:}{rf:}\CD@RC{df:}{\CD@@D}\CD@nF{uf:}{df:}\CD@nF{-f:}{df:}%
\def\CD@BD{\CD@U\null\CD@@D\null\CD@@D\null}\edef\CD@lG{\string\newarrow}\def
\newarrow#1#2#3#4#5#6{\begingroup\edef\@name{#1}\edef\CD@oJ{#2}\edef\CD@iD{#3%
}\edef\CD@QG{#4}\edef\CD@jD{#5}\edef\CD@LE{#6}\let\CD@HE\CD@sG\let\CD@FK
\CD@BH\let\@x\CD@AH\ifx\CD@oJ\CD@iD\let\CD@oJ\empty\fi\ifx\CD@LE\CD@jD\let
\CD@LE\empty\fi\def\CD@LI{r}\def\CD@SF{l}\def\CD@IC{d}\def\CD@yJ{u}\def\CD@gH
{+}\def\@m{-}\ifx\CD@iD\CD@jD\ifx\CD@QG\CD@iD\let\CD@QG\empty\fi\ifx\CD@LE
\empty\ifx\CD@iD\CD@aE\let\@x\CD@yG\else\let\@x\CD@zG\fi\fi\else\edef\CD@a{%
\CD@iD\CD@oJ}\ifx\CD@a\empty\ifx\CD@QG\CD@jD\let\CD@QG\empty\fi\fi\fi\ifmmode
\aftergroup\CD@kG\else\CD@@A\CD@oB rh{head\space\space}\CD@LE\CD@oB rf{filler%
}\CD@iD\CD@oB rm{middle}\CD@QG\ifx\CD@jD\CD@iD\else\CD@oB rf{filler}\CD@jD\fi
\CD@oB rt{tail\space\space}\CD@oJ\CD@gE\CD@HE\CD@FK\@x\CD@nG l-2+2{lu}{nw}%
\NorthWest\CD@nG r+2+2{ru}{ne}\NorthEast\CD@nG l-2-2{ld}{sw}\SouthWest\CD@nG r%
+2-2{rd}{se}\SouthEast\else\aftergroup\CD@b\CD@L{r\@name}\fi\fi\endgroup}\def
\CD@sG{\CD@vG\CD@LI\CD@SF rl\Horizontal@Map}\def\CD@BH{\CD@vG\CD@IC\CD@yJ du%
\Vertical@Map}\def\CD@AH{\CD@vG\CD@gH\@m+-\Vector@Map}\def\CD@yG{\CD@vG\CD@gH
\@m+-\Slant@Map}\def\CD@zG{\CD@vG\CD@gH\@m+-\Slope@Map}\catcode`\/=\active
\def\CD@vG#1#2#3#4#5{\CD@jG#1#3#5t:\CD@oJ/f:\CD@iD/m:\CD@QG/f:\CD@jD/h:\CD@LE
//\CD@jG#2#4#5h:\CD@LE/f:\CD@jD/m:\CD@QG/f:\CD@iD/t:\CD@oJ//}\def\CD@jG#1#2#3%
#4//{\edef\CD@fG{#2}\aftergroup\sdef\CD@L{#1\@name}\aftergroup{\aftergroup#3%
\CD@M#4//\aftergroup}}\def\CD@M#1/{\edef\CD@EH{#1}\ifx\CD@EH\empty\else\CD@L{%
\CD@fG#1}\expandafter\CD@M\fi}\catcode`\/=12 \def\CD@nG#1#2#3#4#5#6#7#8{%
\aftergroup\sdef\CD@L{#6\@name}\aftergroup{\CD@L{#2\@name}\if#2#4\aftergroup
\CD@CI\else\aftergroup\CD@BI\fi\CD@L{#1\@name}%
\aftergroup(\aftergroup#3\aftergroup,\aftergroup#5\aftergroup)\aftergroup}}%
\def\CD@oB#1#2#3#4{\expandafter\ifx\csname#1#2:#4\endcsname\relax\CD@y\CD@gB{%
arrow#3 "#4" undefined}\fi}\CD@rG\CD@VE{All five components must be defined
before an arrow.}\CD@rG\CD@SE{\CD@lG, unlike \string\HorizontalMap, is a
declaration.}\def\CD@b#1{\CD@YA{Arrows \string#1 etc could not be defined}%
\CD@VE}\def\CD@kG{\CD@YA{misplaced \CD@lG}\CD@SE}\def\newdiagramgrid#1#2#3{%
\CD@RC{cdgh@#1}{#2,],}
\CD@RC{cdgv@#1}{#3,],}}
\CD@tG\ifincommdiag\incommdiagtrue\incommdiagfalse\CD@tG\CD@@F\CD@IF\CD@HF
\newcount\CD@VA\CD@VA=0 \def\CD@yH{\CD@VA6 }\def\CD@OB{\CD@VA1 \global\CD@yA1
\CD@DE\CD@YF\empty}\def\CD@YF{}\def\CD@nB#1{\relax\CD@MD\edef\CD@vJ{#1}%
\begingroup\CD@rE\else\ifcase\CD@VA\ifmmode\else\CD@YG\CD@E0\fi\or\CD@cE5\or
\CD@YG\CD@F5\or\CD@YG\CD@B5\or\CD@YG\CD@B5\or\CD@YG\CD@C5\or\CD@cE7\or\CD@YG
\CD@D7\fi\fi\endgroup\xdef\CD@YF{#1}}\def\CD@pB#1#2#3#4#5{\relax\CD@MD\xdef
\CD@vJ{#4}\begingroup\ifnum\CD@VA<#1 \expandafter\CD@cE\ifcase\CD@VA0\or#2\or
#3\else#2\fi\else\ifnum\CD@VA<6 \CD@tJ\CD@YG\CD@B#2\else\CD@YG\CD@G#2\fi\fi
\endgroup\CD@DE\CD@YF\CD@vJ\ifincommdiag\let\CD@ZD#5\else\let\CD@ZD\CD@LK\fi}%
\def\CD@yI{\global\CD@yA=\ifnum\CD@VA<5 1\else2\fi\relax}\def\CD@OI{\CD@VA
\CD@yA}\def\CD@cE#1{\aftergroup\CD@VA\aftergroup#1\aftergroup\relax}\def
\CD@HH{\def\CD@nB##1{\relax}\let\CD@pB\CD@FH\let\CD@yH\relax\let\CD@OB\relax
\let\CD@yI\relax\let\CD@OI\relax}\def\CD@FH#1#2#3#4#5{\ifincommdiag\let\CD@ZD
#5\else\xdef\CD@vJ{#4}\let\CD@ZD\CD@LK\fi}\def\CD@YG#1{\aftergroup#1%
\aftergroup\relax\CD@cE}\def\CD@B{\CD@YE\CD@S\CD@ME\CD@Q}\def\CD@G{\CD@YE{%
\CD@yC\CD@S}\CD@XE\CD@QD\CD@Q}\def\CD@F{\CD@YE{*\CD@S}\CD@RE\clubsuit\CD@Q}%
\def\CD@C{\CD@YE{\CD@S*\CD@S}\CD@RE\CD@Q\clubsuit\CD@Q}\def\CD@D{\CD@YE\CD@EC
\CD@TE\\}\def\CD@E{\CD@YE\CD@nC\CD@QE\CD@k}\def\CD@LK{\CD@YA{\CD@vJ\space
ignored \CD@dH}\CD@WE}\def\CD@FE{}\def\CD@d{\CD@YA{maps must never be enclosed
in braces}\CD@OE}\def\CD@dH{outside diagram}\def\CD@FC{\string\HonV, \string
\VonH\space and \string\HmeetV}\CD@rG\CD@ME{The way that horizontal and
vertical arrows are terminated implicitly means\CD@uG that they cannot be
mixed with each other or with \CD@FC.}\CD@rG\CD@XE{\string\pile\space is for
parallel horizontal arrows; verticals can just be put together in\CD@uG a cell%
. \CD@FC\space are not meaningful in a \string\pile.}\CD@rG\CD@RE{The
horizontal maps must point to an object, not each other (I've put in\CD@uG one
which you're unlikely to want). Use \string\pile\space if you want them
parallel.}\CD@rG\CD@TE{Parallel horizontal arrows must be in separate layers
of a \string\pile.}\CD@rG\CD@QE{Horizontal arrows may be used \CD@dH s, but
must still be in maths.}\CD@rG\CD@WE{Vertical arrows, \CD@FC\space\CD@dH s don%
't know where\CD@uG where to terminate.}\CD@rG\CD@OE{This prevents them from
stretching correctly.}\def\CD@YE#1{\CD@YA{"#1" inserted \ifx\CD@YF\empty
before \CD@vJ\else between \CD@YF\ifx\CD@YF\CD@vJ s\else\space and \CD@vJ\fi
\fi}}\count@=\year\multiply\count@12 \advance\count@\month\ifnum\count@>24391
\expandafter\endinput\fi\def
\Horizontal@Map{\CD@nB{horizontal map}\CD@LC\CD@TJ\CD@qD}\def\CD@TJ{\CD@GB-%
9999 \let\CD@ZD\CD@XD\ifincommdiag\else\CD@cJ\ifinpile\else\skip2\z@ plus 1.5%
\CD@VK minus .5\CD@UK\skip4\skip2 \fi\fi\let\CD@kD\@fillh\CD@nF{fill@dot}{rf:%
.}}\def\Vector@Map{\CD@HK4}\def\Slant@Map{\CD@HK{\CD@EF255\else6\fi}}\def
\Slope@Map{\CD@HK\CD@OC}\def\CD@HK#1#2#3#4#5#6{\CD@LC\def\CD@WK{2}\def\CD@aK{%
2}\def\CD@ZK{1}\def\CD@bK{1}\let\Horizontal@Map\CD@nI\def\CD@OG{#1}\def\CD@NI
{\CD@U#2#3#4#5#6}}\def\CD@nI{\CD@TJ\CD@JB\let\CD@ZD\CD@TD\CD@qD}\CD@tG\CD@pE
\CD@rA\CD@qA\CD@rA\def\cds@missives{\CD@rA}\def\CD@TD{\CD@vE\let\CD@OG\CD@OC
\CD@x\CD@zE\CD@WF\fi\setbox0\hbox{\incommdiagfalse\CD@HI}\CD@pE\CD@aD\else
\global\CD@YC\CD@bD\fi\ifvoid6 \ifvoid7 \CD@eE\fi\fi\CD@zE\else\CD@BD\global
\CD@YC\let\CD@CG\CD@IH\CD@YD\fi\else\CD@NI\CD@MI\global\CD@YC\CD@YD\fi}\def
\CD@LC{\begingroup\dimen1=\MapShortFall\dimen2=\CD@RG\dimen5=\MapShortFall
\setbox3=\box\voidb@x\setbox6=\box\voidb@x\setbox7=\box\voidb@x\CD@pD
\mathsurround\z@\skip2\z@ plus1fill minus 1000pt\skip4\skip2 \CD@TB}\CD@tG
\CD@tE\CD@UB\CD@TB\def\CD@U#1#2#3#4#5{\let\CD@oJ#1\let\CD@iD#2\let\CD@QG#3%
\let\CD@jD#4\let\CD@LE#5\CD@TB\ifx\CD@iD\CD@jD\CD@UB\fi}\def\CD@qD#1#2#3#4#5{%
\CD@U#1#2#3#4#5\CD@tD}\def\Vertical@Map{\CD@pB433{vertical map}\CD@cD\CD@LC
\CD@GB-9995 \let\CD@kD\@fillv\CD@nF{fill@dot}{df:.}\CD@qD}\def\break@args{%
\def\CD@tD{\CD@ZD}\CD@ZD\endgroup\aftergroup\CD@FE}\def\CD@MJ{\setbox1=\CD@oJ
\setbox5=\CD@LE\ifvoid3 \ifx\CD@QG\null\else\setbox3=\CD@QG\fi\fi\CD@@G2%
\CD@iD\CD@@G4\CD@jD}\def\CD@pF#1{\ifvoid1\else\CD@oF1#1\fi\ifvoid2\else\CD@oF
2#1\fi\ifvoid3\else\CD@oF3#1\fi\ifvoid4\else\CD@oF4#1\fi\ifvoid5\else\CD@oF5#%
1\fi} \def\CD@oF#1#2{\setbox#1\vbox{\offinterlineskip\box#1\dimen@\prevdepth
\advance\dimen@-#2\relax\setbox0\null\dp0\dimen@\ht0-\dimen@\box0}}\def\CD@@G
#1#2{\ifx#2\CD@kD\setbox#1=\box\voidb@x\else\setbox#1=#2\def#2{\xleaders\box#%
1}\fi}\CD@ZA\CD@BK{\string\HorizontalMap, \string\VerticalMap\space and
\string\DiagonalMap\CD@uG are obsolete - use \CD@lG\space to pre-define maps}%
\def\HorizontalMap#1#2#3#4#5{\CD@BK\CD@nB{old horizontal map}\CD@LC\CD@TJ\def
\CD@oJ{\CD@UH{#1}}\CD@SH\CD@iD{#2}\def\CD@QG{\CD@UH{#3}}\CD@SH\CD@jD{#4}\def
\CD@LE{\CD@UH{#5}}\CD@tD}\def\VerticalMap#1#2#3#4#5{\CD@BK\CD@pB433{vertical
map}\CD@cD\CD@LC\CD@GB-9995 \let\CD@kD\@fillv\def\CD@oJ{\CD@GG{#1}}\CD@VH
\CD@iD{#2}\def\CD@QG{\CD@GG{#3}}\CD@VH\CD@jD{#4}\def\CD@LE{\CD@GG{#5}}\CD@tD}%
\def\DiagonalMap#1#2#3#4#5{\CD@BK\CD@LC\def\CD@OG{4}\let\CD@kD\CD@qK\let
\CD@ZD\CD@YD\def\CD@WK{2}\def\CD@aK{2}\def\CD@ZK{1}\def\CD@bK{1}\def\CD@QG{%
\CD@vF{#3}}\ifPositiveGradient\let\mv\raise\def\CD@oJ{\CD@vF{#5}}\def\CD@iD{%
\CD@vF{#4}}\def\CD@jD{\CD@vF{#2}}\def\CD@LE{\CD@vF{#1}}\else\let\mv\lower\def
\CD@oJ{\CD@vF{#1}}\def\CD@iD{\CD@vF{#2}}\def\CD@jD{\CD@vF{#4}}\def\CD@LE{%
\CD@vF{#5}}\fi\CD@tD}\def\CD@aE{-}\def\CD@AD{\empty}\def\CD@SH{\CD@EG\CD@bE
\CD@aE\@fillh}\def\CD@VH{\CD@EG\CD@MK\CD@KK\@fillv}\def\CD@EG#1#2#3#4#5{\def
\CD@CH{#5}\ifx\CD@CH#2\let#4#3\else\let#4\null\ifx\CD@CH\empty\else\ifx\CD@CH
\CD@AD\else\let#4\CD@CH\fi\fi\fi}\def\CD@UH#1{\hbox{\mathsurround\z@
\offinterlineskip\def\CD@CH{#1}\ifx\CD@CH\empty\else\ifx\CD@CH\CD@AD\else
\CD@k\mkern-1.5mu{\CD@CH}\mkern-1.5mu\CD@ND\fi\fi}}\def\CD@yD#1#2{\setbox#1=%
\hbox\bgroup\setbox0=\hbox{\CD@k\labelstyle()\CD@ND}
\setbox1=\null\ht1\ht0\dp1\dp0\box1 \kern.1\CD@zC\CD@k\bgroup\labelstyle
\aftergroup\CD@LD\CD@xD}\def\CD@LD{\CD@ND\kern.1\CD@zC\egroup\CD@tD}\def
\CD@xD{\futurelet\CD@EH\CD@mJ}\def\CD@mJ{
\catcase\bgroup:\CD@v;\catcase\egroup:\missing@label;\catcase\space:\CD@TF;%
\tokcase[:\CD@XF;
\default:\CD@zJ;\endswitch}\def\CD@v{\let\CD@MD\CD@c\let\CD@CH}\def\CD@zJ#1{%
\let\CD@UF\egroup{\let\actually@braces@missing@around@macro@in@label\CD@ZH
\let\CD@MD\CD@xC\let\CD@UF\CD@VF#1%
\actually@braces@missing@around@macro@in@label}\CD@UF}\def
\actually@braces@missing@around@macro@in@label{\let\CD@CH=}\def\missing@label
{\egroup\CD@YA{missing label}\CD@PE}\def\CD@xC{\egroup\missing@label}\outer
\def\CD@ZH{}\def\CD@UF{}\def\CD@VF{\CD@wC\CD@UF}\def\CD@MD{}\def\CD@XF{\let
\CD@N\CD@xD\get@square@arg\CD@AE}\CD@rG\CD@PE{The text which has just been
read is not allowed within map labels.}\def\CD@c{\egroup\CD@YA{missing \CD@yC
\space inserted after label}\CD@PE}\def\upper@label{\CD@oD\CD@yD6}\def
\lower@label{\def\positional@{\CD@@A\break@args}\CD@yD7}\def\middle@label{%
\CD@yD3}\CD@tG\CD@yE\CD@pD\CD@oD\def\CD@iF{\ifPositiveGradient\CD@tJ
\expandafter\upper@label\else\expandafter\lower@label\fi}\def\CD@iI{%
\ifPositiveGradient\CD@tJ\expandafter\lower@label\else\expandafter
\upper@label\fi}\def\positional@{\CD@gB{labels as positional arguments are
obsolete}\CD@yE\CD@tJ\expandafter\upper@label\else\expandafter\lower@label\fi
-}\def\CD@tD{\futurelet\CD@EH\switch@arg}\def\eat@space{\afterassignment
\CD@tD\let\CD@EH= }\def\CD@TF{\afterassignment\CD@xD\let\CD@EH= }\def\CD@BC{%
\get@round@pair\CD@uD}\def\CD@uD#1#2{\def\CD@WK{#1}\def\CD@aK{#2}\CD@tD}\def
\optional@{\let\CD@N\CD@tD\get@square@arg\CD@AE}\def\CD@JJ.{\CD@sC\CD@tD}\def
\CD@sC{\let\CD@iD\fill@dot\let\CD@jD\fill@dot\def\CD@MI{\let\CD@iD\dfdot\let
\CD@jD\dfdot}}\def\CD@MI{}\def\CD@@E#1,{\CD@nH#1,\begingroup\ifx\@name\CD@RD
\CD@FF\aftergroup\CD@e\fi\aftergroup\CD@jC\else\expandafter\def\expandafter
\CD@RF\expandafter{\csname\@name\endcsname}\expandafter\CD@vD\CD@RF\CD@KD\ifx
\CD@RF\empty\aftergroup\CD@pC\expandafter\aftergroup\csname\CD@FB\@name
\endcsname\expandafter\aftergroup\csname\CD@FB @\@name\endcsname\else\gdef
\CD@GE{#1}\CD@gB{\string\relax\space inserted before `[\CD@GE'}\message{(I was
trying to read this as a \CD@tA\ option.)}\aftergroup\CD@H\fi\fi\endgroup}%
\def\CD@vD#1#2\CD@KD{\def\CD@RF{#2}}\def\CD@jC{\let\CD@CH\CD@N\let\CD@N\relax
\CD@CH}\def\CD@H#1],{
\CD@jC\relax\def\CD@RF{#1}\ifx\CD@RF\empty\def\CD@RF{[\CD@GE]}%
\else\def\CD@RF{[\CD@GE,#1]}
\fi\CD@RF}\def\CD@pC#1#2{\ifx#2\CD@qK\ifx#1\CD@qK\CD@gB{option `\@name'
undefined}\else#1\fi\else\CD@FF\expandafter#2\CD@GK\CD@PK\else\CD@QK\fi\fi
\CD@DH}\CD@tG\CD@FF\CD@QK\CD@PK\def\CD@nH#1,{\CD@FF\ifx\CD@GK\CD@qK\CD@e\else
\expandafter\CD@oH\CD@GK,#1,(,),(,)[]%
\fi\fi\CD@FF\else\CD@mH#1==,\fi}\def\CD@e{\CD@gB{option `\@name' needs (x,y)
value}\CD@PK\let\@name\empty}\def\CD@mH#1=#2=#3,{\def\@name{#1}\def\CD@GK{#2}%
\def\CD@RF{#3}\ifx\CD@RF\empty\let\CD@GK\CD@qK\fi}%
\def\CD@oH#1(#2,#3)#4,(#5,#6)#7[]{\def\CD@GK{{#2}{#3}}\def\CD@RF{#1#4#5#6}%
\ifx\CD@RF\empty\def\CD@RF{#7}\ifx\CD@RF\empty\CD@e\fi\else\CD@e\fi}\def
\CD@FB{cds@}\let\CD@N\relax\def\CD@zD#1{\ifx\CD@GK\CD@qK\CD@gB{option `\@name
' needs a value}\else#1\CD@GK\relax\fi}\def\CD@BE#1#2{\ifx\CD@GK\CD@qK#1#2%
\relax\else#1\CD@GK\relax\fi}\def\cds@@showpair#1#2{\message{x=#1,y=#2}}\def
\cds@@diagonalbase#1#2{\edef\CD@ZK{#1}\edef\CD@bK{#2}}\def\CD@DI#1{\def\CD@CH
{#1}\CD@nF{@x}{cdps@#1}\ifx\CD@CH\empty\CD@f\CD@CH{cannot be used}\else\ifx
\CD@CH\relax\CD@f\CD@CH{unknown}\else\let\CD@IK\@x\fi\fi}\def\CD@f#1#2{\CD@gB
{PostScript translator `#1' #2}}\def\CD@PH{}\def\CD@PJ{\CD@fA\edef\CD@PH{%
\noexpand\CD@KB{\@name\space ignored within maths}}}\def\diagramstyle{\CD@cJ
\let\CD@N\relax\CD@CF\CD@AE\CD@AE}\CD@tG\CD@sE
\CD@SB\CD@RB\CD@tG\CD@qE\CD@EB\CD@DB\CD@tG\CD@oE\CD@pA\CD@oA\CD@tG\CD@iE
\CD@HA\CD@GA\CD@HA\CD@tG\CD@jE\CD@JA\CD@IA\CD@tG\CD@kE\CD@LA\CD@KA\CD@tG
\CD@EF\CD@DK\CD@CK\CD@tG\CD@rE\CD@JB\CD@IB\CD@tG\CD@mE\CD@gA\CD@fA\CD@tG
\CD@nE\CD@kA\CD@jA\CD@tG\CD@AF\CD@iG\CD@hG\CD@RC{cds@ }{}\CD@RC{cds@}{}\CD@RC
{cds@1em}{\CellSize1\CD@zC}\CD@RC{cds@1.5em}{\CellSize1.5\CD@zC}\CD@RC{cds@2%
em}{\CellSize2\CD@zC}\CD@RC{cds@2.5em}{\CellSize2.5\CD@zC}\CD@RC{cds@3em}{%
\CellSize3\CD@zC}\CD@RC{cds@3.5em}{\CellSize3.5\CD@zC}\CD@RC{cds@4em}{%
\CellSize4\CD@zC}\CD@RC{cds@4.5em}{\CellSize4.5\CD@zC}\CD@RC{cds@5em}{%
\CellSize5\CD@zC}\CD@RC{cds@6em}{\CellSize6\CD@zC}\CD@RC{cds@7em}{\CellSize7%
\CD@zC}\CD@RC{cds@8em}{\CellSize8\CD@zC}\def\cds@abut{\MapsAbut\dimen1\z@
\dimen5\z@}\def\cds@alignlabels{\CD@IA\CD@KA}\def\cds@amstex{\ifincommdiag
\CD@O\else\def\CD{\diagram[amstex]}
\fi\CD@T\catcode`\@\active}\def\cds@b{\let\CD@dB\CD@bB}\def\cds@balance{\let
\CD@hA\CD@AA}\let\cds@bottom\cds@b\def\cds@center{\cds@vcentre\cds@nobalance}%
\let\cds@centre\cds@center\def\cds@centerdisplay{\CD@HA\CD@PJ\cds@balance}%
\let\cds@centredisplay\cds@centerdisplay\def\cds@crab{\CD@BE\CD@DC{.5%
\PileSpacing}}\CD@RC{cds@crab-}{\CD@DC-.5\PileSpacing}\CD@RC{cds@crab+}{%
\CD@DC.5\PileSpacing}\CD@RC{cds@crab++}{\CD@DC1.5\PileSpacing}\CD@RC{cds@crab%
--}{\CD@DC-1.5\PileSpacing}\def\cds@defaultsize{\CD@BE{\let\CD@QC}{3em}\CD@NJ
}\def\cds@displayoneliner{\CD@DB}\let\cds@dotted\CD@sC\def\cds@dpi{\CD@RJ{1%
truein}}\def\cds@dpm{\CD@RJ{100truecm}}\let\CD@XA\CD@qK\def\cds@eqno{\let
\CD@XA\CD@GK\let\CD@EJ\empty}\def\cds@fixed{\CD@qA}\CD@tG\CD@fE\CD@J\CD@I\def
\cds@flushleft{\CD@I\CD@GA\CD@PJ\cds@nobalance\CD@BE\CD@nA\CD@nA}\def\cds@gap
{\CD@AJ\setbox3=\null\ht3=\CD@tI\dp3=\CD@sI\CD@BE{\wd3=}\MapShortFall} \def
\cds@grid{\ifx\CD@GK\CD@qK\let\h@grid\relax\let\v@grid\relax\else\CD@nF{%
h@grid}{cdgh@\CD@GK}\CD@nF{v@grid}{cdgv@\CD@GK}\ifx\h@grid\relax\CD@gB{%
unknown grid `\CD@GK'}\else\CD@WB\fi\fi}\let\h@grid\relax\let\v@grid\relax
\def\cds@gridx{\ifx\CD@GK\CD@qK\else\cds@grid\fi\let\CD@CH\h@grid\let\h@grid
\v@grid\let\v@grid\CD@CH}\def\cds@h{\CD@zD\DiagramCellHeight}\def\cds@hcenter
{\let\CD@hA\CD@aA}\let\cds@hcentre\cds@hcenter\def\cds@heads{\CD@BE{\let
\CD@sJ}\CD@sJ\CD@@J\CD@vE\else\ifx\CD@sJ\CD@eF\else\CD@MC\fi\fi}\let
\cds@height\cds@h\let\cds@hmiddle\cds@balance\def\cds@htriangleheight{\CD@BE
\DiagramCellHeight\DiagramCellHeight\DiagramCellWidth1.73205%
\DiagramCellHeight}\def\cds@htrianglewidth{\CD@BE\DiagramCellWidth
\DiagramCellWidth\DiagramCellHeight.57735\DiagramCellWidth}\CD@tG\CD@zE\CD@eE
\CD@dE\CD@eE\def\cds@hug{\CD@eE} \def\cds@inline{\CD@gA\let\CD@PH\empty}\def
\cds@inlineoneliner{\CD@EB}\CD@RC{cds@l>}{\CD@zD{\let\CD@RG}\dimen2=\CD@RG}%
\def\cds@labelstyle{\CD@zD{\let\labelstyle}}\def\cds@landscape{\CD@kA}\def
\cds@large{\CellSize5\CD@zC}\let\CD@EJ\empty\def\CD@FJ{\refstepcounter{%
equation}\def\CD@XA{\hbox{\@eqnnum}}}\def\cds@LaTeXeqno{\let\CD@EJ\CD@FJ}\def
\cds@lefteqno{\CD@pA}\def\cds@leftflush{\cds@flushleft\CD@J}\def
\cds@leftshortfall{\CD@zD{\dimen1 }}\def\cds@lowershortfall{%
\ifPositiveGradient\cds@leftshortfall\else\cds@rightshortfall\fi}\def
\cds@loose{\CD@VB}\def\cds@midhshaft{\CD@JA}\def\cds@midshaft{\CD@JA}\def
\cds@midvshaft{\CD@LA}\def\cds@moreoptions{\CD@@A}\let\cds@nobalance
\cds@hcenter\def\cds@nohcheck{\CD@HH}\def\cds@nohug{\CD@dE} \def
\cds@nooptions{\def\CD@aC{\CD@WD}}\let\cds@noorigin\cds@nobalance\def
\cds@nopixel{\CD@@I4\CD@XH\CD@cJ}\def\cds@UO{\CD@oK\global\let\CD@n\empty}%
\def\cds@UglyObsolete{\cds@UO\let\cds@PS\empty}\def\CD@rK#1{\CD@gB{option `#1%
' renamed as `UglyObsolete'}}\def\cds@noPostScript{\CD@rK{noPostScript}}\def
\cds@noPS{\CD@rK{noPostScript}}\def\cds@notextflow{\CD@RB}\def\cds@noTPIC{%
\CD@CK}\def\cds@objectstyle{\CD@zD{\let\objectstyle}}\def\cds@origin{\let
\CD@hA\CD@iB}\def\cds@p{\CD@zD\PileSpacing}\let\cds@pilespacing\cds@p\def
\cds@pixelsize{\CD@zD\CD@@I\CD@gI}\def\cds@portrait{\CD@jA}\def
\cds@PostScript{\CD@nK\global\let\CD@n\empty\CD@BE\CD@DI\empty}\def\cds@PS{%
\CD@nK\global\let\CD@n\empty}\CD@GF\CD@n{\typeout{\CD@tA: try the PostScript
option for better results}}\def\cds@repositionpullbacks{\let\make@pbk\CD@fH
\let\CD@qH\CD@pH}\def\cds@righteqno{\CD@oA}\def\cds@rightshortfall{\CD@zD{%
\dimen5 }}\def\cds@ruleaxis{\CD@zD{\let\axisheight}}\def\cds@cmex{\let\CD@GG
\CD@sB\let\CD@QJ\CD@CJ}\def\cds@s{\cds@height\DiagramCellWidth
\DiagramCellHeight}\def\cds@scriptlabels{\let\labelstyle\scriptstyle}\def
\cds@shortfall{\CD@zD\MapShortFall\dimen1\MapShortFall\dimen5\MapShortFall}%
\def\cds@showfirstpass{\CD@BE{\let\CD@nD}\z@}\def\cds@silent{\def\CD@KB##1{}%
\def\CD@gB##1{}}\let\cds@size\cds@s\def\cds@small{\CellSize2\CD@zC}\def
\cds@snake{\CD@BE\CD@eJ\z@}\def\cds@t{\let\CD@dB\CD@fB}\def\cds@textflow{%
\CD@SB\CD@PJ}\def\cds@thick{\let\CD@rF\tenlnw\CD@LF\CD@NC\CD@BE\MapBreadth{2%
\CD@LF}\CD@@J}\def\cds@thin{\let\CD@rF\tenln\CD@BE\MapBreadth{\CD@NC}\CD@@J}%
\def\cds@tight{\CD@WB}\let\cds@top\cds@t\def\cds@TPIC{\CD@DK}\def
\cds@uppershortfall{\ifPositiveGradient\cds@rightshortfall\else
\cds@leftshortfall\fi}\def\cds@vcenter{\let\CD@dB\CD@cB}\let\cds@vcentre
\cds@vcenter\def\cds@vtriangleheight{\CD@BE\DiagramCellHeight
\DiagramCellHeight\DiagramCellWidth.577035\DiagramCellHeight}\def
\cds@vtrianglewidth{\CD@BE\DiagramCellWidth\DiagramCellWidth
\DiagramCellHeight1.73205\DiagramCellWidth}\def\cds@vmiddle{\let\CD@dB\CD@eB}%
\def\cds@w{\CD@zD\DiagramCellWidth}\let\cds@width\cds@w\def\diagram{\relax
\protect\CD@bC}\def\enddiagram{\protect\CD@SG}\def\CD@bC{\CD@g\CD@uI
\incommdiagtrue\edef\CD@wI{\the\CD@NB}\global\CD@NB\z@\boxmaxdepth\maxdimen
\everycr{}\CD@sK\everymath{}\everyhbox{}\ifx\pdfsyncstop\CD@qK\else
\pdfsyncstop\fi\CD@aC}\def\CD@aC{\CD@y\let\CD@N\CD@ZC\CD@CF\CD@AE\CD@WD}\def
\CD@ZC{\CD@gE\expandafter\CD@aC\else\expandafter\CD@WD\fi}\def\CD@WD{\let
\CD@EH\relax\CD@nE\CD@vE\else\CD@hK\else\CD@KB{landscape ignored without
PostScript}\CD@jA\fi\fi\fi\CD@EJ\setbox2=\vbox\bgroup\CD@JF\CD@VD}\def\CD@cH{%
\CD@nE\CD@fB\else\CD@dB\fi\CD@hA\nointerlineskip\setbox0=\null\ht0-\CD@pI\dp0%
\CD@pI\wd0\CD@kI\box0 \global\CD@QA\CD@kF\global\CD@yA\CD@XB\ifx\CD@NK\CD@qK
\global\CD@RA\CD@kF\else\global\CD@RA\CD@NK\fi\egroup\CD@zF\CD@nE\setbox2=%
\hbox to\dp2{\vrule height\wd2 depth\CD@QA width\z@\global\CD@QA\ht2\ht2\z@
\dp2\z@\wd2\z@\CD@hK\CD@tK{q 0 1 -1 0 0 0 cm}\else\global\CD@iG\CD@IK{0 1
bturn}\fi\box2\CD@gK\hss}\CD@DB\fi\ifnum\CD@yA=1 \else\CD@DB\fi\global
\@ignorefalse\CD@mE\leavevmode\fi\ifvmode\CD@TA\else\ifmmode\CD@PH\CD@GI\else
\CD@qE\CD@gA\fi\ifinner\CD@gA\fi\CD@mE\CD@GI\else\CD@sE\CD@QB\else\CD@TA\fi
\fi\fi\fi\CD@dD}\def\CD@dD{\global\CD@NB\CD@wI\relax\CD@xE\global\CD@ID\else
\aftergroup\CD@mC\fi\if@ignore\aftergroup\ignorespaces\fi\CD@wC\ignorespaces}%
\def\CD@fB{\advance\CD@pI\dimen1\relax}\def\CD@eB{\advance\CD@pI.5\dimen1%
\relax}\def\CD@bB{}\def\CD@cB{\CD@fB\advance\CD@pI\CD@YB\divide\CD@pI2
\advance\CD@pI-\axisheight\relax}\def\CD@aA{}\def\CD@iB{\CD@kF\z@}\def\CD@AA{%
\ifdim\dimen2>\CD@kF\CD@kF\dimen2 \else\dimen2\CD@kF\CD@kI\dimen0 \advance
\CD@kI\dimen2 \fi}\def\CD@QB{\skip0\z@\relax\loop\skip1\lastskip\ifdim\skip1>%
\z@\unskip\advance\skip0\skip1 \repeat\vadjust{\prevdepth\dp\strutbox\penalty
\predisplaypenalty\vskip\abovedisplayskip\CD@UA\penalty\postdisplaypenalty
\vskip\belowdisplayskip}\ifdim\skip0=\z@\else\hskip\skip0 \global\@ignoretrue
\fi}\def\CD@TA{\CD@LG\kern-\displayindent\CD@UA\CD@LG\global\@ignoretrue}\def
\CD@UA{\hbox to\hsize{\CD@fE\ifdim\CD@RA=\z@\else\advance\CD@QA-\CD@RA\setbox
2=\hbox{\kern\CD@RA\box2}\fi\fi\setbox1=\hbox{\ifx\CD@XA\CD@qK\else\CD@k
\CD@XA\CD@ND\fi}\CD@oE\CD@iE\else\advance\CD@QA\wd1 \fi\wd1\z@\box1 \fi\dimen
0\wd2 \advance\dimen0\wd1 \advance\dimen0-\hsize\ifdim\dimen0>-\CD@nA\CD@HA
\fi\advance\dimen0\CD@QA\ifdim\dimen0>\z@\CD@KB{wider than the page by \the
\dimen0 }\CD@HA\fi\CD@iE\hss\else\CD@V\CD@QA\CD@nA\fi\CD@GI\hss\kern-\wd1\box
1 }}\def\CD@GI{\CD@AF\CD@@F\else\CD@SC\global\CD@hG\fi\fi\kern\CD@QA\box2 }%
\CD@tG\CD@wE\CD@YC\CD@XC\def\CD@JF{\CD@cJ\ifdim\DiagramCellHeight=-\maxdimen
\DiagramCellHeight\CD@QC\fi\ifdim\DiagramCellWidth=-\maxdimen
\DiagramCellWidth\CD@QC\fi\global\CD@XC\CD@IF\let\CD@FE\empty\let\CD@z\CD@Q
\let\overprint\CD@eH\let\CD@s\CD@rJ\let\enddiagram\CD@ED\let\\\CD@cC\let\par
\CD@jH\let\CD@MD\empty\let\switch@arg\CD@PB\let\shift\CD@iA\baselineskip
\DiagramCellHeight\lineskip\z@\lineskiplimit\z@\mathsurround\z@\tabskip\z@
\CD@OB}\def\CD@VD{\penalty-123 \begingroup\CD@jA\aftergroup\CD@K\halign
\bgroup\global\advance\CD@NB1 \vadjust{\penalty1}\global\CD@FA\z@\CD@OB\CD@j#%
#\CD@DD\CD@Q\CD@Q\CD@OI\CD@j##\CD@DD\cr}\def\CD@ED{\CD@MD\CD@GD\crcr\egroup
\global\CD@JD\endgroup}\def\CD@j{\global\advance\CD@FA1 \futurelet\CD@EH\CD@i
}\def\CD@i{\ifx\CD@EH\CD@DD\CD@tJ\hskip1sp plus 1fil \relax\let\CD@DD\relax
\CD@vI\else\hfil\CD@k\objectstyle\let\CD@FE\CD@d\fi}\def\CD@DD{\CD@MD\relax
\CD@yI\CD@vI\global\CD@QA\CD@iA\penalty-9993 \CD@ND\hfil\null\kern-2\CD@QA
\null}\def\CD@cC{\cr}\def\across#1{\span\omit\mscount=#1 \global\advance
\CD@FA\mscount\global\advance\CD@FA\m@ne\CD@sF\ifnum\mscount>2 \CD@fJ\repeat
\ignorespaces}\def\CD@fJ{\relax\span\omit\advance\mscount\m@ne}\def\CD@qJ{%
\ifincommdiag\ifx\CD@iD\@fillh\ifx\CD@jD\@fillh\ifdim\dimen3>\z@\else\ifdim
\dimen2>93\CD@@I\ifdim\dimen2>18\p@\ifdim\CD@LF>\z@\count@\CD@bJ\advance
\count@\m@ne\ifnum\count@<\z@\count@20\let\CD@aJ\CD@uJ\fi\xdef\CD@bJ{\the
\count@}\fi\fi\fi\fi\fi\fi\fi}\def\CD@cG#1{\vrule\horizhtdp width#1\dimen@
\kern2\dimen@}\def\CD@uJ{\rlap{\dimen@\CD@@I\CD@V\dimen@{.182\p@}\CD@zH
\dimen@\advance\CD@tI\dimen@\CD@cG0\CD@cG0\CD@cG2\CD@cG6\CD@cG6\CD@cG2\CD@cG0%
\CD@cG0\CD@cG2\CD@cG6\CD@cG0\CD@cG0\CD@cG2\CD@cG2\CD@cG6\CD@cG0\CD@cG0\CD@cG2%
\CD@cG6\CD@cG2\CD@cG2\CD@cG0\CD@cG0}}\def\CD@bJ{10}\def\CD@aJ{}\def\CD@XD{%
\CD@gE\CD@TB\fi\CD@x\CD@WF\CD@HI}\def\CD@x{\CD@QJ\CD@DC\CD@MJ\ifdim\CD@DC=\z@
\else\CD@pF\CD@DC\fi\ifvoid3 \setbox3=\null\ht3\CD@tI\dp3\CD@sI\else\CD@V{\ht
3}\CD@tI\CD@V{\dp3}\CD@sI\fi\dimen3=.5\wd3 \ifdim\dimen3=\z@\CD@tE\else\dimen
3-\CD@XH\fi\else\CD@TB\fi\CD@V{\dimen2}{\wd7}\CD@V{\dimen2}{\wd6}\CD@qJ
\advance\dimen2-2\dimen3 \dimen4.5\dimen2 \dimen2\dimen4 \advance\dimen2%
\CD@eJ\advance\dimen4-\CD@eJ\advance\dimen2-\wd1 \advance\dimen4-\wd5 \ifvoid
2 \else\CD@V{\ht3}{\ht2}\CD@V{\dp3}{\dp2}\CD@V{\dimen2}{\wd2}\fi\ifvoid4 \else
\CD@V{\ht3}{\ht4}\CD@V{\dp3}{\dp4}\CD@V{\dimen4}{\wd4}\fi\advance\skip2\dimen
2 \advance\skip4\dimen4 \CD@tE\advance\skip2\skip4 \dimen0\dimen5 \advance
\dimen0\wd5 \skip3-\skip4 \advance\skip3-\dimen0 \let\CD@jD\empty\else\skip3%
\z@\relax\dimen0\z@\fi}\def\CD@WF{\offinterlineskip\lineskip.2\CD@zC\ifvoid6
\else\setbox3=\vbox{\hbox to2\dimen3{\hss\box6\hss}\box3}\fi\ifvoid7 \else
\setbox3=\vtop{\box3 \hbox to2\dimen3{\hss\box7\hss}}\fi}\def\CD@HI{\kern
\dimen1 \box1 \CD@aJ\CD@iD\hskip\skip2 \kern\dimen0 \ifincommdiag\CD@jE
\penalty1\fi\kern\dimen3 \penalty\CD@GB\hskip\skip3 \null\kern-\dimen3 \else
\hskip\skip3 \fi\box3 \CD@jD\hskip\skip4 \box5 \kern\dimen5}\def\CD@MF{\ifnum
\CD@LH>\CD@TC\CD@V{\dimen1}\objectheight\CD@V{\dimen5}\objectheight\else\CD@V
{\dimen1}\objectwidth\CD@V{\dimen5}\objectwidth\fi}\def\CD@Y{\begingroup
\ifdim\dimen7=\z@\kern\dimen8 \else\ifdim\dimen6=\z@\kern\dimen9 \else\dimen5%
\dimen6 \dimen6\dimen9 \CD@KJ\dimen4\dimen2 \CD@dG{\dimen4}\dimen6\dimen5
\dimen7\dimen8 \CD@KJ\CD@iC{\dimen2}\ifdim\dimen2<\dimen4 \kern\dimen2 \else
\kern\dimen4 \fi\fi\fi\endgroup}\def\CD@jJ{\CD@JI\setbox\z@\hbox{\lower
\axisheight\hbox to\dimen2{\CD@DF\ifPositiveGradient\dimen8\ht\CD@MH\dimen9%
\CD@mI\else\dimen8\dp3 \dimen9\dimen1 \fi\else\dimen8 \ifPositiveGradient
\objectheight\else\z@\fi\dimen9\objectwidth\fi\advance\dimen8
\ifPositiveGradient-\fi\axisheight\CD@Y\unhbox\z@\CD@DF\ifPositiveGradient
\dimen8\dp3 \dimen9\dimen0 \else\dimen8\ht\CD@MH\dimen9\CD@mF\fi\else\dimen8
\ifPositiveGradient\z@\else\objectheight\fi\dimen9\objectwidth\fi\advance
\dimen8 \ifPositiveGradient\else-\fi\axisheight\CD@Y}}}\def\CD@bD{\dimen6
\CD@aK\DiagramCellHeight\dimen7 \CD@WK\DiagramCellWidth\CD@jJ
\ifPositiveGradient\advance\dimen7-\CD@ZK\DiagramCellWidth\else\dimen7 \CD@ZK
\DiagramCellWidth\dimen6\z@\fi\advance\dimen6-\CD@bK\DiagramCellHeight\CD@mK
\setbox0=\rlap{\kern-\dimen7 \lower\dimen6\box\z@}\ht0\z@\dp0\z@\raise
\axisheight\box0 }\def\CD@mK{\setbox0\hbox{\ht\z@\z@\dp\z@\z@\wd\z@\z@\CD@hK
\expandafter\CD@tK{q \CD@eK\space\CD@lK\space\CD@kK\space\CD@eK\space0 0 cm}%
\else\global\CD@iG\CD@eD{\the\CD@TC\space\ifPositiveGradient\else-\fi\the
\CD@LH\space bturn}\fi\box\z@\CD@gK}}\def\CD@vB{\advance\CD@hF-\CD@mI\CD@wJ
\CD@hF\advance\CD@wJ\CD@hI\ifvoid\CD@sH\ifdim\CD@wJ<.1em\ifnum\CD@gD=\@m\else
\CD@aG h\CD@wJ<.1em:objects overprint:\CD@FA\CD@gD\fi\fi\else\ifhbox\CD@sH
\CD@SK\else\CD@TK\fi\advance\CD@wJ\CD@mI\CD@bH{-\CD@mI}{\box\CD@sH}{\CD@wJ}%
\z@\fi\CD@hF-\CD@mF\CD@gD\CD@FA\CD@hI\z@}\def\CD@SK{\setbox\CD@sH=\hbox{%
\unhbox\CD@sH\unskip\unpenalty}\setbox\CD@tH=\hbox{\unhbox\CD@tH\unskip
\unpenalty}\setbox\CD@sH=\hbox to\CD@wJ{\CD@OA\wd\CD@sH\unhbox\CD@sH\CD@PA
\lastkern\unkern\ifdim\CD@PA=\z@\CD@UB\advance\CD@OA-\wd\CD@tH\else\CD@TB\fi
\ifnum\lastpenalty=\z@\else\CD@JA\unpenalty\fi\kern\CD@PA\ifdim\CD@hF<\CD@OA
\CD@JA\fi\ifdim\CD@hI<\wd\CD@tH\CD@JA\fi\CD@jE\CD@hI\CD@wJ\advance\CD@hI-%
\CD@OA\advance\CD@hI\wd\CD@tH\ifdim\CD@hI<2\wd\CD@tH\CD@aG h\CD@hI<2\wd\CD@tH
:arrow too short:\CD@FA\CD@gD\fi\divide\CD@hI\tw@\CD@hF\CD@wJ\advance\CD@hF-%
\CD@hI\fi\CD@tE\kern-\CD@hI\fi\hbox to\CD@hI{\unhbox\CD@tH}\CD@HG}}\CD@tG
\ifinpile\inpiletrue\inpilefalse\inpilefalse\def\pile{\protect\CD@UJ\protect
\CD@uH}\def\CD@uH#1{\CD@l#1\CD@QD}\def\CD@UJ{\CD@nB{pile}\setbox0=\vtop
\bgroup\aftergroup\CD@lD\inpiletrue\let\CD@FE\empty\let\pile\CD@KF\let\CD@QD
\CD@PD\let\CD@GD\CD@FD\CD@yH\baselineskip.5\PileSpacing\lineskip.1\CD@zC
\relax\lineskiplimit\lineskip\mathsurround\z@\tabskip\z@\let\\\CD@wH}\def
\CD@l{\CD@DE\CD@YF\empty\halign\bgroup\hfil\CD@k\let\CD@FE\CD@d\let\\\CD@vH##%
\CD@MD\CD@ND\hfil\CD@Q\CD@R##\cr}\CD@rG\CD@NE{pile only allows one column.}%
\CD@rG\CD@UE{you left it out!}\def\CD@R{\CD@QD\CD@Q\relax\CD@YA{missing \CD@yC
\space inserted after \string\pile}\CD@NE}\def\CD@PD{\CD@MD\crcr\egroup
\egroup}\def\CD@GD{\CD@MD}\def\CD@FD{\CD@MD\relax\CD@QD\CD@YA{missing \CD@yC
\space inserted between \string\pile\space and \CD@HD}\CD@UE}\def\CD@QD{%
\CD@MD}\def\CD@lD{\vbox{\dimen1\dp0 \unvbox0 \setbox0=\lastbox\advance\dimen1%
\dp0 \nointerlineskip\box0 \nointerlineskip\setbox0=\null\dp0.5\dimen1\ht0-%
\dp0 \box0}\ifincommdiag\CD@tJ\penalty-9998 \fi\xdef\CD@YF{pile}}\def\CD@vH{%
\cr}\def\CD@wH{\noalign{\skip@\prevdepth\advance\skip@-\baselineskip
\prevdepth\skip@}}\def\CD@KF#1{#1}\def\CD@TK{\setbox\CD@sH=\vbox{\unvbox
\CD@sH\setbox1=\lastbox\setbox0=\box\voidb@x\CD@tF\setbox\CD@sH=\lastbox
\ifhbox\CD@sH\CD@rC\repeat\unvbox0 \global\CD@QA\CD@ZE}\CD@ZE\CD@QA}\def
\CD@rC{\CD@jE\setbox\CD@sH=\hbox{\unhbox\CD@sH\unskip\setbox\CD@sH=\lastbox
\unskip\unhbox\CD@sH}\ifdim\CD@wJ<\wd\CD@sH\CD@aG h\CD@wJ<\wd\CD@sH:arrow in
pile too short:\CD@FA\CD@gD\else\setbox\CD@sH=\hbox to\CD@wJ{\unhbox\CD@sH}%
\fi\else\CD@gJ\fi\setbox0=\vbox{\box\CD@sH\nointerlineskip\ifvoid0 \CD@tJ\box
1 \else\vskip\skip0 \unvbox0 \fi}\skip0=\lastskip\unskip}\def\CD@gJ{\penalty7
\noindent\unhbox\CD@sH\unskip\setbox\CD@sH=\lastbox\unskip\unhbox\CD@sH
\endgraf\setbox\CD@tH=\lastbox\unskip\setbox\CD@tH=\hbox{\CD@JG\unhbox\CD@tH
\unskip\unskip\unpenalty}\ifcase\prevgraf\cd@shouldnt P\or\ifdim\CD@wJ<\wd
\CD@tH\CD@aG h\CD@wJ<\wd\CD@sH:object in pile too wide:\CD@FA\CD@gD\setbox
\CD@sH=\hbox to\CD@wJ{\hss\unhbox\CD@tH\hss}\else\setbox\CD@sH=\hbox to\CD@wJ
{\hss\kern\CD@hF\unhbox\CD@tH\kern\CD@hI\hss}\fi\or\setbox\CD@sH=\lastbox
\unskip\CD@SK\else\cd@shouldnt Q\fi\unskip\unpenalty}\def\CD@cD{\CD@MJ\ifvoid
3 \setbox3=\null\ht3\axisheight\dp3-\ht3 \dimen3.5\CD@LF\else\dimen4\dp3
\dimen3.5\wd3 \setbox3=\CD@GG{\box3}\dp3\dimen4 \ifdim\ht3=-\dp3 \else\CD@TB
\fi\fi\dimen0\dimen3 \advance\dimen0-.5\CD@LF\setbox0\null\ht0\ht3\dp0\dp3\wd
0\wd3 \ifvoid6\else\setbox6\hbox{\unhbox6\kern\dimen0\kern2pt}\dimen0\wd6 \fi
\ifvoid7\else\setbox7\hbox{\kern2pt\kern\dimen3\unhbox7}\dimen3\wd7 \fi
\setbox3\hbox{\ifvoid6\else\kern-\dimen0\unhbox6\fi\unhbox3 \ifvoid7\else
\unhbox7\kern-\dimen3\fi}\ht3\ht0\dp3\dp0\wd3\wd0 \CD@tE\dimen4=\ht\CD@MH
\advance\dimen4\dp5 \advance\dimen4\dimen1 \let\CD@jD\empty\else\dimen4\ht3
\fi\setbox0\null\ht0\dimen4 \offinterlineskip\setbox8=\vbox spread2ex{\kern
\dimen5 \box1 \CD@iD\vfill\CD@tE\else\kern\CD@eJ\fi\box0}\ht8=\z@\setbox9=%
\vtop spread2ex{\kern-\ht3 \kern-\CD@eJ\box3 \CD@jD\vfill\box5 \kern\dimen1}%
\dp9=\z@\hskip\dimen0plus.0001fil \box9 \kern-\CD@LF\box8 \CD@kE\penalty2 \fi
\CD@tE\penalty1 \fi\kern\PileSpacing\kern-\PileSpacing\kern-.5\CD@LF\penalty
\CD@GB\null\kern\dimen3}\def\CD@cI{\ifhbox\CD@VA\CD@KB{clashing verticals}\ht
\CD@MH.5\dp\CD@VA\dp\CD@MH-\ht5 \CD@yB\ht\CD@MH\z@\dp\CD@MH\z@\fi\dimen1\dp
\CD@VA\CD@xA\prevgraf\unvbox\CD@VA\CD@wA\lastpenalty\unpenalty\setbox\CD@VA=%
\null\setbox\CD@lI=\hbox{\CD@JG\unhbox\CD@lI\unskip\unpenalty\dimen0\lastkern
\unkern\unkern\unkern\kern\dimen0 \CD@HG}\setbox\CD@lF=\hbox{\unhbox\CD@lF
\dimen0\lastkern\unkern\unkern\global\CD@QA\lastkern\unkern\kern\dimen0 }%
\CD@tF\ifnum\CD@xA>4 \CD@zI\repeat\unskip\unskip\advance\CD@mF.5\wd\CD@VA
\advance\CD@mF\wd\CD@lF\advance\CD@mI.5\wd\CD@VA\advance\CD@mI\wd\CD@lI\ifnum
\CD@FA=\CD@lA\CD@OA.5\wd\CD@VA\edef\CD@NK{\the\CD@OA}\fi\setbox\CD@VA=\hbox{%
\kern-\CD@mF\box\CD@lF\unhbox\CD@VA\box\CD@lI\kern-\CD@mI\penalty\CD@wA
\penalty\CD@NB}\ht\CD@VA\dimen1 \dp\CD@VA\z@\wd\CD@VA\CD@tB\CD@vB}\def\CD@zI{%
\ifdim\wd\CD@lF<\CD@QA\setbox\CD@lF=\hbox to\CD@QA{\CD@JG\unhbox\CD@lF}\fi
\advance\CD@xA\m@ne\setbox\CD@VA=\hbox{\box\CD@lF\unhbox\CD@VA}\unskip\setbox
\CD@lF=\lastbox\setbox\CD@lF=\hbox{\unhbox\CD@lF\unskip\unpenalty\dimen0%
\lastkern\unkern\unkern\global\CD@QA\lastkern\unkern\kern\dimen0 }}\def\CD@yB
{\dimen1\dp\CD@VA\ifhbox\CD@VA\CD@xB\else\CD@zB\fi\setbox\CD@VA=\vbox{%
\penalty\CD@NB}\dp\CD@VA-\dp\CD@MH\wd\CD@VA\CD@tB}\def\CD@zB{\unvbox\CD@VA
\CD@wA\lastpenalty\unpenalty\ifdim\dimen1<\ht\CD@MH\CD@aG v\dimen1<\ht\CD@MH:%
rows overprint:\CD@NB\CD@wA\fi}\def\CD@xB{\dimen0=\ht\CD@VA\setbox\CD@VA=%
\hbox\bgroup\advance\dimen1-\ht\CD@MH\unhbox\CD@VA\CD@xA\lastpenalty
\unpenalty\CD@wA\lastpenalty\unpenalty\global\CD@RA-\lastkern\unkern\setbox0=%
\lastbox\CD@tF\setbox\CD@VA=\hbox{\box0\unhbox\CD@VA}\setbox0=\lastbox\ifhbox
0 \CD@kJ\repeat\global\CD@SA-\lastkern\unkern\global\CD@QA\CD@JK\unhbox\CD@VA
\egroup\CD@JK\CD@QA\CD@bH{\CD@SA}{\box\CD@VA}{\CD@RA}{\dimen1}}\def\CD@kJ{%
\setbox0=\hbox to\wd0\bgroup\unhbox0 \unskip\unpenalty\dimen7\lastkern\unkern
\ifnum\lastpenalty=1 \unpenalty\CD@UB\else\CD@TB\fi\ifnum\lastpenalty=2
\unpenalty\dimen2.5\dimen0\advance\dimen2-.5\dimen1\advance\dimen2-%
\axisheight\else\dimen2\z@\fi\setbox0=\lastbox\dimen6\lastkern\unkern\setbox1%
=\lastbox\setbox0=\vbox{\unvbox0 \CD@tE\kern-\dimen1 \else\ifdim\dimen2=\z@
\else\kern\dimen2 \fi\fi}\ifdim\dimen0<\ht0 \CD@aG v\dimen0<\ht0:upper part of
vertical too short:{\CD@tE\CD@NB\else\CD@wA\fi}\CD@xA\else\setbox0=\vbox to%
\dimen0{\unvbox0}\fi\setbox1=\vtop{\unvbox1}\ifdim\dimen1<\dp1 \CD@aG v\dimen
1<\dp1:lower part of vertical too short:\CD@NB\CD@wA\else\setbox1=\vtop to%
\dimen1{\ifdim\dimen2=\z@\else\kern-\dimen2 \fi\unvbox1 }\fi\box1 \kern\dimen
6 \box0 \kern\dimen7 \CD@HG\global\CD@QA\CD@JK\egroup\CD@JK\CD@QA\relax}%
\countdef\CD@u=14 \newcount\CD@CA\newcount\CD@XB\newcount\CD@NB\let\CD@LB
\insc@unt\newcount\CD@FA\newcount\CD@lA\let\CD@mA\CD@XB\newcount\CD@MB\CD@tG
\CD@DF\CD@bI\CD@aI\CD@aI\def\CD@nD{-1}\def\CD@K{\CD@t\ifnum\CD@nD<\z@\else
\begingroup\scrollmode\showboxdepth\CD@nD\showboxbreadth\maxdimen\showlists
\endgroup\fi\CD@bI\CD@zF\CD@CA=\CD@u\advance\CD@CA1 \CD@XB=\CD@CA\ifnum\CD@NB
=1 \CD@JA\fi\advance\CD@XB\CD@NB\dimen1\z@\skip0\z@\count@=\insc@unt\advance
\count@\CD@u\divide\count@2 \ifnum\CD@XB>\count@\CD@KB{The diagram has too
many rows! It can't be reformatted.}\else\CD@NG\CD@WI\fi\CD@cH}\def\CD@NG{%
\CD@NB\CD@CA\CD@uF\ifnum\CD@NB<\CD@XB\setbox\CD@NB\box\voidb@x\advance\CD@NB1%
\relax\repeat\CD@NB\CD@CA\skip\z@\z@\CD@uF\CD@GB\lastpenalty\unpenalty\ifnum
\CD@GB>\z@\CD@KE\repeat\ifnum\CD@GB=-123 \CD@tJ\unpenalty\else\cd@shouldnt D%
\fi\ifx\v@grid\relax\else\CD@NB\CD@XB\advance\CD@NB\m@ne\expandafter\CD@VJ
\v@grid\fi\CD@MB\CD@mA\CD@tB\z@\CD@XG\ifx\h@grid\relax\else\expandafter\CD@LJ
\h@grid\fi\count@\CD@XB\advance\count@\m@ne\CD@YB\ht\count@}\def\CD@KE{%
\ifcase\CD@GB\or\CD@MG\else\CD@uA-\lastpenalty\unpenalty\CD@vA\lastpenalty
\unpenalty\setbox0=\lastbox\CD@WG\fi\CD@wD}\def\CD@wD{\skip1\lastskip\unskip
\advance\skip0\skip1 \ifdim\skip1=\z@\else\expandafter\CD@wD\fi}\def\CD@MG{%
\setbox0=\lastbox\CD@pI\dp0 \advance\CD@pI\skip\z@\skip\z@\z@\advance\CD@NF
\CD@pI\CD@uE\ifnum\CD@NB>\CD@CA\CD@NF\DiagramCellHeight\CD@pI\CD@NF\advance
\CD@pI-\CD@qI\fi\fi\CD@qI\ht0 \CD@NF\CD@qI\setbox\CD@NB\hbox{\unhbox\CD@NB
\unhbox0}\dp\CD@NB\CD@pI\ht\CD@NB\CD@qI\advance\CD@NB1 }\def\CD@WG{\ifnum
\CD@uA<\z@\advance\CD@uA\CD@XB\ifnum\CD@uA<\CD@CA\CD@UG\else\CD@OA\dp\CD@uA
\CD@PA\ht\CD@uA\setbox\CD@uA\hbox{\box\z@\penalty\CD@vA\penalty\CD@GB\unhbox
\CD@uA}\dp\CD@uA\CD@OA\ht\CD@uA\CD@PA\fi\else\CD@UG\fi}\def\CD@UG{\CD@KB{%
diagonal goes outside diagram (lost)}}\def\CD@fI{\advance\CD@uA\CD@XB\ifnum
\CD@uA<\CD@CA\CD@UG\else\ifnum\CD@uA=\CD@NB\CD@VG\else\ifnum\CD@uA>\CD@NB
\cd@shouldnt M\else\CD@OA\dp\CD@uA\CD@PA\ht\CD@uA\setbox\CD@uA\hbox{\box\z@
\penalty\CD@vA\penalty\CD@GB\unhbox\CD@uA}\dp\CD@uA\CD@OA\ht\CD@uA\CD@PA\fi
\fi\fi}\def\CD@WI{\CD@AJ\setbox\CD@PC=\hbox{\CD@k A\@super f\CD@lJ f\CD@ND}%
\CD@ZE\z@\CD@JK\z@\CD@kI\z@\CD@kF\z@\CD@NB=\CD@XB\CD@NF\z@\CD@uB\z@\CD@uF
\ifnum\CD@NB>\CD@CA\advance\CD@NB\m@ne\CD@qI\ht\CD@NB\CD@pI\dp\CD@NB\advance
\CD@NF\CD@qI\CD@rI\advance\CD@uB\CD@NF\CD@KC\CD@ZI\CD@w\ht\CD@NB\CD@qI\dp
\CD@NB\CD@pI\nointerlineskip\box\CD@NB\CD@NF\CD@pI\setbox\CD@NB\null\ht\CD@NB
\CD@uB\repeat\CD@wB\nointerlineskip\box\CD@NB\CD@gG\CD@ZE\DiagramCellWidth{%
width}\CD@gG\CD@JK\DiagramCellHeight{height}\CD@VA\CD@LB\advance\CD@VA-\CD@lA
\advance\CD@VA\m@ne\advance\CD@VA\CD@mA\dimen0\wd\CD@VA\CD@tI\axisheight
\dimen1\CD@uB\advance\dimen1-\CD@YB\dimen2\CD@kI\advance\dimen2-\dimen0
\advance\CD@XB-\CD@CA\advance\CD@LB-\CD@lA}\count@\year\multiply\count@12
\advance\count@\month\ifnum\count@>24398 \loop\iftrue\repeat\fi\def\CD@wB{\CD@qI-\CD@NF\CD@pI\CD@NF\setbox\CD@MH=\null\dp
\CD@MH\CD@NF\ht\CD@MH-\CD@NF\CD@mF\z@\CD@mI\z@\CD@lA\CD@LB\advance\CD@lA-%
\CD@MB\advance\CD@lA\CD@mA\CD@FA\CD@LB\CD@VA\CD@MB\CD@sF\ifnum\CD@FA>\CD@lA
\advance\CD@FA\m@ne\advance\CD@VA\m@ne\CD@tB\wd\CD@VA\setbox\CD@FA=\box
\voidb@x\CD@yB\repeat\CD@w\ht\CD@NB\CD@qI\dp\CD@NB\CD@pI}\def\CD@gG#1#2#3{%
\ifdim#1>.01\CD@zC\CD@PA#2\relax\advance\CD@PA#1\relax\advance\CD@PA.99\CD@zC
\count@\CD@PA\divide\count@\CD@zC\CD@KB{increase cell #3 to \the\count@ em}%
\fi}\def\CD@rI{\CD@FA=\CD@LB\penalty4 \noindent\unhbox\CD@NB\CD@sF\unskip
\setbox0=\lastbox\ifhbox0 \advance\CD@FA\m@ne\setbox\CD@FA\hbox to\wd0{\null
\penalty-9990\null\unhbox0}\repeat\CD@lA\CD@FA\advance\CD@FA\CD@MB\advance
\CD@FA-\CD@mA\ifnum\CD@FA<\CD@LB\count@\CD@FA\advance\count@\m@ne\dimen0=\wd
\count@\count@\CD@MB\advance\count@\m@ne\CD@tB\wd\count@\CD@sF\ifnum\CD@FA<%
\CD@LB\CD@DJ\CD@XG\dimen0\wd\CD@FA\advance\CD@FA1 \repeat\fi\CD@sF\CD@GB
\lastpenalty\unpenalty\ifnum\CD@GB>\z@\CD@vA\lastpenalty\unpenalty\CD@VG
\repeat\endgraf\unskip\ifnum\lastpenalty=4 \unpenalty\else\cd@shouldnt S\fi}%
\def\CD@VG{\advance\CD@vA\CD@lA\advance\CD@vA\m@ne\setbox0=\lastbox\ifnum
\CD@vA<\CD@LB\setbox\CD@vA\hbox{\box0\penalty\CD@GB\unhbox\CD@vA}\else\CD@UG
\fi}\def\CD@bG{}\CD@tG\CD@uE\CD@WB\CD@VB\def\CD@DJ{\advance\dimen0\wd\CD@FA
\divide\dimen0\tw@\CD@uE\dimen0\DiagramCellWidth\else\CD@V{\dimen0}%
\DiagramCellWidth\CD@pJ\fi\advance\CD@tB\dimen0 }\def\CD@XG{\setbox\CD@MB=%
\vbox{}\dp\CD@MB=\CD@uB\wd\CD@MB\CD@tB\advance\CD@MB1 }\def\CD@LJ#1,{\def
\CD@GK{#1}\ifx\CD@GK\CD@RD\else\advance\CD@tB\CD@GK\DiagramCellWidth\CD@XG
\expandafter\CD@LJ\fi}\def\CD@VJ#1,{\def\CD@GK{#1}\ifx\CD@GK\CD@RD\else\ifnum
\CD@NB>\CD@CA\CD@NF\CD@GK\DiagramCellHeight\advance\CD@NF-\dp\CD@NB\advance
\CD@NB\m@ne\ht\CD@NB\CD@NF\fi\expandafter\CD@VJ\fi}\def\CD@pJ{\CD@wE\CD@OA
\dimen0 \advance\CD@OA-\DiagramCellWidth\ifdim\CD@OA>2\MapShortFall\CD@KB{%
badly drawn diagonals (see manual)}\let\CD@pJ\empty\fi\else\let\CD@pJ\empty
\fi}\def\CD@KC{\CD@VA\CD@mA\CD@sF\ifnum\CD@VA<\CD@MB\dimen0\dp\CD@VA\advance
\dimen0\CD@NF\dp\CD@VA\dimen0 \advance\CD@VA1 \repeat}\def\CD@bH#1#2#3#4{%
\ifnum\CD@FA<\CD@LB\CD@OA=#1\relax\setbox\CD@FA=\hbox{\setbox0=#2\dimen7=#4%
\relax\dimen8=#3\relax\ifhbox\CD@FA\unhbox\CD@FA\advance\CD@OA-\lastkern
\unkern\fi\ifdim\CD@OA=\z@\else\kern-\CD@OA\fi\raise\dimen7\box0 \kern-\dimen
8 }\ifnum\CD@FA=\CD@lA\CD@V\CD@kF\CD@OA\fi\else\cd@shouldnt O\fi}\def\CD@w{%
\setbox\CD@NB=\hbox{\CD@FA\CD@lA\CD@VA\CD@mA\CD@PA\z@\relax\CD@sF\ifnum\CD@FA
<\CD@LB\CD@tB\wd\CD@VA\relax\CD@eI\advance\CD@FA1 \advance\CD@VA1 \repeat}%
\CD@V\CD@kI{\wd\CD@NB}\wd\CD@NB\z@}\def\CD@eI{\ifhbox\CD@FA\CD@OA\CD@tB\relax
\advance\CD@OA-\CD@PA\relax\ifdim\CD@OA=\z@\else\kern\CD@OA\fi\CD@PA\CD@tB
\advance\CD@PA\wd\CD@FA\relax\unhbox\CD@FA\advance\CD@PA-\lastkern\unkern\fi}%
\def\CD@ZI{\setbox\CD@sH=\box\voidb@x\CD@VA=\CD@MB\CD@FA\CD@LB\CD@VA\CD@mA
\advance\CD@VA\CD@FA\advance\CD@VA-\CD@lA\advance\CD@VA\m@ne\CD@tB\wd\CD@VA
\count@\CD@LB\advance\count@\m@ne\CD@hF.5\wd\count@\advance\CD@hF\CD@tB\CD@A
\m@ne\CD@gD\@m\CD@sF\ifnum\CD@FA>\CD@lA\advance\CD@FA\m@ne\advance\CD@hF-%
\CD@tB\CD@PI\wd\CD@VA\CD@tB\advance\CD@hF\CD@tB\advance\CD@VA\m@ne\CD@tB\wd
\CD@VA\repeat\CD@mF\CD@kF\CD@mI-\CD@mF\CD@vB}\newcount\CD@GB\def\CD@s{}\def
\CD@t{\CD@vK\mathsurround\z@\hsize\z@\rightskip\z@ plus1fil minus\maxdimen
\parfillskip\z@\linepenalty9000 \looseness0 \hfuzz\maxdimen\hbadness10000
\clubpenalty0 \widowpenalty0 \displaywidowpenalty0 \interlinepenalty0
\predisplaypenalty0 \postdisplaypenalty0 \interdisplaylinepenalty0
\interfootnotelinepenalty0 \floatingpenalty0 \brokenpenalty0 \everypar{}%
\leftskip\z@\parskip\z@\parindent\z@\pretolerance10000 \tolerance10000
\hyphenpenalty10000 \exhyphenpenalty10000 \binoppenalty10000 \relpenalty10000
\adjdemerits0 \doublehyphendemerits0 \finalhyphendemerits0 \baselineskip\z@
\CD@IA\def\parskip{\cd@shouldnt{PS}}\prevdepth\z@}\newbox\CD@KG\newbox\CD@IG
\def\CD@JG{\unhcopy\CD@KG}\def\CD@HG{\unhcopy\CD@IG}\def\CD@iJ{\hbox{}%
\penalty1\nointerlineskip}\def\CD@PI{\penalty5 \noindent\setbox\CD@MH=\null
\CD@mF\z@\CD@mI\z@\ifnum\CD@FA<\CD@LB\ht\CD@MH\ht\CD@FA\dp\CD@MH\dp\CD@FA
\unhbox\CD@FA\skip0=\lastskip\unskip\else\CD@OK\skip0=\z@\fi\endgraf\ifcase
\prevgraf\cd@shouldnt Y \or\cd@shouldnt Z \or\CD@RI\or\CD@XI\else\CD@QI\fi
\unskip\setbox0=\lastbox\unskip\unskip\unpenalty\noindent\unhbox0\setbox0%
\lastbox\unpenalty\unskip\unskip\unpenalty\setbox0\lastbox\CD@tF\CD@GB
\lastpenalty\unpenalty\ifnum\CD@GB>\z@\setbox\z@\lastbox\CD@lB\repeat\endgraf
\unskip\unskip\unpenalty}\def\CD@YJ{\CD@uA\CD@XB\advance\CD@uA-\CD@NB\CD@vA
\CD@FA\advance\CD@vA-\CD@lA\advance\CD@vA1 \expandafter\message{prevgraf=\the
\prevgraf at (\the\CD@uA,\the\CD@vA)}}\def\CD@XI{\CD@CE\setbox\CD@lI=\lastbox
\setbox\CD@lI=\hbox{\unhbox\CD@lI\unskip\unpenalty}\unskip\ifdim\ht\CD@lI>\ht
\CD@PC\setbox\CD@MH=\copy\CD@lI\else\ifdim\dp\CD@lI>\dp\CD@PC\setbox\CD@MH=%
\copy\CD@lI\else\CD@FG\CD@lI\fi\fi\advance\CD@mF.5\wd\CD@lI\advance\CD@mI.5%
\wd\CD@lI\setbox\CD@lI=\hbox{\unhbox\CD@lI\CD@HG}\CD@bH\CD@mF{\box\CD@lI}%
\CD@mI\z@\CD@yB\CD@vB}\def\CD@CE{\ifnum\CD@A>0 \advance\dimen0-\CD@tB\CD@iA-.%
5\dimen0 \CD@A-\CD@A\else\CD@A0 \CD@iA\z@\fi\setbox\CD@MH=\lastbox\setbox
\CD@MH=\hbox{\unhbox\CD@MH\unskip\unskip\unpenalty\setbox0=\lastbox\global
\CD@QA\lastkern\unkern}\advance\CD@iA-.5\CD@QA\unskip\setbox\CD@MH=\null
\CD@mI\CD@iA\CD@mF-\CD@iA}\def\CD@Z{\ht\CD@MH\CD@tI\dp\CD@MH\CD@sI}\def\CD@FG
#1{\setbox\CD@MH=\hbox{\CD@V{\ht\CD@MH}{\ht#1}\CD@V{\dp\CD@MH}{\dp#1}\CD@V{%
\wd\CD@MH}{\wd#1}\vrule height\ht\CD@MH depth\dp\CD@MH width\wd\CD@MH}}\def
\CD@QI{\CD@CE\CD@Z\setbox\CD@lI=\lastbox\unskip\setbox\CD@lF=\lastbox\unskip
\setbox\CD@lF=\hbox{\unhbox\CD@lF\unskip\global\CD@yA\lastpenalty\unpenalty}%
\advance\CD@yA9999 \ifcase\CD@yA\CD@VI\or\CD@YI\or\CD@TI\or\CD@dI\or\CD@cI\or
\CD@SI\else\cd@shouldnt9\fi}\def\CD@VI{\CD@FG\CD@lI\CD@UI\setbox\CD@sH=\box
\CD@lF\setbox\CD@tH=\box\CD@lI}\def\CD@YI{\CD@FG\CD@lF\setbox\CD@lI\hbox{%
\penalty8 \unhbox\CD@lI\unskip\unpenalty\ifnum\lastpenalty=8 \else\CD@xH\fi}%
\CD@UI\setbox\CD@lF=\hbox{\unhbox\CD@lF\unskip\unpenalty\global\setbox\CD@DA=%
\lastbox}\ifdim\wd\CD@lF=\z@\else\CD@xH\fi\setbox\CD@sH=\box\CD@DA}\def\CD@xH
{\CD@KB{extra material in \string\pile\space cell (lost)}}\def\CD@UI{\CD@yB
\ifvoid\CD@sH\else\CD@KB{Clashing horizontal arrows}\CD@mI.5\CD@hF\CD@mF-%
\CD@mI\CD@vB\CD@mI\z@\CD@mF\z@\fi\CD@hI\CD@hF\advance\CD@hI-\CD@mI\CD@hF-%
\CD@mF\CD@JC\CD@FA}\def\CD@RI{\setbox0\lastbox\unskip\CD@iA\z@\CD@Z\ifdim
\skip0>\z@\CD@tJ\CD@A0 \else\ifnum\CD@A<1 \CD@A0 \dimen0\CD@tB\fi\advance
\CD@A1 \fi}\def\VonH{\CD@MA46\VonH{.5\CD@LF}}\def\HonV{\CD@MA57\HonV{.5\CD@LF
}}\def\HmeetV{\CD@MA44\HmeetV{-\MapShortFall}}\def\CD@MA#1#2#3#4{\CD@pB34#1{%
\string#3}\CD@SD\CD@GB-999#2 \dimen0=#4\CD@tI\dimen0\advance\CD@tI\axisheight
\CD@sI\dimen0\advance\CD@sI-\axisheight\CD@CF\CD@HC\CD@ZD}\def\CD@HC#1{%
\setbox0=\hbox{\CD@k#1\CD@ND}\dimen0.5\wd0 \CD@tI\ht0 \CD@sI\dp0 \CD@ZD}\def
\CD@SD{\setbox0=\null\ht0=\CD@tI\dp0=\CD@sI\wd0=\dimen0 \copy0\penalty\CD@GB
\box0 }\def\CD@TI{\CD@GC\CD@yB}\def\CD@dI{\CD@GC\CD@vB}\def\CD@SI{\CD@GC
\CD@yB\CD@vB}\def\CD@GC{\setbox\CD@lI=\hbox{\unhbox\CD@lI}\setbox\CD@lF=\hbox
{\unhbox\CD@lF\global\setbox\CD@DA=\lastbox}\ht\CD@MH\ht\CD@DA\dp\CD@MH\dp
\CD@DA\advance\CD@mF\wd\CD@DA\advance\CD@mI\wd\CD@lI}\CD@tG
\ifPositiveGradient\CD@CI\CD@BI\CD@CI\CD@tG\ifClimbing\CD@rB\CD@qB\CD@rB
\newcount\DiagonalChoice\DiagonalChoice\m@ne\ifx\tenln\nullfont\CD@tJ\def
\CD@qF{\CD@KH\ifPositiveGradient/\else\CD@k\backslash\CD@ND\fi}\else\def
\CD@qF{\CD@rF\char\count@}\fi\let\CD@rF\tenln\def\Use@line@char#1{\hbox{#1%
\CD@rF\ifPositiveGradient\else\advance\count@64 \fi\char\count@}}\def\CD@cF{%
\Use@line@char{\count@\CD@TC\multiply\count@8\advance\count@-9\advance\count@
\CD@LH}}\def\CD@ZF{\Use@line@char{\ifcase\DiagonalChoice\CD@gF\or\CD@fF\or
\CD@fF\else\CD@gF\fi}}\def\CD@gF{\ifnum\CD@TC=\z@\count@'33 \else\count@
\CD@TC\multiply\count@\sixt@@n\advance\count@-9\advance\count@\CD@LH\advance
\count@\CD@LH\fi}\def\CD@fF{\count@'\ifcase\CD@LH55\or\ifcase\CD@TC66\or22\or
52\or61\or72\fi\or\ifcase\CD@TC66\or25\or22\or63\or52\fi\or\ifcase\CD@TC66\or
16\or36\or22\or76\fi\or\ifcase\CD@TC66\or27\or25\or67\or22\fi\fi\relax}\def
\CD@uC#1{\hbox{#1\setbox0=\Use@line@char{#1}\ifPositiveGradient\else\raise.3%
\ht0\fi\copy0 \kern-.7\wd0 \ifPositiveGradient\raise.3\ht0\fi\box0}}\def
\CD@jF#1{\hbox{\setbox0=#1\kern-.75\wd0 \vbox to.25\ht0{\ifPositiveGradient
\else\vss\fi\box0 \ifPositiveGradient\vss\fi}}}\def\CD@jI#1{\hbox{\setbox0=#1%
\dimen0=\wd0 \vbox to.25\ht0{\ifPositiveGradient\vss\fi\box0
\ifPositiveGradient\else\vss\fi}\kern-.75\dimen0 }}\CD@RC{+h:>}{%
\Use@line@char\CD@fF}\CD@RC{-h:>}{\Use@line@char\CD@gF}\CD@nF{+t:<}{-h:>}%
\CD@nF{-t:<}{+h:>}\CD@RC{+t:>}{\CD@jF{\Use@line@char\CD@fF}}\CD@RC{-t:>}{%
\CD@jI{\Use@line@char\CD@gF}}\CD@nF{+h:<}{-t:>}\CD@nF{-h:<}{+t:>}\CD@RC{+h:>>%
}{\CD@uC\CD@fF}\CD@RC{-h:>>}{\CD@uC\CD@gF}\CD@nF{+t:<<}{-h:>>}\CD@nF{-t:<<}{+%
h:>>}\CD@nF{+h:>->}{+h:>>}\CD@nF{-h:>->}{-h:>>}\CD@nF{+t:<-<}{-h:>>}\CD@nF{-t%
:<-<}{+h:>>}\CD@RC{+t:>>}{\CD@jF{\CD@uC\CD@fF}}\CD@RC{-t:>>}{\CD@jI{\CD@uC
\CD@gF}}\CD@nF{+h:<<}{-t:>>}\CD@nF{-h:<<}{+t:>>}\CD@nF{+t:>->}{+t:>>}\CD@nF{-%
t:>->}{-t:>>}\CD@nF{+h:<-<}{-t:>>}\CD@nF{-h:<-<}{+t:>>}\CD@RC{+f:-}{\CD@EF
\null\else\CD@cF\fi}\CD@nF{-f:-}{+f:-}\def\CD@tC#1#2{\vbox to#1{\vss\hbox to#%
2{\hss.\hss}\vss}}\def\hfdot{\CD@tC{2\axisheight}{.5em}}%
\def\vfdot{\CD@tC{1ex}\z@}
\def\CD@bF{\hbox{\dimen0=.3\CD@zC\dimen1\dimen0 \ifnum\CD@LH>\CD@TC\CD@iC{%
\dimen1}\else\CD@dG{\dimen0}\fi\CD@tC{\dimen0}{\dimen1}}}\newarrowfiller{.}%
\hfdot\hfdot\vfdot\vfdot\def\dfdot{\CD@bF\CD@CK}\CD@RC{+f:.}{\dfdot}\CD@RC{-f%
:.}{\dfdot}\def\CD@@K#1{\hbox\bgroup\def\CD@CH{#1\egroup}\afterassignment
\CD@CH
\count@='}\def\lnchar{\CD@@K\CD@qF}\def\CD@dF#1{\setbox#1=\hbox{\dimen5\dimen
#1 \setbox8=\box#1 \dimen1\wd8 \count@\dimen5 \divide\count@\dimen1 \ifnum
\count@=0 \box8 \ifdim\dimen5<.95\dimen1 \CD@gB{diagonal line too short}\fi
\else\dimen3=\dimen5 \advance\dimen3-\dimen1 \divide\dimen3\count@\dimen4%
\dimen3 \CD@dG{\dimen4}\ifPositiveGradient\multiply\dimen4\m@ne\fi\dimen6%
\dimen1 \advance\dimen6-\dimen3 \loop\raise\count@\dimen4\copy8 \ifnum\count@
>0 \kern-\dimen6 \advance\count@\m@ne\repeat\fi}}\def\CD@CG#1{\CD@EF\CD@xJ{#1%
}\else\CD@dF{#1}\fi}\def\CD@IH#1{}\newdimen\objectheight\objectheight1.8ex
\newdimen\objectwidth\objectwidth1em \def\CD@YD{\dimen6=\CD@aK
\DiagramCellHeight\dimen7=\CD@WK\DiagramCellWidth\CD@KJ\ifnum\CD@LH>0 \ifnum
\CD@TC>0 \CD@aF\else\aftergroup\CD@VC\fi\else\aftergroup\CD@UC\fi}\def\CD@VC{%
\CD@YA{diagonal map is nearly vertical}\CD@NA}\def\CD@UC{\CD@YA{diagonal map
is nearly horizontal}\CD@NA}\CD@rG\CD@NA{Use an orthogonal map instead}\def
\CD@aF{\CD@MJ\dimen3\dimen7\dimen7\dimen6\CD@iC{\dimen7}\advance\dimen3-%
\dimen7 \CD@MF\ifnum\CD@LH>\CD@TC\advance\dimen6-\dimen1\advance\dimen6-%
\dimen5 \CD@iC{\dimen1}\CD@iC{\dimen5}\else\dimen0\dimen1\advance\dimen0%
\dimen5\CD@dG{\dimen0}\advance\dimen6-\dimen0 \fi\dimen2.5\dimen7\advance
\dimen2-\dimen1 \dimen4.5\dimen7\advance\dimen4-\dimen5 \ifPositiveGradient
\dimen0\dimen5 \advance\dimen1-\CD@WK\DiagramCellWidth\advance\dimen1 \CD@ZK
\DiagramCellWidth\setbox6=\llap{\unhbox6\kern.1\ht2}\setbox7=\rlap{\kern.1\ht
2\unhbox7}\else\dimen0\dimen1 \advance\dimen1-\CD@ZK\DiagramCellWidth\setbox7%
=\llap{\unhbox7\kern.1\ht2}\setbox6=\rlap{\kern.1\ht2\unhbox6}\fi\setbox6=%
\vbox{\box6\kern.1\wd2}\setbox7=\vtop{\kern.1\wd2\box7}\CD@dG{\dimen0}%
\advance\dimen0-\axisheight\advance\dimen0-\CD@bK\DiagramCellHeight\dimen5-%
\dimen0 \advance\dimen0\dimen6 \advance\dimen1.5\dimen3 \ifdim\wd3>\z@\ifdim
\ht3>-\dp3\CD@TB\fi\fi\dimen3\dimen2 \dimen7\dimen2\advance\dimen7\dimen4
\ifvoid3 \else\CD@tE\else\dimen8\ht3\advance\dimen8-\axisheight\CD@iC{\dimen8%
}\CD@X{\dimen8}{.5\wd3}\dimen9\dp3\advance\dimen9\axisheight\CD@iC{\dimen9}%
\CD@X{\dimen9}{.5\wd3}\ifPositiveGradient\advance\dimen2-\dimen9\advance
\dimen4-\dimen8 \else\advance\dimen4-\dimen9\advance\dimen2-\dimen8 \fi\fi
\advance\dimen3-.5\wd3 \fi\dimen9=\CD@aK\DiagramCellHeight\advance\dimen9-2%
\DiagramCellHeight\CD@tE\advance\dimen2\dimen4 \CD@CG{2}\dimen2-\dimen0%
\advance\dimen2\dp2 \else\CD@CG{2}\CD@CG{4}\ifPositiveGradient\dimen2-\dimen0%
\advance\dimen2\dp2 \dimen4\dimen5\advance\dimen4-\ht4 \else\dimen4-\dimen0%
\advance\dimen4\dp4 \dimen2\dimen5\advance\dimen2-\ht2 \fi\fi\setbox0=\hbox to%
\z@{\kern\dimen1 \ifvoid1 \else\ifPositiveGradient\advance\dimen0-\dp1 \lower
\dimen0 \else\advance\dimen5-\ht1 \raise\dimen5 \fi\rlap{\unhbox1}\fi\raise
\dimen2\rlap{\unhbox2}\ifvoid3 \else\lower.5\dimen9\rlap{\kern\dimen3\unhbox3%
}\fi\kern.5\dimen7 \lower.5\dimen9\box6 \lower.5\dimen9\box7 \kern.5\dimen7
\CD@tE\else\raise\dimen4\llap{\unhbox4}\fi\ifvoid5 \else\ifPositiveGradient
\advance\dimen5-\ht5 \raise\dimen5 \else\advance\dimen0-\dp5 \lower\dimen0 \fi
\llap{\unhbox5}\fi\hss}\ht0=\axisheight\dp0=-\ht0\box0 }\def\NorthWest{\CD@BI
\CD@rB\DiagonalChoice0 }\def\NorthEast{\CD@CI\CD@rB\DiagonalChoice1 }\def
\SouthWest{\CD@CI\CD@qB\DiagonalChoice3 }\def\SouthEast{\CD@BI\CD@qB
\DiagonalChoice2 }\def\CD@aD{\vadjust{\CD@uA\CD@FA\advance\CD@uA
\ifPositiveGradient\else-\fi\CD@ZK\relax\CD@vA\CD@NB\advance\CD@vA-\CD@bK
\relax\hbox{\advance\CD@uA\ifPositiveGradient-\fi\CD@WK\advance\CD@vA\CD@aK
\hbox{\box6 \kern\CD@DC\kern\CD@eJ\penalty1 \box7 \box\z@}\penalty\CD@uA
\penalty\CD@vA}\penalty\CD@uA\penalty\CD@vA\penalty104}}\def\CD@eH#1{\relax
\vadjust{\hbox@maths{#1}\penalty\CD@FA\penalty\CD@NB\penalty\tw@}}\def\CD@lB{%
\ifcase\CD@GB\or\or\CD@bH{.5\wd0}{\box0}{.5\wd0}\z@\or\unhbox\z@\setbox\z@
\lastbox\CD@bH{.5\wd0}{\box0}{.5\wd0}\z@\unpenalty\unpenalty\setbox\z@
\lastbox\or\CD@TG\else\advance\CD@GB-100 \ifnum\CD@GB<\z@\cd@shouldnt B\fi
\setbox\z@\hbox{\kern\CD@mF\copy\CD@MH\kern\CD@mI\CD@uA\CD@XB\advance\CD@uA-%
\CD@NB\penalty\CD@uA\CD@uA\CD@FA\advance\CD@uA-\CD@lA\penalty\CD@uA\unhbox\z@
\global\CD@yA\lastpenalty\unpenalty\global\CD@zA\lastpenalty\unpenalty}\CD@uA
-\CD@yA\CD@vA\CD@zA\CD@fI\fi}\def\CD@TG{\unhbox\z@\setbox\z@\lastbox\CD@uA
\lastpenalty\unpenalty\advance\CD@uA\CD@mA\CD@vA\CD@XB\advance\CD@vA-%
\lastpenalty\unpenalty\dimen1\lastkern\unkern\setbox3\lastbox\dimen0\lastkern
\unkern\setbox0=\hbox to\z@{\unhbox0\setbox0\lastbox\setbox7\lastbox
\unpenalty\CD@eJ\lastkern\unkern\CD@DC\lastkern\unkern\setbox6\lastbox\dimen7%
\CD@tB\advance\dimen7-\wd\CD@uA\ifdim\dimen7<\z@\CD@CI\multiply\dimen7\m@ne
\let\mv\empty\else\CD@BI\def\mv{\raise\ht1}\kern-\dimen7 \fi\ifnum\CD@vA>%
\CD@NB\dimen6\CD@uB\advance\dimen6-\ht\CD@vA\else\dimen6\z@\fi\CD@jJ\CD@mK
\setbox1\null\ht1\dimen6\wd1\dimen7 \dimen7\dimen2 \dimen6\wd1 \CD@KJ\CD@uA
\CD@LH\CD@vA\CD@TC\dimen6\ht1 \CD@KJ\setbox2\null\divide\dimen2\tw@\advance
\dimen2\CD@eJ\CD@eG{\dimen2}\wd2\dimen2 \dimen0.5\dimen7 \advance\dimen0%
\ifPositiveGradient\else-\fi\CD@eJ\CD@dG{\dimen0}\advance\dimen0-\axisheight
\ht2\dimen0 \dimen0\CD@DC\CD@eG{\dimen0}\advance\dimen0\ht2\ht2\dimen0 \dimen
0\ifPositiveGradient-\fi\CD@DC\CD@dG{\dimen0}\advance\dimen0\wd2\wd2\dimen0
\setbox4\null\dimen0 .6\CD@zC\CD@eG{\dimen0}\ht4\dimen0 \dimen0 .2\CD@zC
\CD@dG{\dimen0}\wd4\dimen0 \dimen0\wd2 \ifvoid6\else\dimen1\ht4 \advance
\dimen1\ht2 \CD@CC6+-\raise\dimen1\rlap{\ifPositiveGradient\advance\dimen0-%
\wd6\advance\dimen0-\wd4 \else\advance\dimen0\wd4 \fi\kern\dimen0\box6}\fi
\dimen0\wd2 \ifvoid7\else\dimen1\ht4 \advance\dimen1-\ht2 \CD@CC7-+\lower
\dimen1\rlap{\ifPositiveGradient\advance\dimen0\wd4 \else\advance\dimen0-\wd7%
\advance\dimen0-\wd4 \fi\kern\dimen0\box7}\fi\mv\box0\hss}\ht0\z@\dp0\z@
\CD@bH{\z@}{\box\z@}{\z@}{\axisheight}}\def\CD@CC#1#2#3{\dimen4.5\wd#1 \ifdim
\dimen4>.25\dimen7\dimen4=.25\dimen7\fi\ifdim\dimen4>\CD@zC\dimen4.4\dimen4
\advance\dimen4.6\CD@zC\fi\CD@eG{\dimen4}\dimen5\axisheight\CD@dG{\dimen5}%
\advance\dimen4-\dimen5 \dimen5\dimen4\CD@eG{\dimen5}\advance\dimen0%
\ifPositiveGradient#2\else#3\fi\dimen5 \CD@dG{\dimen4}\advance\dimen1\dimen4 }
\def\CD@eD#1{\expandafter\CD@IK{#1}}\CD@ZA\CD@EK{output is PostScript
dependent}\def\CD@SC{\CD@IK{/bturn {gsave currentpoint currentpoint translate
4 2 roll neg exch atan rotate neg exch neg exch translate } def /eturn {%
currentpoint grestore moveto} def}}\def\CD@gK{\relax\CD@hK\CD@tK{Q}\else
\CD@IK{eturn}\fi} \def\CD@OJ#1{\count@#1\relax\multiply\count@7\advance
\count@16577\divide\count@33154 }\def\CD@fD#1{\expandafter\special{#1}} \def
\CD@xJ#1{\setbox#1=\hbox{\dimen0\dimen#1\CD@dG{\dimen0}\CD@OJ{\dimen0}\setbox
0=\null\ifPositiveGradient\count@-\count@\ht0\dimen0 \else\dp0\dimen0 \fi\box
0 \CD@uA\count@\CD@OJ\CD@LF\CD@fD{pn \the\count@}\CD@fD{pa 0 0}\CD@OJ{\dimen#%
1}\CD@fD{pa \the\count@\space\the\CD@uA}\CD@fD{fp}\kern\dimen#1}}\def\CD@JI{%
\CD@KJ\begingroup\ifdim\dimen7<\dimen6 \dimen2=\dimen6 \dimen6=\dimen7 \dimen
7=\dimen2 \count@\CD@LH\CD@LH\CD@TC\CD@TC\count@\else\dimen2=\dimen7 \fi
\ifdim\dimen6>.01\p@\CD@KI\global\CD@QA\dimen0 \else\global\CD@QA\dimen7 \fi
\endgroup\dimen2\CD@QA\CD@iK\CD@lK{\ifPositiveGradient\else-\fi\dimen6}\CD@iK
\CD@kK{\ifPositiveGradient-\fi\dimen6}\CD@iK\CD@eK{\dimen7}}\def\CD@KI{\CD@hJ
\ifdim\dimen7>1.73\dimen6 \divide\dimen2 4 \multiply\CD@TC2 \else\dimen2=0.%
353553\dimen2 \advance\CD@LH-\CD@TC\multiply\CD@TC4 \fi\dimen0=4\dimen2 \CD@ZG
4\CD@ZG{-2}\CD@ZG2\CD@ZG{-2.5}}\def\CD@AI{\begingroup\count@\dimen0 \dimen2 45%
pt \divide\count@\dimen2 \ifdim\dimen0<\z@\advance\count@\m@ne\fi\ifodd
\count@\advance\count@1\CD@@A\else\CD@y\fi\advance\dimen0-\count@\dimen2
\CD@gE\multiply\dimen0\m@ne\fi\ifnum\count@<0 \multiply\count@-7 \fi\dimen3%
\dimen1 \dimen6\dimen0 \dimen7 3754936sp \ifdim\dimen0<6\p@\def\CD@OG{4000}%
\fi\CD@KJ\dimen2\dimen3\CD@dG{\dimen2}\CD@hJ\multiply\CD@TC-6 \dimen0\dimen2
\CD@ZG1\CD@ZG{0.3}\dimen1\dimen0 \dimen2\dimen3 \dimen0\dimen3 \CD@ZG3\CD@ZG{%
1.5}\CD@ZG{0.3}\divide\count@2 \CD@gE\multiply\dimen1\m@ne\fi\ifodd\count@
\dimen2\dimen1\dimen1\dimen0\dimen0-\dimen2 \fi\divide\count@2 \ifodd\count@
\multiply\dimen0\m@ne\multiply\dimen1\m@ne\fi\global\CD@QA\dimen0\global
\CD@RA\dimen1\endgroup\dimen6\CD@QA\dimen7\CD@RA}\def\CD@OC{255}\let\CD@OG
\CD@OC\def\CD@KJ{\begingroup\ifdim\dimen7<\dimen6 \dimen9\dimen7\dimen7\dimen
6\dimen6\dimen9\CD@@A\else\CD@y\fi\dimen2\z@\dimen3\CD@XH\dimen4\CD@XH\dimen0%
\z@\dimen8=\CD@OG\CD@XH\CD@lC\global\CD@yA\dimen\CD@gE0\else3\fi\global\CD@zA
\dimen\CD@gE3\else0\fi\endgroup\CD@LH\CD@yA\CD@TC\CD@zA}\def\CD@lC{\count@
\dimen6 \divide\count@\dimen7 \advance\dimen6-\count@\dimen7 \dimen9\dimen4
\advance\dimen9\count@\dimen0 \ifdim\dimen9>\dimen8 \CD@@C\else\CD@AC\ifdim
\dimen6>\z@\dimen9\dimen6 \dimen6\dimen7 \dimen7\dimen9 \expandafter
\expandafter\expandafter\CD@lC\fi\fi}\def\CD@@C{\ifdim\dimen0=\z@\ifdim\dimen
9<2\dimen8 \dimen0\dimen8 \fi\else\advance\dimen8-\dimen4 \divide\dimen8%
\dimen0 \ifdim\count@\CD@XH<2\dimen8 \count@\dimen8 \dimen9\dimen4 \advance
\dimen9\count@\dimen0 \CD@AC\fi\fi}\def\CD@AC{\dimen4\dimen0 \dimen0\dimen9
\advance\dimen2\count@\dimen3 \dimen9\dimen2 \dimen2\dimen3 \dimen3\dimen9 }%
\def\CD@ZG#1{\CD@dG{\dimen2}\advance\dimen0 #1\dimen2 }\def\CD@dG#1{\divide#1%
\CD@TC\multiply#1\CD@LH}\def\CD@eG#1{\divide#1\CD@vA\multiply#1\CD@uA}\def
\CD@iC#1{\divide#1\CD@LH\multiply#1\CD@TC}\def\CD@hJ{\dimen6\CD@LH\CD@XH
\multiply\dimen6\CD@LH\dimen7\CD@TC\CD@XH\multiply\dimen7\CD@TC\CD@KJ}\def
\CD@iK#1#2{\begingroup\dimen@#2\relax\loop\ifdim\dimen2<.4\maxdimen\multiply
\dimen2\tw@\multiply\dimen@\tw@\repeat\divide\dimen2\@cclvi\divide\dimen@
\dimen2\relax\multiply\dimen@\@cclvi\expandafter\CD@jK\the\dimen@\endgroup
\let#1\CD@fK}{\catcode`p=12 \catcode`0=12 \catcode`.=12 \catcode`t=12 \gdef
\CD@jK#1pt{\gdef\CD@fK{#1}}}\ifx\errorcontextlines\CD@qK\CD@tJ\let\CD@GH
\relax\else\def\CD@GH{\errorcontextlines\m@ne}\fi\ifnum\inputlineno<0 \let
\CD@CD\empty\let\CD@W\empty\let\CD@mD\relax\let\CD@uI\relax\let\CD@vI\relax
\let\CD@zF\relax\else\def\CD@W{ at line \number\inputlineno}\def\CD@mD{ - first occurred%
}\def\CD@uI{\edef\CD@h{\the\inputlineno}\global\let\CD@jB\CD@h}\def\CD@h{9999%
}\def\CD@vI{\xdef\CD@jB{\the\inputlineno}}\def\CD@jB{\CD@h}\def\CD@zF{\ifnum
\CD@h<\inputlineno\edef\CD@CD{\space at lines \CD@h--\the\inputlineno}\else
\edef\CD@CD{\CD@W}\fi}\fi\let\CD@CD\empty\def\CD@YA#1#2{\CD@GH\errhelp=#2%
\expandafter\errmessage{\CD@tA: #1}}\def\CD@KB#1{\begingroup\expandafter
\message{! \CD@tA: #1\CD@CD}\ifnum\CD@XB>\CD@NB\ifnum\CD@CA>\CD@NB\else\ifnum
\CD@lA>\CD@FA\else\ifnum\CD@LB>\CD@FA\advance\CD@XB-\CD@NB\advance\CD@FA-%
\CD@lA\advance\CD@FA1\relax\expandafter\message{! (error detected at row \the
\CD@XB, column \the\CD@FA, but probably caused elsewhere)}\fi\fi\fi\fi
\endgroup}\def\CD@gB#1{{\expandafter\message{\CD@tA\space Warning: #1\CD@W}}}%
\def\CD@CB#1#2{\CD@gB{#1 \string#2 is obsolete\CD@mD}}\def\CD@AB#1{\CD@CB{%
Dimension}{#1}\CD@DE#1\CD@BB\CD@BB}\def\CD@BB{\CD@OA=}\def\CD@@B#1{\CD@CB{%
Count}{#1}\CD@DE#1\CD@OH\CD@OH}\def\CD@OH{\count@=}\def\HorizontalMapLength{%
\CD@AB\HorizontalMapLength}\def\VerticalMapHeight{\CD@AB\VerticalMapHeight}%
\def\VerticalMapDepth{\CD@AB\VerticalMapDepth}\def\VerticalMapExtraHeight{%
\CD@AB\VerticalMapExtraHeight}\def\VerticalMapExtraDepth{\CD@AB
\VerticalMapExtraDepth}\def\DiagonalLineSegments{\CD@@B\DiagonalLineSegments}%
\ifx\tenln\nullfont\CD@ZA\CD@KH{\CD@eF\space diagonal line and arrow font not
available}\else\let\CD@KH\relax\fi\def\CD@aG#1#2<#3:#4:#5#6{\begingroup\CD@PA
#3\relax\advance\CD@PA-#2\relax\ifdim.1em<\CD@PA\CD@uA#5\relax\CD@vA#6\relax
\ifnum\CD@uA<\CD@vA\count@\CD@vA\advance\count@-\CD@uA\CD@KB{#4 by \the\CD@PA
}\if#1v\let\CD@CH\CD@JK\edef\tmp{\the\CD@uA--\the\CD@vA,\the\CD@FA}\else
\advance\count@\count@\if#1l\advance\count@-\CD@A\else\if#1r\advance\count@
\CD@A\fi\fi\advance\CD@PA\CD@PA\let\CD@CH\CD@ZE\edef\tmp{\the\CD@NB,\the
\CD@uA--\the\CD@vA}\fi\divide\CD@PA\count@\ifdim\CD@CH<\CD@PA\global\CD@CH
\CD@PA\fi\fi\fi\endgroup}\CD@tG\CD@xE\CD@JD\CD@ID\CD@rG\CD@xI{See the message
above.}\CD@rG\CD@lH{Perhaps you've forgotten to end the diagram before
resuming the text, in\CD@uG which case some garbage may be added to the
diagram, but we should be ok now.\CD@uG Alternatively you've left a blank line
in the middle - TeX will now complain\CD@uG that the remaining \CD@S s are
misplaced - so please use comments for layout.}\CD@rG\CD@hD{You have already
closed too many brace pairs or environments; an \CD@HD\CD@uG command was (%
over)due.}\CD@rG\CD@hH{\CD@dC\space and \CD@HD\space commands must match.}%
\def\CD@jH{\ifnum\inputlineno=0 \else\expandafter\CD@iH\fi}\def\CD@iH{\CD@MD
\CD@GD\crcr\CD@YA{missing \CD@HD\space inserted before \CD@kH- type "h"}%
\CD@lH\enddiagram\CD@AG\CD@kH\par}\def\CD@AG#1{\edef\enddiagram{\noexpand
\CD@rD{#1\CD@W}}}\def\CD@rD#1{\CD@YA{\CD@HD\space(anticipated by #1) ignored}%
\CD@xI\let\enddiagram\CD@SG}\def\CD@SG{\CD@YA{misplaced \CD@HD\space ignored}%
\CD@hH}\def\CD@mC{\CD@YA{missing \CD@HD\space inserted.}\CD@hD\CD@AG{closing
group}}\ifx\ifx
\fi\fi

\catcode`\$=3 
\def\vboxtoz{\vbox to\z@}

\def\scriptaxis#1{\@scriptaxis{$\scriptstyle#1$}}
\def\ssaxis#1{\@ssaxis{$\scriptscriptstyle#1$}}
\def\@scriptaxis#1{\dimen0\axisheight\advance\dimen0-\ss@axisheight\raise
\dimen0\hbox{#1}}\def\@ssaxis#1{\dimen0\axisheight\advance\dimen0-%
\ss@axisheight\raise\dimen0\hbox{#1}}

\ifx\boldmath\CD@qK
\let\boldscriptaxis\scriptaxis
\def\boldscript#1{\hbox{$\scriptstyle#1$}}
\else\def\boldscriptaxis#1{\@scriptaxis{\boldmath$\scriptstyle#1$}}
\def\boldscript#1{\hbox{\boldmath$\scriptstyle#1$}}
\fi

\def\raisehook#1#2#3{\hbox{\setbox3=\hbox{#1$\scriptscriptstyle#3$}%
\dimen0\ss@axisheight
\dimen1\axisheight\advance\dimen1-\dimen0
\dimen2\ht3\advance\dimen2-\dimen0%
\advance\dimen2-0.021em\advance\dimen1 #2\dimen2%
\raise\dimen1\box3}}
\def\shifthook#1#2#3{\setbox1=\hbox{#1$\scriptscriptstyle#3$}\dimen0\wd1%
\divide\dimen0 12\CD@zH{\dimen0}
\dimen1\wd1\advance\dimen1-2\dimen0 \advance\dimen1-2\CD@oI\CD@zH{\dimen1}%
\kern#2\dimen1\box1}

\def\@cmex{\mathchar"03}

\def\make@pbk#1{\setbox\tw@\hbox to\z@{#1}\ht\tw@\z@\dp\tw@\z@\box\tw@}\def
\CD@fH#1{\overprint{\hbox to\z@{#1}}}\def\CD@qH{\kern0.11em}\def\CD@pH{\kern0%
.35em}

\def\dblvert{\def\CD@rH{\kern.5\PileSpacing}}\def\CD@rH{}

\def\SEpbk{\make@pbk{\CD@qH\CD@rH\vrule depth 2.87ex height -2.75ex width 0.%
95em \vrule height -0.66ex depth 2.87ex width 0.05em \hss}}

\def\SWpbk{\make@pbk{\hss\vrule height -0.66ex depth 2.87ex width 0.05em
\vrule depth 2.87ex height -2.75ex width 0.95em \CD@qH\CD@rH}}

\def\NEpbk{\make@pbk{\CD@qH\CD@rH\vrule depth -3.81ex height 4.00ex width 0.%
95em \vrule height 4.00ex depth -1.72ex width 0.05em \hss}}

\def\NWpbk{\make@pbk{\hss\vrule height 4.00ex depth -1.72ex width 0.05em
\vrule depth -3.81ex height 4.00ex width 0.95em \CD@qH\CD@rH}}

\def\puncture{{\setbox0\hbox{A}\vrule height.53\ht0 depth-.47\ht0 width.35\ht
0 \kern.12\ht0 \vrule height\ht0 depth-.65\ht0 width.06\ht0 \kern-.06\ht0
\vrule height.35\ht0 depth0pt width.06\ht0 \kern.12\ht0 \vrule height.53\ht0
depth-.47\ht0 width.35\ht0 }}

\def\NEclck{\overprint{\raise2.5ex\rlap{ \CD@rH$\scriptstyle\searrow$}}}
\def\NEanti{\overprint{\raise2.5ex\rlap{ \CD@rH$\scriptstyle\nwarrow$}}}
\def\NWclck{\overprint{\raise2.5ex\llap{$\scriptstyle\nearrow$ \CD@rH}}}
\def\NWanti{\overprint{\raise2.5ex\llap{$\scriptstyle\swarrow$ \CD@rH}}}
\def\SEclck{\overprint{\lower1ex\rlap{ \CD@rH$\scriptstyle\swarrow$}}}
\def\SEanti{\overprint{\lower1ex\rlap{ \CD@rH$\scriptstyle\nearrow$}}}
\def\SWclck{\overprint{\lower1ex\llap{$\scriptstyle\nwarrow$ \CD@rH}}}
\def\SWanti{\overprint{\lower1ex\llap{$\scriptstyle\searrow$ \CD@rH}}}

\def\rhvee{\mkern-10mu\greaterthan}
\def\lhvee{\lessthan\mkern-10mu}
\def\dhvee{\vboxtoz{\vss\hbox{$\vee$}\kern0pt}}
\def\uhvee{\vboxtoz{\hbox{$\wedge$}\vss}}
\newarrowhead{vee}\rhvee\lhvee\dhvee\uhvee

\def\dhlvee{\vboxtoz{\vss\hbox{$\scriptstyle\vee$}\kern0pt}}
\def\uhlvee{\vboxtoz{\hbox{$\scriptstyle\wedge$}\vss}}
\newarrowhead{littlevee}{\mkern1mu\scriptaxis\rhvee}{\scriptaxis\lhvee}%
\dhlvee\uhlvee\ifx\boldmath\CD@qK
\newarrowhead{boldlittlevee}{\mkern1mu\scriptaxis\rhvee}{\scriptaxis\lhvee}%
\dhlvee\uhlvee\else
\def\dhblvee{\vboxtoz{\vss\boldscript\vee\kern0pt}}
\def\uhblvee{\vboxtoz{\boldscript\wedge\vss}}
\newarrowhead{boldlittlevee}{\mkern1mu\boldscriptaxis\rhvee}{\boldscriptaxis
\lhvee}\dhblvee\uhblvee
\fi

\def\rhcvee{\mkern-10mu\succ}
\def\lhcvee{\prec\mkern-10mu}
\def\dhcvee{\vboxtoz{\vss\hbox{$\curlyvee$}\kern0pt}}
\def\uhcvee{\vboxtoz{\hbox{$\curlywedge$}\vss}}
\newarrowhead{curlyvee}\rhcvee\lhcvee\dhcvee\uhcvee

\def\rhvvee{\mkern-13mu\gg}
\def\lhvvee{\ll\mkern-13mu}
\def\dhvvee{\vboxtoz{\vss\hbox{$\vee$}\kern-.6ex\hbox{$\vee$}\kern0pt}}
\def\uhvvee{\vboxtoz{\hbox{$\wedge$}\kern-.6ex \hbox{$\wedge$}\vss}}
\newarrowhead{doublevee}\rhvvee\lhvvee\dhvvee\uhvvee

\def\rhtriangle{\triangleright\mkern1.2mu}
\def\lhtriangle{\triangleleft\mkern.8mu}
\def\uhtriangle{\vbox{\kern-.2ex \hbox{$\scriptscriptstyle\bigtriangleup$}%
\kern-.25ex}}
\def\dhtriangle{\vbox{\kern-.28ex \hbox{$\scriptscriptstyle\bigtriangledown$}%
\kern-.1ex}}
\def\dhblack{\vbox{\kern-.25ex\nointerlineskip\hbox{$\blacktriangledown$}}}%
\def\uhblack{\vbox{\kern-.25ex\nointerlineskip\hbox{$\blacktriangle$}}}%
\def\dhlblack{\vbox{\kern-.25ex\nointerlineskip\hbox{$\scriptstyle
\blacktriangledown$}}}
\def\uhlblack{\vbox{\kern-.25ex\nointerlineskip\hbox{$\scriptstyle
\blacktriangle$}}}
\newarrowhead{triangle}\rhtriangle\lhtriangle\dhtriangle\uhtriangle
\newarrowhead{blacktriangle}{\mkern-1mu\blacktriangleright\mkern.4mu}{%
\blacktriangleleft}\dhblack\uhblack\newarrowhead{littleblack}{\mkern-1mu%
\scriptaxis\blacktriangleright}{\scriptaxis\blacktriangleleft\mkern-2mu}%
\dhlblack\uhlblack

\def\rhla{\hbox{\setbox0=\lnchar55\dimen0=\wd0\kern-.6\dimen0\ht0\z@\raise
\axisheight\box0\kern.1\dimen0}}
\def\lhla{\hbox{\setbox0=\lnchar33\dimen0=\wd0\kern.05\dimen0\ht0\z@\raise
\axisheight\box0\kern-.5\dimen0}}
\def\dhla{\vboxtoz{\vss\rlap{\lnchar77}}}
\def\uhla{\vboxtoz{\setbox0=\lnchar66 \wd0\z@\kern-.15\ht0\box0\vss}}
\newarrowhead{LaTeX}\rhla\lhla\dhla\uhla

\def\lhlala{\lhla\kern.3em\lhla}
\def\rhlala{\rhla\kern.3em\rhla}
\def\uhlala{\hbox{\uhla\raise-.6ex\uhla}}
\def\dhlala{\hbox{\dhla\lower-.6ex\dhla}}
\newarrowhead{doubleLaTeX}\rhlala\lhlala\dhlala\uhlala

\def\hhO{\scriptaxis\bigcirc\mkern.4mu} \def\hho{{\circ}\mkern1.2mu}%
\newarrowhead{o}\hho\hho\circ\circ
\newarrowhead{O}\hhO\hhO{\scriptstyle\bigcirc}{\scriptstyle\bigcirc}

\def\rhtimes{\mkern-5mu{\times}\mkern-.8mu}\def\lhtimes{\mkern-.8mu{\times}%
\mkern-5mu}\def\uhtimes{\setbox0=\hbox{$\times$}\ht0\axisheight\dp0-\ht0%
\lower\ht0\box0 }\def\dhtimes{\setbox0=\hbox{$\times$}\ht0\axisheight\box0 }%
\newarrowhead{X}\rhtimes\lhtimes\dhtimes\uhtimes\newarrowhead+++++

\newarrowhead{Y}{\mkern-3mu\Yright}{\Yleft\mkern-3mu}\Ydown\Yup

\newarrowhead{->}\rightarrow\leftarrow\downarrow\uparrow

\newarrowhead{=>}\Rightarrow\Leftarrow{\@cmex7F}{\@cmex7E}

\newarrowhead{harpoon}\rightharpoonup\leftharpoonup\downharpoonleft
\upharpoonleft

\def\twoheaddownarrow{\rlap{$\downarrow$}\raise-.5ex\hbox{$\downarrow$}}
\def\twoheaduparrow{\rlap{$\uparrow$}\raise.5ex\hbox{$\uparrow$}}
\newarrowhead{->>}\twoheadrightarrow\twoheadleftarrow\twoheaddownarrow
\twoheaduparrow

\def\ltvee{\mkern-1mu{\lessthan}\mkern.4mu}
\newarrowtail{vee}\greaterthan\ltvee\vee\wedge

\newarrowtail{littlevee}{\scriptaxis\greaterthan}{\mkern-1mu\scriptaxis
\lessthan}{\scriptstyle\vee}{\scriptstyle\wedge}\ifx\boldmath\CD@qK
\newarrowtail{boldlittlevee}{\scriptaxis\greaterthan}{\mkern-1mu\scriptaxis
\lessthan}{\scriptstyle\vee}{\scriptstyle\wedge}\else\newarrowtail{%
boldlittlevee}{\boldscriptaxis\greaterthan}{\mkern-1mu\boldscriptaxis
\lessthan}{\boldscript\vee}{\boldscript\wedge}\fi

\newarrowtail{curlyvee}\succ{\mkern-1mu{\prec}\mkern.4mu}\curlyvee\curlywedge

\def\rttriangle{\mkern1.2mu\triangleright}
\newarrowtail{triangle}\rttriangle\lhtriangle\dhtriangle\uhtriangle
\newarrowtail{blacktriangle}\blacktriangleright{\mkern-1mu\blacktriangleleft
\mkern.4mu}\dhblack\uhblack\newarrowtail{littleblack}{\scriptaxis
\blacktriangleright\mkern-2mu}{\mkern-1mu\scriptaxis\blacktriangleleft}%
\dhlblack\uhlblack

\def\rtla{\hbox{\setbox0=\lnchar55\dimen0=\wd0\kern-.5\dimen0\ht0\z@\raise
\axisheight\box0\kern-.2\dimen0}}
\def\ltla{\hbox{\setbox0=\lnchar33\dimen0=\wd0\kern-.15\dimen0\ht0\z@\raise
\axisheight\box0\kern-.5\dimen0}}
\def\dtla{\vbox{\setbox0=\rlap{\lnchar77}\dimen0=\ht0\kern-.7\dimen0\box0%
\kern-.1\dimen0}}
\def\utla{\vbox{\setbox0=\rlap{\lnchar66}\dimen0=\ht0\kern-.1\dimen0\box0%
\kern-.6\dimen0}}
\newarrowtail{LaTeX}\rtla\ltla\dtla\utla

\def\rtvvee{\gg\mkern-3mu}
\def\ltvvee{\mkern-3mu\ll}
\def\dtvvee{\vbox{\hbox{$\vee$}\kern-.6ex \hbox{$\vee$}\vss}}
\def\utvvee{\vbox{\vss\hbox{$\wedge$}\kern-.6ex \hbox{$\wedge$}\kern\z@}}
\newarrowtail{doublevee}\rtvvee\ltvvee\dtvvee\utvvee

\def\ltlala{\ltla\kern.3em\ltla}
\def\rtlala{\rtla\kern.3em\rtla}
\def\utlala{\hbox{\utla\raise-.6ex\utla}}
\def\dtlala{\hbox{\dtla\lower-.6ex\dtla}}
\newarrowtail{doubleLaTeX}\rtlala\ltlala\dtlala\utlala

\def\utbar{\vrule height 0.093ex depth0pt width 0.4em}
\let\dtbar\utbar
\def\rtbar{\mkern1.5mu\vrule height 1.1ex depth.06ex width .04em\mkern1.5mu}%
\let\ltbar\rtbar
\newarrowtail{mapsto}\rtbar\ltbar\dtbar\utbar
\newarrowtail{|}\rtbar\ltbar\dtbar\utbar

\def\rthooka{\raisehook{}+\subset\mkern-1mu}
\def\lthooka{\mkern-1mu\raisehook{}+\supset}
\def\rthookb{\raisehook{}-\subset\mkern-2mu}
\def\lthookb{\mkern-1mu\raisehook{}-\supset}

\def\dthooka{\shifthook{}+\cap}
\def\dthookb{\shifthook{}-\cap}
\def\uthooka{\shifthook{}+\cup}
\def\uthookb{\shifthook{}-\cup}

\newarrowtail{hooka}\rthooka\lthooka\dthooka\uthooka\newarrowtail{hookb}%
\rthookb\lthookb\dthookb\uthookb

\ifx\boldmath\CD@qK\newarrowtail{boldhooka}\rthooka\lthooka\dthooka\uthooka
\newarrowtail{boldhookb}\rthookb\lthookb\dthookb\uthookb\newarrowtail{%
boldhook}\rthooka\lthooka\dthookb\uthooka\else\def\rtbhooka{\raisehook
\boldmath+\subset\mkern-1mu}
\def\ltbhooka{\mkern-1mu\raisehook\boldmath+\supset}
\def\rtbhookb{\raisehook\boldmath-\subset\mkern-2mu}
\def\ltbhookb{\mkern-1mu\raisehook\boldmath-\supset}
\def\dtbhooka{\shifthook\boldmath+\cap}
\def\dtbhookb{\shifthook\boldmath-\cap}
\def\utbhooka{\shifthook\boldmath+\cup}
\def\utbhookb{\shifthook\boldmath-\cup}
\newarrowtail{boldhooka}\rtbhooka\ltbhooka\dtbhooka\utbhooka\newarrowtail{%
boldhookb}\rtbhookb\ltbhookb\dtbhookb\utbhookb\newarrowtail{boldhook}%
\rtbhooka\ltbhooka\dtbhooka\utbhooka\fi

\def\dtsqhooka{\shifthook{}+\sqcap}
\def\ltsqhooka{\mkern-1mu\raisehook{}+\sqsupset}
\def\rtsqhooka{\raisehook{}+\sqsubset\mkern-1mu}
\def\utsqhooka{\shifthook{}+\sqcup}
\newarrowtail{sqhook}\rtsqhooka\ltsqhooka\dtsqhooka\utsqhooka

\newarrowtail{hook}\rthooka\lthookb\dthooka\uthooka\newarrowtail{C}\rthooka
\lthookb\dthooka\uthooka

\newarrowtail{o}\hho\hho\circ\circ
\newarrowtail{O}\hhO\hhO{\scriptstyle\bigcirc}{\scriptstyle\bigcirc}

\newarrowtail{X}\lhtimes\rhtimes\uhtimes\dhtimes\newarrowtail+++++

\newarrowtail{Y}\Yright\Yleft\Ydown\Yup

\newarrowtail{harpoon}\leftharpoondown\rightharpoondown\upharpoonright
\downharpoonright

\newarrowtail{<=}\Leftarrow\Rightarrow{\@cmex7E}{\@cmex7F}

\newarrowfiller{=}=={\@cmex77}{\@cmex77}
\def\vfthree{\mid\!\!\!\mid\!\!\!\mid}
\newarrowfiller{3}\equiv\equiv\vfthree\vfthree

\def\vfdashstrut{\vrule width0pt height1.3ex depth0.7ex}
\def\vfthedash{\vrule width\CD@LF height0.6ex depth 0pt}
\def\hfthedash{\CD@AJ\vrule\horizhtdp width 0.26em}
\def\hfdash{\mkern5.5mu\hfthedash\mkern5.5mu}
\def\vfdash{\vfdashstrut\vfthedash}
\newarrowfiller{dash}\hfdash\hfdash\vfdash\vfdash

\newarrowmiddle+++++

\iffalse
\newarrow{To}----{vee}
\newarrow{Arr}----{LaTeX}
\newarrow{Dotsto}....{vee}
\newarrow{Dotsarr}....{LaTeX}
\newarrow{Dashto}{}{dash}{}{dash}{vee}
\newarrow{Dasharr}{}{dash}{}{dash}{LaTeX}
\newarrow{Mapsto}{mapsto}---{vee}
\newarrow{Mapsarr}{mapsto}---{LaTeX}
\newarrow{IntoA}{hooka}---{vee}
\newarrow{IntoB}{hookb}---{vee}
\newarrow{Embed}{vee}---{vee}
\newarrow{Emarr}{LaTeX}---{LaTeX}
\newarrow{Onto}----{doublevee}
\newarrow{Dotsonarr}....{doubleLaTeX}
\newarrow{Dotsonto}....{doublevee}
\newarrow{Dotsonarr}....{doubleLaTeX}
\else
\newarrow{To}---->
\newarrow{Arr}---->
\newarrow{Dotsto}....>
\newarrow{Dotsarr}....>
\newarrow{Dashto}{}{dash}{}{dash}>
\newarrow{Dasharr}{}{dash}{}{dash}>
\newarrow{Mapsto}{mapsto}--->
\newarrow{Mapsarr}{mapsto}--->
\newarrow{IntoA}{hooka}--->
\newarrow{IntoB}{hookb}--->
\newarrow{Embed}>--->
\newarrow{Emarr}>--->
\newarrow{Onto}----{>>}
\newarrow{Dotsonarr}....{>>}
\newarrow{Dotsonto}....{>>}
\newarrow{Dotsonarr}....{>>}
\fi

\newarrow{Implies}===={=>}
\newarrow{Project}----{triangle}
\newarrow{Pto}----{harpoon}
\newarrow{Relto}{harpoon}---{harpoon}

\newarrow{Eq}=====
\newarrow{Line}-----
\newarrow{Dots}.....
\newarrow{Dashes}{}{dash}{}{dash}{}

\newarrow{SquareInto}{sqhook}--->

\newarrowhead{cmexbra}{\@cmex7B}{\@cmex7C}{\@cmex3B}{\@cmex38}
\newarrowtail{cmexbra}{\@cmex7A}{\@cmex7D}{\@cmex39}{\@cmex3A}
\newarrowmiddle{cmexbra}{\braceru\bracelu}{\bracerd\braceld}{\vcenter{%
\hbox@maths{\@cmex3D\mkern-2mu}}}
{\vcenter{\hbox@maths{\mkern2mu\@cmex3C}}}
\newarrow{@brace}{cmexbra}-{cmexbra}-{cmexbra}
\newarrow{@parenth}{cmexbra}---{cmexbra}
\def\rightBrace{\d@brace[thick,cmex]}
\def\leftBrace{\u@brace[thick,cmex]}
\def\upperBrace{\r@brace[thick,cmex]}
\def\lowerBrace{\l@brace[thick,cmex]}
\def\rightParenth{\d@parenth[thick,cmex]}
\def\leftParenth{\u@parenth[thick,cmex]}
\def\upperParenth{\r@parenth[thick,cmex]}
\def\lowerParenth{\l@parenth[thick,cmex]}

\let\hEq\rEq
\let\vEq\uEq

\def\labelstyle{
\ifincommdiag
\textstyle
\else
\scriptstyle
\fi}
\let\objectstyle\displaystyle

\newdiagramgrid{pentagon}{0.618034,0.618034,1,1,1,1,0.618034,0.618034}{1.%
17557,1.17557,1.902113,1.902113}

\newdiagramgrid{perspective}{0.75,0.75,1.1,1.1,0.9,0.9,0.95,0.95,0.75,0.75}{0%
.75,0.75,1.1,1.1,0.9,0.9}

\diagramstyle[
dpi=300,
vmiddle,nobalance,
loose,
thin,
pilespacing=10pt,%
shortfall=4pt,
]

\ifx\CD@qK\else\CD@PK\relax\CD@FF\CD@e\fi\fi

\CD@vE\CD@hK\else\fi\else\fi\cdrestoreat
\usepackage{amssymb}
\usepackage{mathrsfs}
\usepackage{cite}
\usepackage{tikz-cd}
\usepackage{MnSymbol}
\usepackage{a4wide}

\DeclareMathOperator\C{\mathbb C}
\DeclareMathOperator\R{\mathbb R}
\DeclareMathOperator\Z{\mathbb Z}

\newtheorem{theorem}{Theorem}[section]
\newtheorem{lemma}[theorem]{Lemma}
\newtheorem{cor}[theorem]{Corollary}
\newtheorem{prop}[theorem]{Proposition}
\theoremstyle{definition}

\theoremstyle{remark}
\newtheorem{remark}[theorem]{Remark}

\numberwithin{equation}{section}

\newcommand{\dontprint}[1]\relax

\newcommand{\SO}{\operatorname{SO}}
\newcommand{\rL}{\operatorname{L}}
\newcommand{\Ind}{\operatorname{Ind}}
\newcommand{\forg}{\operatorname{for}}

\newcommand{\fg}{{\frak g}}

\newcommand{\Ad}{\operatorname{Ad}}
\newcommand{\diag}{\operatorname{diag}}

\renewcommand{\div}{\operatorname{div}}

\newcommand{\Ga}{\Gamma}

\newcommand{\und}{\underline}
\newcommand{\Pic}{\operatorname{Pic}}
\newcommand{\hra}{\hookrightarrow}
\newcommand{\we}{\wedge}

\renewcommand{\P}{{\mathbb P}}
\newcommand{\A}{{\mathbb A}}

\newcommand{\wt}{\widetilde}
\newcommand{\ot}{\otimes}
\newcommand{\fu}{{\mathfrak u}}
\newcommand{\Hom}{\operatorname{Hom}}

\newcommand{\Om}{\Omega}

\newcommand{\NN}{{\mathcal N}}

\renewcommand{\SS}{{\mathcal S}}
\newcommand{\FF}{{\mathcal F}}

\newcommand{\GG}{{\mathcal G}}
\newcommand{\LL}{{\mathcal L}}
\newcommand{\MM}{{\mathcal M}}
\newcommand{\OO}{{\mathcal O}}
\newcommand{\PP}{{\mathcal P}}

\newcommand{\si}{\sigma}
\newcommand{\de}{\delta}
\newcommand{\sub}{\subset}
\newcommand{\Spec}{\operatorname{Spec}}
\newcommand{\Res}{\operatorname{Res}}
\newcommand{\PGL}{\operatorname{PGL}}
\newcommand{\ov}{\overline}

\newcommand{\om}{\omega}
\newcommand{\la}{\lambda}
\renewcommand{\a}{\alpha}
\renewcommand{\b}{\beta}

\newcommand{\tr}{\operatorname{tr}}

\newcommand{\GL}{\operatorname{GL}}
\newcommand{\tot}{\operatorname{tot}}

\renewcommand{\th}{\theta}
\newcommand{\ga}{\gamma}
\newcommand{\lan}{\langle}
\newcommand{\ran}{\rangle}

\newcommand{\SL}{{\operatorname{SL}}}
\newcommand{\End}{{\operatorname{End}}}
\newcommand{\Bun}{{\operatorname{Bun}}}

\newcommand{\fm}{{\mathfrak m}}

\newcommand{\Nm}{{\operatorname{Nm}}}
\newcommand{\Tr}{{\operatorname{Tr}}}
\newcommand{\eps}{\epsilon}
\newcommand{\Gal}{{\operatorname{Gal}}}
\newcommand{\vol}{{\operatorname{vol}}}
\newcommand{\bSt}{{\bf St}}
\newcommand{\St}{{\operatorname{St}}}

\title{Analog of theta-lifting for a curve over dual numbers over a finite field}
\author{David Kazhdan}
\author{Alexander Polishchuk}
\thanks{D.K. is partially supported by the ERC grant No 101142781.
A.P. is partially supported by the NSF grant DMS-2349388, by the Simons Travel grant MPS-TSM-00002745,
and within the framework of the HSE University Basic Research Program}

\address{Einstein Institute of Mathematics,
The Hebrew University of Jerusalem,
Jerusalem 91904, Israel}
\email{kazhdan@math.huji.ac.il}
\address{
    Department of Mathematics, 
    University of Oregon, 
    Eugene, OR 97403, USA; National Research University Higher School of Economics, Moscow, Russia
  }
  \email{apolish@uoregon.edu}

\begin{document}

\maketitle
\begin{abstract}
We continue the study of automorphic functions associated with a curve $C$ over the ring $k[\eps]/(\eps^2)$, where $k$ is a finite field, begun in \cite{BKP-aut}.
Namely, we study an example of theta-lifting in this framework and show that it can be understood in terms of the orbit decomposition of the space of automorphic functions
$\SS(\SL_2(F)\backslash \SL_2(\A_C))$ introduced in \cite{BKP-aut}.
We prove that all strongly cuspidal functions in $\SS(\SL_2(F)\backslash \SL_2(\A_C))$ can be constructed using theta-lifting for an appropriate double covering
$\wt{C}\to C$. 
\end{abstract}

\section{Introduction}

The classical theta-lifting is a correspondence relating automorphic functions for a pair of groups forming a Howe dual pair
(see \cite{Prasad}). In the case of the pair $(\SO_2,\SL_2)$ for the function field of a curve $C_0$ over a finite field $k$, this gives 
a way to produce cuspidal functions on $\SL_2$-bundles on $C_0$ from data on a double covering $\wt{C}_0\to C_0$.

In this paper we study an analog of this picture, with the curve $C_0$ replaced by a curve $C$ over the ring of dual numbers $A=k[\eps]/(\eps^2)$ (we denote
by $C_0$ the reduction of $C$ from $A$ to $k$).
The study of automorphic functions for such $C$ was initiated in \cite{BKP-aut}, and we refer to loc.\ cit. for details on this setup.


We set $F=k(C)$, the local ring of the general point of $C$. This is a square-zero extension of the function field $F_0=k(C_0)$. 
As in the case of a finite field, we have a naturally defined ring of adeles
$\A_C$ which is a square-zero extension of the usual adeles $\A_{C_0}$.
We denote by $\hat{\OO}\sub \A_C$ the subring of integer adeles. 

Let $G$ be a split connected reductive group over $\Z$, and let $\fg$ denote its Lie algebra.
As in the classical case, we are interested in the representation of
$G(\A_C)$ on the space $\SS(G(F)\backslash G(\A_C))$ (locally constant functions with compact support), 
and more specifically, in the subspace of $G(\hat{\OO})$-invariant vectors
and the action of the Hecke operators on it. As in the classical case, this subspace can be interpreted as functions on isomorphism
classes of $G$-bundles over $C$.

The algebra of Hecke operators in this case is not commutative, but one can define a natural commuting family of Hecke operators
associated with dominant coweights and relative Cartier divisors in $C\to \Spec(A)$, extending closed points in $C_0$ (see \cite{BKP-Hecke}).
Furthermore, if our curve $C$ over dual numbers is the reduction of a curve $C_O$ over the ring of integers $O$ in a local field $K$,
then our picture should be related to a conjectural picture with the Hecke operators acting on the Schwartz space of $G$-bundles over $C_K$, discussed
in \cite{BK}.

In the present paper we develop an analog of theta-lifting for automorphic functions over $C$. This involves constructing
a representation $\FF_{\LL}$ of $\SL_2(\A_C)$ associated with a double covering $\pi:\wt{C}\to C$ and a line bundle $\LL$ on $\wt{C}$ equipped
with an isomorphism $\Nm(\LL)\simeq \om_{C/A}$ (also depending on a choice of an additive character $\psi$ of the finite field $k$), and then
using a natural $\SL_2(F)$-invariant functional on $\FF_{\LL}$ to construct and study the {\it theta-lifting operator} 
$$\kappa_{\LL}:\FF_{\LL}\to \SS(\SL_2(F)\backslash \SL_2(\A_C)).$$

Our main result is the interpretation of this construction in terms of the orbit decomposition of $\SS(G(F)\backslash G(\A_C))$ for $G=\SL_2$,
defined in \cite{BKP-Hecke}. Namely, we consider orbits $\Om$ of the adjoint action of $G(F_0)$ on $\fg\ot \om_{C_0}(F_0)$. Using $\psi$
and residues of differentials, we view $\Om$ as a $G(F_0)$-orbit of characters of $\fg\ot \A_{C_0}/F_0$. Furthermore, we view $\fg\ot \A_{C_0}$ 
as a normal subgroup in $G(\A_C)$, and define the subspace $\SS(G(F)\backslash G(\A_C))_\Om\sub \SS(G(F)\backslash G(\A_C))$ 
associated with the corresponding
$G(F_0)$-orbit of characters of this subgroup (see \cite[Sec.\ 3]{BKP-aut} and Sec.\ \ref{recoll-sec} below). 

The $\SL_2(\A_C)$-action on $\FF_\LL$ is centralized by a natural action of the group $\A^*_{\wt{C},1}$ of ideles with the norm $1$. Restricting the latter action to
the subgroup of ideles of $\wt{C}$ trivial modulo $\eps$, we study the spaces of twisted coinvariants $\FF_{\LL,c_0}$, where $c_0$ runs through rational antiinvariant differentials on 
$\wt{C}_0$, the reduction of $\wt{C}$.
We show that for $c_0\neq 0$ the theta-lifting operator from $\FF_{\LL,c_0}$ lands in the subspace 
$\SS(\SL_2(F)\backslash \SL_2(\A_C))_{\Omega_{\LL,c_0}}$, where $\Omega_{\LL,c_0}$ is a regular elliptic orbit associated with $\LL$ and $c_0$ (see Theorem \ref{two-reps-two-funct-thm}).
Furthermore, we show that for every regular elliptic orbit $\Om$, all functions in $\SS(\SL_2(F)\backslash \SL_2(\A_C))_{\Omega}$ 
(which are known to be cuspidal) arise from theta-lifting for appropriate data $(\wt{C}\to C,\LL,c_0)$ (see Theorem \ref{main-thm1}).

Note that in this setup the space of cuspidal functions in $\SS(G(F)\backslash G(\A_C))$
contains a subspace of strongly cuspidal functions (see \cite{BKP-aut} and Sec.\ \ref{recoll-sec} below). 
Our main result implies that all strongly cuspidal functions arise from theta-lifting (see Remark \ref{reg-ell-rem}).

In the case when the double covering $\wt{C}_0\to C_0$ is associated with a quadratic differential on $C_0$ with simple zeros, we give a more precise information on the
theta-lifting of spherical (i.e., $\SL_2(\hat{\OO})$-invariant) vectors in $\FF_{\LL}$. Namely, in \cite[Sec.\ 3]{BKP-aut} we proved that spherical vectors in $\SS(G(F)\backslash G(\A_C))_\Om$ have a natural
interpretation as the space of functions on a certain Hitchin fiber (see also Sec.\ \ref{Higgs-sec} below). We show that a natural spherical vector in $\FF_{\LL,c_0}$ (corresponding to the characteristic function of integer points),
where $c_0$ is the canonical antiinvariant differential on $\wt{C}_0$, maps to the delta-function of a certain point in the Hitchin fiber (see Theorem \ref{main-thm2}).
In addition, we compute the action of some Hecke operators 
(a more general local computation is done in Sec.\ \ref{modt2-sph-sec}) on these spherical vectors, and in this
way produce explicit eigenfunctions for these Hecke operators in $\SS(G(F)\backslash G(\A_C))_\Om$ (see Theorem \ref{main-thm2}).


The paper is organized as follows. In Sec.\ \ref{nilp-setup-sec} we give some background on automorphic functions for a nilpotent extension $C$ of $C_0$ and for arbitrary $G$,
recalling relevant results from \cite{BKP-aut}. In Sections \ref{Higgs-sec} and \ref{Four-sec} we give an interpretation of the relation between 
$\SS(G(F)\backslash G(\A_C))_\Om$ and functions on Higgs $G$-bundles as a certain Fourier transform. 
In Sec.\ \ref{fin-field-sec} we review the classical picture for a curve over a finite field, including computation of the Hecke action in Sec.\ \ref{fin-field-Hecke-sec} and
the interpretation in terms of vector bundles in Sec.\ \ref{fin-field-vec-sec}. Then in Sec. \ref{nilp-theta-sec} we study the theta-lifting for a curve $C$ over $A=k[\eps]/(\eps^2)$.
First, we study the corresponding local picture in Sec.\ \ref{nilp-local-sec}. Then we compute spherical vectors and some Hecke operators in Sec.\ \ref{modt2-sph-sec}.
Finally, we study the global picture in Sec.\ \ref{nilp-global-sec} proving our main results.

\medskip

\noindent{\it Convention.}
When considering quadratic extensions (starting from Section \ref{fin-field-sec}), we always assume the characteristic of the finite field $k$ to be different from $2$.

\section{Cuspidal functions for nilpotent extensions}\label{nilp-setup-sec}

In this section we work with a nilpotent extension $C$ of a curve $C_0$ over a finite field $k$, so $\OO_C$ is a square zero extension of $\OO_{C_0}$,
such that the corresponding ideal $\NN\sub \OO_C$ is a line bundle on $C_0$.

\subsection{Recollections from \cite{BKP-aut}}\label{recoll-sec}

We fix a nontrivial additive character $\psi$ of the finite field $k$, and denote by
$$\psi_{C_0}:\om_{C_0}(\A_{C_0})\to U(1)$$
the induced additive character $\a\mapsto \psi(\sum_p \Res_p \a)$.

Let $G$ be a split reductive group over $\Z$. Let $N_{\A}\sub G(\A_C)$ (resp., $N_F\sub G(F)$) denote the kernel of the reduction homomorphism
to $G(\A_{C_0})$ (resp., to $G(F_0)$). We have a natural identification of groups $N_\A\simeq \fg\ot \NN(\A_{C_0})$ (resp., $N_F\simeq \fg\ot \NN(F_0)$).
With each $\eta\in \fg^\vee\ot (\NN^{-1}\om_{C_0})(F_0)$ we associate a character 
$$\psi_{\eta}=\psi_{C_0}(\lan\eta,?\ran)$$ 
of $N_{\A}$, trivial on $N_F$.
Let $\Om$ denote the $G(F_0)$-orbit of $\eta_0$. Then we have a well defined projector operator $\Pi_{\Om}$ on $\SS(G(F)\backslash G(\A_C))$, given by
$$\Pi_{\Om}(f)=\sum_{\eta\in\Om} \Pi_{\eta}(f), \ \ \Pi_{\eta}f(g)=\vol(N_\A/N_F)^{-1}\cdot \int_{N_\A/N_F}\psi_{\eta}(u)^{-1}f(ug)du.$$
We have a direct sum decomposition
$$\SS(G(F)\backslash G(\A_C))=\bigoplus_{\Om} \SS(G(F)\backslash G(\A_C))_{\Om},$$
where the summation is over $G(F_0)$-orbits on $\fg^\vee\ot (\NN^{-1}\om_{C_0})(F_0)$, and $\SS(G(F)\backslash G(\A_C))_{\Om}$ denote the image of $\Pi_{\Om}$
(see \cite[Eq.\ (3.1)]{BKP-aut}). 

Given an element $\eta\in \fg^\vee\ot (\NN^{-1}\om_{C_0})(F_0)$, we denote by $\bSt_{\eta}\sub G$ its centralizer
viewed as an algebraic subgroup defined over $F_0$. We denote by $\St_{\eta}\sub G(\A_C)$ the preimage of
$\bSt_{\eta}(\A_{C_0})$ under the projection $G(\A_C)\to G(\A_{C_0})$.

Furthermore, for each $\eta\in \fg^\vee\ot (\NN^{-1}\om_{C_0})(F_0)$, there is an identification of $G(\A_C)$-representations, 
$$\SS(G(F)\backslash G(\A_C))_{\Om_{\eta}}\simeq \wt{\SS}_\eta\sub \C_{lc}(N_F\backslash G(\A_C)),$$
where $\Om_\eta$ is the $G(F_0)$-orbit of $\eta$, and
$\wt{\SS}_\eta$ consists of functions $f$ such that $f(ug)=\psi_{\eta}(u)f(g)$ for $u\in N_\A$ and $f(\ga g)=f(g)$ for $\ga\in \St_\eta(F)$ (where 
$\St_\eta(F)=\St_\eta\cap G(F)$), and the support of $f$ modulo $N_\A\cdot \St_\eta(F)$ is compact.
The isomorphism
\begin{equation}\label{kappa-eta-eq}
\kappa_\eta:\wt{\SS}_{\eta}\rTo{\sim}\SS(G(F)\backslash G(\A_C))_{\Om_{\eta}} 
\end{equation}
is associated with the $G(F)$-invariant functional $\Theta_\eta$
on $\wt{\SS}_\eta$ given by 
\begin{equation}\label{Theta-eta-def}
\Theta_\eta(f)=\sum_{\ga\in \St_\eta(F)\backslash G(F)} f(\ga),
\end{equation}
so that $\kappa_\eta$ is given by
\begin{equation}\label{kappa-eta-def-eq}
\kappa_\eta(f)(g)=\Theta_\eta(g\cdot f)
\end{equation} 
(see \cite[Lem.\ 3.13]{BKP-aut}).

Recall that a function $f$ on $G(F)\backslash G(\A_C)$ is called {\it cuspidal} (resp., {\it strongly cuspidal}) if
$$\int_{u\in U_P(\A_C)/U_P(F)}f(ug)du=0 \ \ (\text{resp.,} \ \int_{u\in U_P(\A_C)\cap N_\A/U_P(F)\cap N_F)}f(ug)du=0$$
where $U_P$ is the unipotent radical of a standard parabolic subgroup in $G$.
The following result is a generalization of \cite[Prop.\ 6.5]{BKP-aut} (that deals with the case $G=\PGL_2$).

\begin{lemma}\label{reg-ell-str-cusp-lem}
Assume that $\eta$ is regular semisimple and elliptic (i.e., the torus $\bSt_{\eta}$ is anisotropic over $F_0$).
Then any $f\in \SS(G(F)\backslash G(\A_C))_{\Om_{\eta}}$ is strongly cuspidal.
\end{lemma}

\begin{proof}
This is proved analogously to the case $G=\PGL_2$ considered in \cite[Prop.\ 6.5]{BKP-aut}.
It is enough to check that if $\fu_P$ is the Lie algebra of the unipotent radical in a standard parabolic subgroup $P$ of $G$,
then for any $f\in \SS(G(F)\backslash G(\A_C))$, one has 
$$\int_{a\in \fu_P\ot \NN(\A_{C_0})/\fu_P\ot \NN(F_0)}\Pi_{\eta}f((1+a)g)=0$$
for any $\eta$ which is regular semisimple and elliptic.
We can rewrite the integral as
\begin{align*}
&\int_{X\in\fg\ot\NN(\A_{C_0})/\fg\ot \NN(F_0)}\int_{a\in \fu_P\ot\NN(\A_{C_0})/\fu_P\ot \NN(F_0)} \psi_\eta(-X)f((1+X+a)g)dadX=\\
&\int_{X\in\fg\ot \NN(\A_{C_0})/\fg\ot \NN(F_0)} \int_{a\in \fu_P\ot\NN(\A_{C_0})/\fu_P\ot\NN(F_0)}\psi_{\eta}(a)\psi_\eta(-X)f((1+X)g)dadX.
\end{align*}
Now we observe that $\int_{a\in \fu_P\ot\NN(\A_{C_0})/\fu_P\ot \NN(F_0)}\psi_{\eta}(a)da=0$ unless the restriction $\psi_{\eta}|_{\fu_P\ot\NN(\A_{C_0})}$ is trivial,
i.e., $\eta$ is orthogonal to $\fu_P$ with respect to the Killing form. But the latter orthogonal is exactly the Lie algebra of $P$, so
this cannot happen (otherwise $\eta$ would be centralized by a split torus).
\end{proof}

\begin{remark}\label{reg-ell-rem}
Conversely, the argument of \cite[Prop.\ 6.5]{BKP-aut} shows that if $\eta$ belongs to the Lie algebra of a standard parabolic subgroup then
the space $\SS(G(F)\backslash G(\A_C))_{\Om_{\eta}}$ does not contain any nonzero strongly cuspidal functions. In the case of $G=\SL_2$ 
(and characteristic $\neq 2$) this means that all strongly cuspidal functions come from regular elliptic orbits.
\end{remark}

\subsection{Functions on $\Bun_G(C)$ and Higgs bundles}\label{Higgs-sec}

Next, we recall the relation of the space of $G(\hat{\OO})$-invariant vectors in $\wt{\SS}_\eta$ with Higgs bundles on $C_0$ (see \cite[Sec.\ 3.6]{BKP-aut}).
Let $\MM^{Higgs,\NN^{-1}}(C_0)$ denote the groupoid of $\NN^{-1}$-twisted Higgs $G$-bundles on $C_0$, i.e., pairs
$(P_0,\phi)$, where $P_0$ is a $G$-bundle on $C_0$ (trivial at the general point), and $\phi$ is a global section of the vector bundle $\fg^\vee_{P_0}\ot \NN^{-1}\om_{C_0}$,
and let $\MM_{\eta}^{Higgs,\NN^{-1}}(C_0)\sub \MM^{Higgs,\NN^{-1}}(C_0)$ denote the subgroupoid of $(P_0,\phi)$ such that the orbit of $\phi$ at the general point is $\Om_{\eta}$.

The starting point is 
a natural identification of groupoids
\begin{equation}\label{Higgs-groupoids-eq}
\MM_{\eta}^{Higgs,\NN^{-1}}(C_0)\simeq \bSt_\eta(F_0)\backslash G(\A_{C_0})_\eta/G(\hat{\OO}_0) \sub \bSt_\eta(F_0)\backslash G(\A_{C_0})/G(\hat{\OO}_0),
\end{equation}
where $\hat{\OO}_0\sub \A_{C_0}$ are integer adeles, and
$G(\A_{C_0})_\eta\sub G(\A_{C_0})$ is the subset of $g\in G(\A_{C_0})$ such that 
$$\Ad(g^{-1})(\eta)\in \fg^\vee\ot \NN^{-1}\om_{C_0}(\hat{\OO}_0)\sub\fg^\vee\ot \NN^{-1}\om_{C_0}(\A_{C_0}).$$

Recall that for $g=(g_p)\in G(\A_{C_0})$, we have the $G$-bundle $P(g)$ on $C_0$, equipped with trivializations $t_{gen}$ and $(t_p)$ at the general point and at
all closed points, such that $t_p=t_{gen}g_p$ (we think of $P(g)$ as a right $G$-torsor). 
Then $\Ad(g^{-1})(\eta)$ defines a global section of $\fg^\vee_{P(g)}\ot \NN^{-1}\om_{C_0}$ if and only if $g\in G(\A_{C_0})_\eta$.
Under the isomorphism \eqref{Higgs-groupoids-eq} the double coset of $g\in G(\A_{C_0})\eta$ corresponds to the Higgs bundle $(P(g),\Ad(g^{-1})(\eta))$.

Let $G(\A_C)_\eta\sub G(\A_C)$ denote the preimage of $G(\A_{C_0})_\eta\sub G(\A_{C_0})$. Then the natural projection
$$\pi:\St_\eta(F)\backslash G(\A_C)_\eta/G(\hat{\OO})\to \bSt_\eta(F_0)\backslash G(\A_{C_0})_\eta/G(\hat{\OO}_0)$$
can be identified with the morphism of groupoids
$$\pi^{Higgs}:\MM_\eta^{Higgs,\NN^{-1}}(C)\to \MM_\eta^{Higgs,\NN^{-1}}(C_0),$$
where $\MM_\eta^{Higgs,\NN^{-1}}(C)$ is the groupoid of pairs $(P,\phi)$, where $P$ is a $G$-bundle on $C$ (trivial at the general point) and $\phi$ is an $\NN^{-1}$-twisted
Higgs field on $P_0$, the induced $G$-bundle on $C_0$.

The left multiplication by $N_\A\simeq \fg\ot \NN(\A_{C_0})$ acts transitively on each fiber of $\pi$, and identifies each fiber with a torsor over a finite quotient of $N_\A$.
This can be also seen in terms of the map $\pi^{Higgs}$: the fiber over $(P_0,\phi)$ is identified with the $H^1(C_0,\fg_{P_0}\ot\NN)$-torsor of liftings of $P_0$ to a $G$-bundle
on $C$. The pairing with the $\NN^{-1}$-twisted Higgs field $\phi\in H^0(C_0,\fg^\vee_{P_0}\ot \NN^{-1}\om_{C_0})$ gives a homomorphism
\begin{equation}\label{phi-pair-H1-hom}
\lan \phi,?\ran: H^1(C_0,\fg_{P_0}\ot\NN)\to H^1(C_0,\om_{C_0})\simeq k.
\end{equation}
Under the identification of $H^1(C_0,\fg_{P_0}\ot\NN)$ with a quotient of $\fg\ot \NN(\A_{C_0})$, the latter map is induced by the homomorphism
$$\fg\ot \NN(\A_{C_0})\rTo{\lan \eta, \cdot\ran} \om_{C_0}(\A_{C_0})\rTo{\sum_p \Res_p} k.$$
Composing \eqref{phi-pair-H1-hom} with the additive character $\psi:k\to \C^*$, we get 
a homomorphism 
$$H^1(C_0,\fg_{P_0}\ot\NN)\to\C^*.$$ 
Taking the push out of each fiber of $\pi^{Higgs}$ (or equivalently $\pi$) with respect to this homomorphism,
we get a $\C^*$-torsor $\rL_\psi$ over $\MM_\eta^{Higgs,\NN^{-1}}(C_0)$.

We can think of the space $\SS(\MM_\eta^{Higgs,\NN^{-1}}(C_0),\rL_\psi)$ (of finitely supported sections of $\rL_\psi$ over $\MM_\eta^{Higgs,\NN^{-1}}(C_0)$)
as the space of finitely supported functions $f$ on $\MM_\eta^{Higgs,\NN^{-1}}(C)$ such that
\begin{equation}\label{f-psi-phi-eq}
f(e+x)=\psi(\lan \phi,e\ran)f(x),
\end{equation}
where $x=(P,\phi)$ and $e\in H^1(C_0,\fg_{P_0}\ot\NN)$.
This induces an identification
$$\SS(\MM_\eta^{Higgs,\NN^{-1}}(C_0),\rL_\psi)\simeq \wt{\SS}_{\eta}^{G(\hat{\OO})}.$$

\begin{lemma}\label{kappa-modular-lem} 
The above map 
\begin{equation}\label{kappa-eta-Higgs-eq}
\kappa_\eta:\SS(\MM_\eta^{Higgs,\NN^{-1}}(C_0),\rL_\psi)\simeq \wt{\SS}_{\eta}^{G(\hat{\OO})}\to \SS(G(F)\backslash G(\A_C))^{G(\hat{\OO})}\simeq
\SS(\Bun_G(C))
\end{equation}
is given by the push-forward with respect to the natural map of groupoids 
\begin{equation}\label{forg-map}
\forg: \MM_\eta^{Higgs,\NN^{-1}}(C)\to \Bun_G(C): (P,\phi)\mapsto P.
\end{equation}
\end{lemma}

\begin{proof}
The map \eqref{forg-map} corresponds to the natural map of double coset groupoids
$$\forg:\St_\eta(F)\backslash G(\A_C)_\eta/G(\hat{\OO})\to G(F)\backslash G(\A_C)/G(\hat{\OO}).$$
It is easy to see that taking the push-forward with respect to this map takes a function $f$ on $\St_\eta(F)\backslash G(\A_C)_\eta/G(\hat{\OO})$ to
$$\forg_*(f)(g)=\sum_{\ga\in \St_\eta(F)\backslash G(F)}f(\ga g)=\sum_{\ga\in \St_\eta(F)\backslash G(F)}(g\cdot f)(\ga),$$
where $g\in G(F)$. 
This proves that $\forg_*(f)=\kappa_\eta(f)$.
\end{proof}

\begin{cor} Assume that $\eta$ is regular semisimple and elliptic.
Let $f$ be a finitely supported function on $\MM_\eta^{Higgs,\NN^{-1}}(C)$ satisfying \eqref{f-psi-phi-eq}.
Then the function $\forg_*(f)$ on $\Bun_G(C)$ is strongly cuspidal.
\end{cor}

\begin{proof}
This follows from Lemma \ref{reg-ell-str-cusp-lem}.
\end{proof}

In the case $C=C_0\times \Spec(k[\eps]/(\eps^2))$, for every $G$-bundle $P_0$ on $C_0$ we have a canonical lifting $P_0\times \Spec(k[\eps]/(\eps^2))$ to a $G$-bundle on $C$.
This gives a trivialization of the $\C^*$-torsor $\rL_\psi$, so we can view the map $\kappa_\eta$ as a map from $\SS(\MM_\eta^{Higgs,\NN^{-1}}(C_0))$ to $\SS(\Bun_G(C))$.

\begin{cor}\label{trivial-ext-FT-cor}
In the case $C=C_0\times \Spec(k[\eps]/(\eps^2))$, the map
$$\kappa_\eta:\SS(\MM_\eta^{Higgs,\NN^{-1}}(C_0))\to \SS(\Bun_G(C))$$
is given by
$$\kappa_\eta(f)(P_0+e)=\sum_{\phi\in H^0(C_0,\fg_{P_0}\ot \NN^{-1}\om_{C_0})}\psi(\lan\phi,e\ran)f(P_0,\phi),$$
where $P_0$ is a $G$-bundle, $e\in H^1(C_0,\fg_{P_0}\ot\NN)$, and $P_0+e$ denote the $G$-bundle on $C$ obtained from the canonical lift of $P_0$ to a $G$-bundle by
applying the action of $e$.
\end{cor}

The formula of Corollary \ref{trivial-ext-FT-cor} resembles the Fourier transform.
Below we will show that even when the extension $C_0\sub C$ is non-trivial the map $\kappa_\eta$ is given by a certain kind of the Fourier transform.

\subsection{Fourier transform for torsors}\label{Four-sec}

Let $\Ga$ be a (discrete) groupoid, $V$ a $k$-vector bundle over $\Ga$, where $k$ is a finite field. In other words, $V$ is a functor from $\Ga$ to the category of finite-dimensional
vector spaces over $k$, so it is a collection of vector spaces $V_\ga$ indexed by objects, together with linear isomorphisms $\a:V_\ga\to V_{\ga'}$ corresponding to each arrow
$\a:\ga\to \ga'$ in $\Ga$, such that the compositions are respected. 

The total space $\tot(V)$ is a groupoid of pairs $(\ga,v)$, where $\ga\in \Ga$, $v\in V_\ga$, with maps $(\ga,v)\to (\ga',v')$ given by arrows $\a:\ga\to \ga'$ in $\Ga$
such that $\a(v)=v'$. Let $V^\vee$ denote the dual vector bundle to $V$ over $\Ga$. We denote by
$p:\tot(V)\to \Ga$ and $p^\vee:\tot(V^\vee)\to \Ga$ the natural projections.

Let $\psi:k\to \C^*$ be a nontrivial additive character. There is a natural Fourier transform
$$F:\SS(\tot(V^\vee))\to \SS(\tot(V))$$
given by $(Sf)(\ga,v)=\sum_{v^*\in V^\vee_\ga}\psi(\lan v^*,v\ran)f(\ga,v^*)$.

Now let $T$ be a $V$-torsor over $\Ga$, i.e., a collection of $V_\ga$-torsors $T_\ga$, such that for every arrow $\a:\ga\to \ga'$,
one has an associated isomorphism $\a:T_{\ga}\to T_{\ga'}$ compatible with the isomorphism $\a:V_{\ga}\to V_{\ga'}$ and with the torsor structures (and the compositions are
 respected). We can also consider the total space $\tot(T)$, which is a groupoid equipped with a representable map to $\Ga$, with the fiber $T_\ga$ over $\ga$.
We want to describe an analog of the above Fourier transform where $\tot(V)$ is replaced by $\tot(T)$.
 

Let us define a $\C^*$-torsor $\LL_{T,\psi}$ over $\tot(V^\vee)$ as follows.
Let $p^*T$ denote the $p^*(V)$-torsor over $\tot(V^\vee)$, obtained by taking the pullback of $T$.
We have a natural homomorphism of vector bundles $p^*(V)\to k$ (where $k$ is the trivial rank $1$ bundle), so taking the push-out of $p^*(T)$ we
get a $k$-torsor. The $\C^*$-torsor $\LL_{T,\psi}$ is obtained by further taking the push-out with respect to $\psi$.  
In down-to-earth terms, for every object $(\ga,v^*)$ of $\tot(V^\vee)$, where $v^*\in V^*_\ga$, the $\C^*$-torsor 
$\LL_{T,\psi;(\ga,v^*)}$ is defined as the push-out of the $V_\ga$-torsor $T_\ga$ with respect to the homomorphism $V\to \C^*:v\mapsto \psi(\lan v^*,v\ran)$.

Let us consider the fibered product $\tot(V^\vee)\times_{\Ga} \tot(T)$, with the projections $p_1,p_2$ to the factors.
We define the Fourier transform
$$F:\SS(\tot(V^\vee),\LL_{T,\psi})\to \SS(\tot(T)),$$
by 
$$F(\phi)=p_{2*}p_1^*(\phi),$$
where we use a canonical trivialization of $p_1^*\LL_{T,\psi}$ over $\tot(V^\vee)\times_{\Ga} \tot(T)$ (coming from the trivialization of the pullback of $T$ to $\tot(T)$).

More explicitly, by construction for every $(\ga,v^*)$ we have the push-out map 
$$\psi_{\ga,v^*}:T_\ga\to \LL_{T,\psi;(\ga,v^*)},$$
satisfying 
$$\psi_{\ga,v^*}(v+t)=\psi(\lan v^*,v\ran)\cdot \psi_{\ga,v^*}(t),$$
for any $v\in V_\ga$.
The elements $\phi\in\SS(\tot(V^\vee),\LL_{T,\psi})$ can be thought of as functions on $\tot(\LL_{T,\psi})$ of weight $1$ with respect to the $\C^*$-action.
Now for $t\in T_\ga$, we have
$$F(\phi)(t)=\sum_{v^*\in V_\ga^\vee} \phi(\psi_{\ga,v^*}(t)).$$
For functions on the fibers over fixed $\ga\in \Ga$, we can trivialize $T_\ga$, and this operation reduces to the usual Fourier transform from $\SS(V^*)$ to $\SS(V)$. 

\begin{prop}\label{Fourier-prop}
Let us consider $\Bun_G(C)$ as the torsor over $\Bun_G(C_0)$ over the vector bundle $V$ associating $V_{P_0}:=H^1(C_0,\fg_{P_0}\ot\NN)$ to a $G$-bundle $P_0$ over $C_0$.
Then we have an identification $\tot(V^\vee)\simeq \MM^{Higgs,\NN^{-1}}(C_0)$, and the map 
\eqref{kappa-eta-Higgs-eq} is given by the restriction of the Fourier transform
$$F:\SS(\MM^{Higgs,\NN^{-1}}(C_0),\rL_\psi)\to \SS(\Bun_G(C))$$
to the subspace of sections supported on $\MM_\eta^{Higgs,\NN^{-1}}(C_0)$
\end{prop}

\begin{proof}
This follows immediately from Lemma \ref{kappa-modular-lem} using the duality between $V$ and the fibration $\MM^{Higgs,\NN^{-1}(C_0)}\to \Bun_G(C_0)$.
\end{proof}


\subsection{Relation to induced representations}

Let $\chi$ be a character (smooth $1$-dimensional representation) of $\St_\eta$, trivial on $\St_\eta(F)$ and such that $\chi|_{N_\A}=\psi_\eta$.
We denote by $V_{\eta,\chi}$ the representation of $G(\A_C)$ compactly induced from the character $\chi$ of $\St_{\eta}$.
Note that in \cite{K-YD} the authors study similar induced representations in the local case for $G=\PGL_2$.

\begin{lemma}\label{ind-char-ell-lem}
Assume that $\eta$ is regular semisimple and elliptic. Then we have a natural inclusion of $G(\A_C)$-representations
$V_{\eta,\chi}\sub \wt{\SS}_{\eta}$.
\end{lemma}

\begin{proof}
By definition, $V_{\eta,\chi}$ is the space of locally constant functions $f:G(\A_C)\to \C$ such that $f(hg)=\chi(h)f(g)$ for $h\in \St_\eta$, such that the support of $f$
modulo $\St_\eta$ is compact. We need to check that in fact the support of $f$ modulo $N_\A\cdot \St_\eta(F)$ is compact. It is enough to check that
$\St_\eta/\St_\eta(F)$ is compact. But this follows from the exact sequence
$$1\to N_\A\to \St_\eta\to \bSt_\eta(\A_{C_0})\to 1$$  
and from compactness of $N_\A/N_F$ and $\bSt_\eta(\A_{C_0})/\bSt(F_0)$.
\end{proof}

The restriction of $\Theta_{\eta}$ to $V_{\eta,\chi}$ is the natural $G(F)$-invariant functional still
given by \eqref{Theta-eta-def}. Combining Lemma \ref{ind-char-ell-lem} with Lemma \ref{reg-ell-str-cusp-lem}
we get the following result.

\begin{cor}
For any $\phi\in V_{\eta,\chi}$ the function 
$$g\mapsto \Theta_\eta(g\cdot \phi)=\sum_{\ga\in \St_\eta(F)\backslash G(F)}\phi(\ga\cdot g)$$
on $G(F)\backslash G(\A_C)$ is strongly cuspidal.
\end{cor}



\section{Theta correspondence for $\SL_2$ and a curve over a finite field}\label{fin-field-sec}

\subsection{Local and global representations of $\SL_2$}

\subsubsection{Local picture}\label{local-setup-fin-sec}

Let $E$ be a local nonarchimedean field, $K/E$ either an unramified quadratic field extension, or the algebra $K=E\oplus E$.
\footnote{The picture below can also be extended to the case of a ramified quadratic extension (see
 \cite[Prop.\ 1.3]{JL}). Since we only use this as a motivation, we restrict ourselves to the unramified case.}
Let $\vartheta=\vartheta_{K/E}:E^*\to \{\pm 1\}$ denote the corresponding character (trivial for $K=E\oplus E$), whose kernel is the subgroup of norms from $K$. 

Let $\psi_E$ be an additive character of $E$, trivial on $\OO_E$ but nontrivial on $t^{-1}\OO_E$, where $t\in \fm_E$ is a uniformizer.
We set $\psi_K:=\psi_E\circ \Tr_{K/E}$, so that $\psi_K$ is trivial on $\OO_K$ but nontrivial on $t^{-1}\OO_K$ (note that $t$ is still a uniformizer for $K$,
when $K$ is a field). We normalize the Haar measure on $K$ so that $\vol(\OO_K)=1$.
The corresponding absolute value on $K$ satisfies $|t|_K=q^{-2}$, where $q=|\OO_E/\fm_E|$.
The Fourier transform of $\Phi\in \SS(K)$ is given by
$$\widehat{\Phi}(y)=\int_K \Phi(x)\psi_K(xy) dx.$$
It satisfies 
$$\widehat{\de_{\OO_K}}=\de_{\OO_K},$$
where $\de_{\OO_K}$ is the characteristic function of $\OO_K\sub K$.

Let $\iota:K\to K$ denote the nontrivial $E$-automorphism of $K$.
We can consider the $4$-dimensional symplectic space $K\ot_E (E^2)$ over $E$, where we use the quadratic form $\Tr_{K/E}(x\iota(y))$ on $K$.
Then the corresponding metaplectic extension splits over $\SL_2(E)$, so we get the representation $r$ of $\SL_2(E)$ on $\SS(K)$ (see \cite[Prop.\ 1.3]{JL}),
so that
$$r\left(\begin{matrix} \a & 0\\ 0 & \a^{-1}\end{matrix}\right)\Phi(x)=\vartheta(\a)|\a|_E \Phi(\a x),$$
$$r\left(\begin{matrix} 1 & z\\ 0 & 1\end{matrix}\right)\Phi(x)=\psi_E(z\Nm_{K/E}(x)) \Phi(x),$$
$$r\left(\begin{matrix} 0 & 1\\ -1 & 0\end{matrix}\right)\Phi(x)=\widehat{\Phi}(\iota(x)),$$
where  $|\a|_E=q^{-v_E(\a)}$.

It is clear from these formulas that the function $\de_{\OO_K}$ is invariant under the subgroup $\SL_2(\OO_E)$.

One also has a twisted version of the above picture. Let $L$ be a free $\OO_K$-module of rank $1$, and let $L_K:=L\ot_{\OO_K} K$.
We have the corresponding $E$-module $N(L_K):=\Nm_{K/E}(L_K)$ and an $\OO_E$-submodule $N(L):=\Nm_{\OO_K/\OO_E}(L)$, and a 
natural quadratic map
$$N:L_K\to N(L_K):x\mapsto \Nm_{L/K}(x),$$
sending $L$ to $N(L)$. 
Assume $\psi$ is a character of $N(L_K)$, trivial on $N(L)$, but not on $t^{-1}N(L)$. We
have an induced $E$-linear pairing
$$T:L_K\times L_K\to E: T(x,y)=N(x+y)-N(x)-N(y),$$
which we can use to define the Fourier transform on $\SS(L_K)$ (an analog of $\Phi\mapsto \widehat{\Phi}(\iota(?))$).
 
Now we observe that the above formulas giving a representation of $\SL_2(E)$ on $\SS(K)$ also make sense if we replace $\SS(K)$ by $\SS(L_K)$ and
they still define a representation of $\SL_2(E)$ (since we can use a trivialization $L\simeq \OO_K$).

\subsubsection{Global picture for curves over finite field}

Now, let $\pi:\wt{C}\to C$ be a double \'etale covering
of a curve $C$ over a finite field $k$, with $\wt{C}$ connected.
Let us denote by $\A_C$ and $\A_{\wt{C}}$ the corresponding rings of adeles, and by $\hat{\OO}_C\sub \A_C$, $\hat{\OO}_{\wt{C}}\sub\A_{\wt{C}}$, the subgroups of integer adeles.
We fix a nontrivial additive character $\psi$ of $k$, and consider
the character
$$\psi_{C}:\om_{{C}}(\A_{{C}})/(\hat{\om}_{{C}}+\om_C(k({C})))\to U(1):\a\mapsto \psi(\sum_p \Res_p\a),$$
where we consider the twisted adeles $\om_{{C}}(\A_{{C}})$, and the corresponding integer adeles $\hat{\om}_{C}$.
Note that for each $p\in C$, we have the corresponding local additive character of $\om_{C,p}\ot F_p$, trivial on $\om_{C,p}$ but not on $t^{-1}\om_{C,p}$.

Let us fix a line bundle $\LL$ on $\wt{C}$ together with an isomorphism of line bundles on $C$, 
$$\Nm(\LL)\rTo{\sim} \om_{C},$$ 
and consider the corresponding twisted adeles
$\LL(\A_{\wt{C}})$ with its subgroup $\hat{\LL}$ of twisted integer adeles.

Using the local picture discussed above
we associate with the character $\psi_C$ a representation $r$ of the adelic group $\SL_2(\A_C)$ on $\SS(\LL(\A_{\wt{C}}))$, 
so that the characteristic function $\de_{\hat{\LL}}$ is invariant under $\SL_2(\hat{\OO}_C)$.
On the other hand, we have an $\SL_2(k(C))$-invariant functional
$$\th_{\LL}:\Phi\mapsto \sum_{x\in \LL(k(\wt{C}))} \Phi(x).$$
Hence, $g\mapsto \th_\LL(g\de_{\hat{\LL}})$ descends to a function $f_{\LL}$ on $\SL_2(k(C))\backslash \SL_2(\A_C)/\SL_2(\OO_C)=\Bun_{\SL_2}(C)$.

Note that $\SL_2(k(C))$-invariance of $\th_{\LL}$ follows from the following facts:
\begin{itemize}
\item for $\a\in k(C)^*$, $\prod_p \vartheta_p(\a)=1$, $\prod_p |\a|_p=1$;
\item for $\eta\in \om(k(C))$, $\psi_C(\eta)=1$;
\item $\sum_{x\in \LL(k(\wt{C}))} \hat{\Phi}(x)=\sum_{x\in \LL(k(\wt{C}))} \Phi(x)$,
\end{itemize}
where $\vartheta_p$ is the local multiplicative character associated with the local quadratic extension at a point $p\in C$.
Note that the last property follows from the fact that $\LL(k(\wt{C}))$ is its own orthogonal with respect to
the pairing $\psi_C(\Tr(x\iota(y)))$ on $\LL(\A_{\wt{C}})$. 

Note also that the local characters $\vartheta_p$ are local components of the Artin homomorphism 
$$\vartheta_{\wt{C}/C}:\Pic(C)\to \Gal(k(\wt{C})/k(C))\simeq \{\pm 1\},$$
whose kernel is exactly the image of the norm homomorphism $\Pic(\wt{C})\to \Pic(C)$, and for every point $p\in C$, we have
$$\vartheta_{\wt{C}/C}(\OO_C(p)):=\begin{cases} 1, & |\pi^{-1}(p)|=2,\\ -1, & |\pi^{-1}(p)|=1.\end{cases}$$

\subsection{Cuspidal functions}\label{fin-cusp-sec}

Let $\PP(\wt{C}/C)$ denote the set of isomorphism classes of pairs $(\LL,\a:\Nm(\LL)\rTo{\sim} \om_{C})$.
Clearly, this is a principal homogeneous space over the group $\PP_0(\wt{C}/C)$
of isomorphism classes of pairs $(\LL,\a:\Nm(\LL)\rTo{\sim} \OO_{C})$. We have an exact sequence
$$1\to k^*/(k^*)^2\to \PP_0(\wt{C}/C)\to \ker(\Nm:\Pic(\wt{C})\to \Pic(C))\to 0$$

Let $\A_{\wt{C},1}^*\sub \A^*_{\wt{C}}$ denote the subgroup of ideles of $\wt{C}$ that have norm $1$. This group acts on $\SS(\LL(\A_{\wt{C}}))$ by
$x\Phi(y)=\Phi(xy)$ and this action commutes with that of $\SL_2(\A_C)$. Hence, for any $\a\in \A_{\wt{C},1}^*$, the function $\a\de_{\hat{\LL}}$
is still $\SL_2(\OO_C)$-invariant, and we can consider the function
$$
f_{\LL,\a}:=\th_{\LL}(g \a\de_{\hat{\LL}})$$
on $\Bun_{\SL_2}(C)$.

Consider the group
$$K(\wt{C}/C):=\A^*_{\wt{C},1}/(k(\wt{C})_1\cdot (\A^*_{\wt{C},1}\cap \OO_{\wt{C}}).$$
Note that we have a natural homomorphism
\begin{equation}\label{K-wtC-C-hom}
K(\wt{C}/C)\to \PP_0(\wt{C}/C):\a\mapsto \OO(\a),
\end{equation}
where $\OO(\a)$ is the line bundle associated with the corresponding idele class.

\begin{lemma}
The homomorphism \eqref{K-wtC-C-hom} is an isomorphism. For each $\a\in K(\wt{C}/C)$, one has an equality of functions on $\Bun_{\OO}(C)$,
\begin{equation}\label{f-L-a-eq}
f_{\LL,\a}=f_{\LL\ot \OO(\a)}.
\end{equation}
\end{lemma}

\begin{proof}
Let us first check injectivity of \eqref{K-wtC-C-hom}.
By definition, $\OO(\a)$ is equipped with trivializations $\varphi_p$ and $\varphi_{\eta_{\wt{C}}}$ such that $\varphi_{\eta_{\wt{C}}}=\a_p\cdot \varphi_p$.
We have the natural isomorphism $\Nm(\OO(\a))\simeq \LL(\Nm(\a))=\LL(1)=\OO_C$.
Suppose there exists a trivialization of $\OO(\a)$ with norm $1$. Then the corresponding trivialization at the general point  (resp., at $p$) has form $f\cdot\varphi_{\eta_{\wt{C}}}$ with $f\in k(C)^*$
(resp. $f_p\cdot \varphi_p$ with $f_p\in \hat{\OO}_{\wt{C},p}^*$), such that $f\a_p=f_p$ and $\Nm(f)=1=\Nm(f_p)$. But this means that $\a_p$ is trivial in $K(\wt{C}/C)$.

Now, let $\LL$ be a line bundle on $\wt{C}$, together with an isomorphism $\Nm(\LL)\rTo{\sim} \OO_C$.
Let us choose trivializations $\varphi_p$ of $\LL_p$ and $\varphi_{\eta_{\wt{C}}}$ of $L_{\eta_{\wt{C}}}$.
Then we have the idele $\a\in \A^*_{\wt{C}}$ such that $\varphi_{\eta_{\wt{C}}}=\a_p\cdot \varphi_p$, so we have
an isomorphism $\OO(\a)\rTo{\sim}\LL$. Furthermore, the isomorphism $\Nm(\LL)\to \OO_C$ sends $\Nm(\varphi_{\eta_{\wt{C}}})$
to some $f\in k(C)^*\cap (\Nm(\a)\cdot \hat{\OO}^*_C)$. This implies that $f$ is a norm locally, so it is a norm globally, i.e., $f=\Nm(\wt{f})$
for some $\wt{f}\in k(\wt{C})^*$. Thus, modifying $\eta_{\wt{C}}$ and $\a$, we can achieve that $f=1$. Similarly, modifying $\varphi_p$ and $\a$,
we can achieve that the isomorphism $\Nm(\LL)\to \OO_C$ sends $\Nm(\varphi_p)$ to $1$. It follows then that $\Nm(\a)=1$, so we proved
surjectivity of \eqref{K-wtC-C-hom}.

The identity \eqref{f-L-a-eq} follows from the observation that the map 
$$\LL(\A_{\wt{C}})\to (\LL\ot\OO(\a))(\A_{\wt{C}}): x\mapsto \a x$$
induces an isomorphism of $\SL_2(\A_C)\times \A^*_{\wt{C},1}$-representations $\SS(\LL(\A_{\wt{C}}))$ and $\SS((\LL\ot\a)(\A_{\wt{C}}))$,
so that the functional $\th_{\LL\ot\OO(\a)}$ corresponds to the functional $x\mapsto \th_{\LL}(\a x)$.
\end{proof}

\begin{prop}
Assume that $\LL=\pi^*\LL_0$, where $\LL_0$ is a theta-characteristic on $C$, defined over $k$. 
If $\chi$ is a nontrivial character of $K(\wt{C}/C)$ then the function 
\begin{equation}\label{Phi-L-chi-eq}
\Phi_{\LL,\chi}=\sum_{\a\in K(\wt{C}/C)}\chi(\a) f_{\LL,\a}
\end{equation}
on $\Bun_{\SL_2}(C)$ is cuspidal and nonzero.
\end{prop}

\begin{proof}
We have for $\Phi\in \SS(\LL(\A_{\wt{C}}))$,
$$\int_{u\in U(F)\backslash U(\A_C)}\th_\LL(u\Phi) du=\sum_{x\in \LL(k(\wt{C}))}\int_{z\in \A_C/F} \psi(z\Nm(x))dz \Phi(x).$$
Since $\psi(zy)$ is a nontrivial character of $\A_C/F$ for any $y\in F^*$, the integral in the right-hand side is zero for $x\neq 0$.
Hence, we get
$$\int_{u\in U(F)\backslash U(\A_C)}\th_\LL(u\Phi)du=\vol(\A_C/F)\cdot \Phi(0).$$

Therefore, for $t\in T(\A_C)$, we have
$$Cf_{\LL,\a}(t)=\int_{U(F)\backslash U(\A_C)}\th_\LL(ut\a\cdot\de_{\hat{\LL}}) du=\vol(\A_C/F) (t\a\cdot \de_{\hat{\LL}})(0),$$
which does not depend on $\a$. Hence, summing over $\a$ with weights $\chi(\a)$, we get zero.

Now let us show that $\Phi_{\LL,\chi}$ is nonzero. For this, let us first calculate $f_{\LL,\a}(t)$ for $t=\diag(\b,\b^{-1})$, for some $\b\in \A_C^*$.
By definition, we have
$$f_{\LL,\a}(t)=c(t)\cdot \sum_{x\in \LL(k(\wt{C}))}\de_{\hat{\LL}}(\a\b x)=c(t)\cdot |H^0(\wt{C},\LL\ot \OO(\a\b)|,$$
where $c(t)=\prod_{p}\vartheta_p(\b_p)|\b|_{F_p}$.

Now let $\b$ be such that $\OO(\b)\simeq \LL_0^{-1}$. Then for any $\a\in K(\wt{C}/C)$, the line bundle $\pi^*\LL_0\ot \OO(\a)\ot \OO(\b)$ will have degree $0$,
and it will be trivial only when $\a$ is trivial. Thus, we get
$$\Phi_{\LL,\chi}(t)=\sum_{\a\in K(\wt{C}/C)}\chi(\a)f_{\LL,\a}(t)=c(t)\cdot (|\OO(\wt{C})|+\sum_{\a\in K(\wt{C}/C)\setminus \{0\}}\chi(\a))=c(t)\cdot (|\OO(\wt{C})|-1)\neq 0,$$
which shows that $\Phi_{\LL,\chi}\neq 0$.
\end{proof}

\subsection{Spherical vectors and the Hecke action}\label{fin-field-Hecke-sec}

Let us return to the local picture associated with a degree $2$ extension $K/E$.
Let $K^*_1\sub K^*$ denote the group of elements $x$ such that $\Nm_{K/E}(x)=1$.
We have an action of $K^*_1$ on $\SS(K)$ given by $r(x)\Phi(y)=\Phi(xy)$. It is easy to see that this
action commutes with $\SL_2(E)$.

Now let us consider the space of spherical vectors $\SS(K)^{G(\OO_E)}$, where $G=\SL_2$. 
It has the induced action of $K^*_1$
and also the action of the Hecke algebra $\SS(G(\OO_E)\backslash G(E)/G(\OO_E))$.
Let us set $\OO^*_{K,1}:=\OO_K^*\cap K^*_1$.
Then the subspace $\SS(K)^{G(\OO_E)\times \OO^*_{K,1}}$ is preserved by the Hecke operators, and has
the induced action of $K^*_1/\OO^*_{K,1}$.
Note that $K^*_1/\OO^*_{K,1}$ is trivial if $K$ is a field and is isomorphic to $\Z$ if $K=E\oplus E$.

\begin{lemma}
One has an isomorphism of $K^*_1/\OO^*_{K,1}$-representations
$$\SS(K^*_1/\OO^*_{K,1})\hra \SS(K)^{G(\OO_E)\times \OO^*_{K,1}}:
\varphi\mapsto \sum_{x\in K^*_1/\OO^*_{K,1}}\varphi(x)\de_{x\cdot \OO_K}.$$
\end{lemma}

\begin{proof}
We know that $\de_{\OO_K}$ is $G(\OO_E)\times \OO^*_{K,1}$-invariant. Hence, the same is true for $\de_{x\OO_K}$ for any $x\in K^*_1$.

Assume first that the extension $K/E$ is non-trivial, i.e., $K$ is a field. Then invariance with respect to upper-triangular integer matrices implies that any 
$\Phi\in \SS(K)^{G(\OO_E)}$ is supported on $\OO_K$.
By the Fourier invariance of $\Phi$, this implies that $\Phi$ is invariant under translations by elements of $\OO_K$, hence,
$\Phi$ is proportional to $\de_{\OO_K}$.

In the case of the trivial extension $K=E\oplus E$, the proof is obtained using the Fourier transform in one variable
(see \cite[Prop.\ 1.6]{JL}).
\end{proof}

Thus, in the case when $K/E$ is nontrivial, the space $\SS(K)^{G(\OO_E)\times \OO^*_{K,1}}$ is spanned
by $\de_{\OO_K}$, whereas in the case $K=E\oplus E$, it has as a basis 
$$(t^n,t^{-n})\cdot \de_{\OO_K}=\de_{t^{-n}\OO_E\oplus t^n\OO_E}$$
numbered by $n\in \Z$, where $t\in \OO_E$ is the generator of the maximal ideal.

Now let us consider the action of the Hecke operators $(T_m)_{m\ge 1}$ corresponding to the double cosets
$$G(\OO_E)\cdot \diag(t^m,t^{-m})\cdot G(\OO_E)$$
on the space $\SS(K)^{G(\OO_E)\times \OO^*_{K,1}}$.

Let $A_{G(\OO)}$ denote the $G(\OO_E)$-averaging operator.
Then the Hecke operator $T_{g_0}$ corresponding to the double coset $G(\OO_E)g_0G(\OO_E)$ can be computed as
\begin{equation}\label{T-g0-eq}
T_{g_0}f(g)=|G(\OO)/G(\OO)\cap g_0G(\OO)g_0^{-1}|\cdot A_{G(\OO)}g_0f
\end{equation}
(see \cite[Sec.\ 2]{BKP-aut}). 
In case of the operator $T_m$, i.e., for $g_0=\diag(t^m,t^{-m})$, $G(\OO)\cap g_0G(\OO)g_0^{-1}\sub G(\OO)$ is the subgroup preserving a line in
$(\OO/(t^{2m}))^2$, so the constant factor in \eqref{T-g0-eq} becomes $|\P^1(\OO/(t^{2m}))|=q^{2m}+q^{2m-1}$.

Let $\Pi:\SS(K)\to \SS(K)$ denote the operator of averaging with respect to the group $U(\OO_E)$ (strictly upper triangular integer matrices).
Due to our assumptions on $\psi$, we have
$$\Pi(\Phi)=\de_{\Nm^{-1}(\OO_E)}\cdot \Phi.$$
We also have $A_{G(\OO)}=A_{G(\OO)}\Pi$. 

\begin{lemma} (i) If $K/E$ is nontrivial then
$$T_m\de_{\OO_K}=(-1)^m(q^m+q^{m-1})\de_{\OO_K}.$$

\noindent
(ii) If $K=E\oplus E$ then
$$T_1\de_{\OO_E\oplus \OO_E}=(q-1)\de_{\OO_E\oplus \OO_E}+q\de_{t\OO_E\oplus t^{-1}\OO_E}+q\de_{t^{-1}\OO_E\oplus t\OO_E}.$$
\end{lemma}

\begin{proof}
(i) We have $\diag(t^m,t^{-m})\de_{\OO_K}=(-1)^mq^{-m}\de_{t^{-m}\OO_K}$. Now we observe that for $m\ge 1$, $\Pi \de_{t^{-m}\OO_K}=\de_{\OO_K}$,
and the assertion follows.

\noindent
(ii) We have 
$$T_1\de_{\OO_E\oplus \OO_E}=(q^2+q)A_{G(\OO)}\diag(t,t^{-1})\de_{\OO_E\oplus \OO_E}=(q+1)A_{G(\OO)}\de_{t^{-1}\OO_E\oplus t^{-1}\OO_E}.$$
Applying the operator $\Pi$, we can replace $\de_{t^{-1}\OO_E\oplus t^{-1}\OO_E}$
by 
$$\phi:=\Pi\de_{t^{-1}\OO_E\oplus t^{-1}\OO_E}=\de_{(a,b)\in t^{-1}\OO_E^2: ab\in \OO_E}.$$
We can write 
$$\phi-\de_{\OO_E\oplus \OO_E}=\de_{(a,b)\in t^{-1}\OO_E^2\setminus \OO_E^2: ab\in \OO_E}=\de_{t^{-1}\OO_E\oplus t\OO_E}-\de_{\OO_E\times t\OO_E}+
\de_{t\OO_E\oplus t^{-1}\OO_E}-\de_{t\OO_E\times \OO_E}.$$
Note that
$$S\de_{\OO_E\times t\OO_E}=q^{-1}\de_{t^{-1}\OO_E\times \OO_E}.$$
Hence
\begin{align*}
&A_{G(\OO)}\de_{\OO_E\times t\OO_E}=q^{-1}A_{G(\OO)}\de_{t^{-1}\OO_E\times \OO_E}=q^{-1}A_{G(\OO)}\Pi\de_{t^{-1}\OO_E\times \OO_E}=\\
&q^{-1}A_{G(\OO)}(\de_{\OO_E^2}+\de_{t^{-1}\OO_E\times t\OO_E}-\de_{\OO_E\times t\OO_E}),
\end{align*}
which gives
$$A_{G(\OO)}\de_{\OO_E\times t\OO_E} =\frac{1}{q+1}(\de_{\OO_E^2}+\de_{t^{-1}\OO_E\times t\OO_E}).$$
Similarly,
$$A_{G(\OO)}\de_{t\OO_E\times \OO_E} =\frac{1}{q+1}(\de_{\OO_E^2}+\de_{t\OO_E\times t^{-1}\OO_E}).$$ 
Hence,
$$A_{G(\OO)}\phi=\frac{1}{q+1}((q-1)\de_{\OO_E^2}+q\de_{t^{-1}\OO_E\times t\OO_E}+q\de_{t\OO_E\times t^{-1}\OO_E}),$$
which leads to the claimed formula.
\end{proof}

Now, let us consider the global picture with the double covering $\pi:\wt{C}\to C$.

\begin{prop}
The function $\Phi_{L,\chi}$, given by \eqref{Phi-L-chi-eq}, 
satisfies
\begin{equation}\label{Hecke-fin-field-eq}
T_{1,p}\Phi_{L,\chi}=\begin{cases} -(q_p+1)\Phi_{L,\chi}, & |\pi^{-1}(p)|=1, \\ (q_p-1+q_p\chi(p_1-p_2)+q_p\chi(p_2-p_1))\Phi_{L,\chi}, 
&\pi^{-1}(p)=p_1+p_2, \end{cases}
\end{equation}
where $T_{1,p}$ is the Hecke operator associated with the point $p\in C$, and $q_p=|k(p)|$.
\end{prop}

\begin{remark}
The right-hand side of the formula \eqref{Hecke-fin-field-eq} has the following interpretation in terms of the Satake isomorphism $S$
between the local Hecke algebra of $\SL_2$ and the representation ring $R(\PGL_2)$ of $\PGL_2$. Namely, if $c_{n\a^\vee_1}$ is the element of the Hecke algebra of $\SL_2$
associated with the coroot $n\a^\vee_1$, where $n\ge 0$, then 
$$S(c_{n\a^\vee_1})=q^nV_{n\a^\vee_1}-q^{n-1}V_{(n-1)\a^\vee_1},$$
where $V_\la$ is the irreducible representation of $\PGL_2$ with the highest weight $\la$ (see \cite[Sec.\ 4]{Gross}).
Using this it is easy to see that the eigenvalue in \eqref{Hecke-fin-field-eq} is equal
to the trace of the Frobenius at $p$ in the Galois group $\Gal(k(C))$ on the virtual representation of the Galois group $\Gal(k(C))$ associated with a
$\PGL_2$ local system corresponding to $\chi$ and the element $S(c_{\a^\vee_1})$ of $R(\PGL_2)$.
\end{remark}

\subsection{Interpretation in terms of vector bundles}\label{fin-field-vec-sec}

We can compute the function $f_\LL$
on a vector bundle $V$ over $C$ with the trivial determinant, which is presented as an extension 
\begin{equation}\label{V-L-ext}
0\to L^{-1}\to V\to L\to 0
\end{equation}
where $L$ is a line bundle on $C$.

Let us choose a splitting 
$$\si_{gen}:L_{gen}\to V_{gen}$$ 
over the general point $\Spec k(C)$, and also for each closed  point  $p\in C$, a splitting
$$\si_p:L_p\to V_p$$
of the corresponding $\OO_{C,p}$-modules.
We have
$$\si_p-\si_{gen}=e_p\in L^{-2}_p\ot F_p$$
(these elements assemble into an element of $L^{-2}\ot \A_C$).
Also, let us choose trivializations
$$\phi_{gen}:k(C)\rTo{\sim} L_{gen}, \ \ \phi_p:\OO_{C,p}\rTo{\sim} L_p.$$
We have
$$\phi_p=\a_p\cdot \phi_{gen},$$
for some idele $(\a_p)$. 
We have the induced trivializations $\phi^{-1}_x$ of $L^{-1}_x$, and hence, the induced trivializations of $V$,
$$t_x:\OO_x\oplus \OO_x\rTo{(\phi^{-1}_x,\si_x\phi_x)} V_x.$$

For each closed point $p\in C$, we have
$$t_p=t_{gen}\circ g_p,$$
where $g_p\in \SL_2(F_p)$. Changing trivializations leads to changing $g_p$ in the corresponding double coset.
Thus, the value of our function on $V$ is given by $\th_{\LL}((g_p)\cdot \de_{\hat{\LL}})$.

It is easy to see that
$$g_p=\left(\begin{matrix} \a_p^{-1} & \a_pe_p\phi_{gen}^2\\ 0 & \a_p\end{matrix}\right)=
\left(\begin{matrix} 1 & e_p\phi_{gen}^2\\ 0 & 1\end{matrix}\right)\cdot \left(\begin{matrix} \a_p^{-1} & 0 \\ 0 & \a_p\end{matrix}\right).$$
Hence, 
\begin{equation}\label{gp-ext-action-formula}
((g_p)\cdot \de_{\hat{\LL}})((x_p))=
\psi(\sum_{p\in C}\Res_p(e_p\phi_{gen}^2\cdot \Nm(x_p)))\cdot \left(\begin{matrix} \a_p^{-1} & 0 \\ 0 & \a_p\end{matrix}\right)\cdot 
\prod_p \vartheta_p(\a_p)\cdot q_p^{v_p(\a_p)}\de_{\a_p\cdot \hat{\LL}_{\wt{C},p}}((x_p)).
\end{equation}


An element $f\in k(\wt{C})$ is in the support of $g_p\cdot  \de_{\OO_{\wt{C}}}$ if and only if for every $p$, $f\in \a_p\OO_{\wt{C},p}$.
This is equivalent to the condition that the rational section $s:=\phi_{gen} f$ of the line bundle $\pi^*L$ on $\wt{C}$ is in fact regular.
Also,
$$\vartheta_{\wt{C}/C}(L)=\prod_p \vartheta_p(\a_p), \ \ \sum_p [k(p):k]\cdot v_p(\a_p)=-\deg(L).$$

Thus, we can rewrite \eqref{gp-ext-action-formula} as
$$\th_{\LL}((g_p)\cdot \de_{\hat{\LL}})=\vartheta_{\wt{C}/C}(L)\cdot q^{-\deg(L)}\cdot \sum_{f: \phi_{gen}\cdot f\in H^0(\pi^*L\ot\LL)} \psi(\sum_{p\in C}\Res_p(e_p\phi_{gen}^2\cdot \Nm(f))).$$
Using Serre duality between $H^0(C,L^2\ot \om_C)$ and $H^1(C,L^{-2})$, which is induced by the pairing
$$\lan s, (e_p)\ran=\sum_{p\in C}\Res_p(e_p\cdot s),$$
where $e_p\in L^{-2}\ot \A_C$ and $s\in H^0(C,L^2\ot\om_C)$, we derive the following formula for $f_\LL$.

\begin{prop}\label{fin-field-bun-prop} 
For an $\SL_2$-bundle $V$ on $C$, given as an extension \eqref{V-L-ext}, one has
$$f_\LL(V)=\vartheta_{\wt{C}/C}(L)\cdot q^{-\deg(L)}\cdot\sum_{x\in H^0(\wt{C},\pi^* L\ot \wt{\LL})}\psi(q_{\LL,L}(x)),$$
where $q_{\LL,L}$ is a quadratic form on $H^0(\wt{C},\pi^* L\ot \wt{\LL})$ given by the composition
$$q_{\LL,L}:S^2H^0(\wt{C},\pi^* L\ot \wt{\LL})\rTo{\Nm} H^0(C,L^2\ot\om_C)\rTo{e} k,$$
where $e\in H^1(C,L^{-1})\simeq H^0(C,L^2\ot\om_C)^*$ is the class of the extension \eqref{V-L-ext}.
\end{prop}

\begin{remark} Using standard properties of the Gauss sums, 
it is possible to show directly that the right-hand side of the formula for $f_\LL(V)$ does not depend on a choice of a subbundle $L^{-1}\sub V$.
\end{remark}

\begin{remark}
An analog of the formula of Proposition \ref{fin-field-bun-prop} for symplectic bundles and the corresponding perverse sheaf on the moduli stack 
is discussed in \cite{Lysenko}.
\end{remark}

\section{Theta correspondence for a nilpotent extension}\label{nilp-theta-sec}

We still work with the group $G=\SL_2$.

\subsection{Local representations}\label{nilp-local-sec}

\subsubsection{Local setup and the first definition of local representations}\label{local-setup-sec}


Let $\ov{E}$ be a local field, $\OO_{\ov{E}}$ its ring of integers, $k$ the residue field. We assume that the characteristic of $k$ is $\neq 2$. Let
$$0\to N\OO_E\to \OO_E\to \OO_{\ov{E}}\to 0$$ 
a square zero extension (so $\OO_E$ is a commutative ring, $N\OO_E$ is a square zero ideal), 
such that $N\OO_E$ is a free $\OO_{\ov{E}}$-module of rank 1. Note that $\OO_E$ is a local Noetherian ring.

Note that in the case when $\ov{E}\simeq k(\!(t)\!)$ and $\OO_E$ is a $k$-algebra, by formal smoothness of $\OO_{\ov{E}}$ over $k$, 
we necessarily have $\OO_E\simeq k[\eps]/(\eps^2)\ot_k \OO_{\ov{E}}$.

We denote by $E$ the localization of $\OO_E$ with respect to all non-zero-divisors. Then $E$ is a square zero extension of $\ov{E}$ by an ideal $NE$,
which is a free $\ov{E}$-module of rank 1.

Let $\ov{K}$ be a quadratic extension of $\ov{E}$, i.e., either $\ov{K}/\ov{E}$ is a quadratic field extension (possibly ramified), or $\ov{K}=\ov{E}\oplus\ov{E}$. 
Let $\OO_{\ov{K}}\sub \ov{K}$ integers in $\ov{K}$. 
Assume we are given a flat $\OO_E$-algebra $\OO_K$ together with an isomorphism 
$$\OO_K\ot_{\OO_E}\OO_{\ov{E}}\simeq \OO_{\ov{K}}.$$
Note that this implies that $\OO_K$ is a free $\OO_E$-module of rank $2$.
We set $N\OO_K:=\OO_K\cdot N\OO_E\simeq \OO_{\ov{K}}\ot_{\OO_{\ov{E}}}N\OO_E$. This is a square zero ideal, and $\OO_K/N\OO_K\simeq \OO_{\ov{K}}$.
Finally, we set $K=\OO_K\ot_{\OO_E} E$, which is a square zero extension of $\ov{K}$ by an ideal $NK$ (which is a free $\ov{K}$-module of rank $1$).

Since $\OO_K$ is free of rank $2$ as $\OO_E$-module, 
we have the trace and norm maps $\Tr:\OO_K\to \OO_E$, $\Nm:\OO_K\to \OO_E$ and the corresponding maps $K\to E$.
Furthermore, the map $\iota(x)=\Tr(x)-x$ is an $\OO_E$-linear automorphism of $\OO_K$, such that $\Nm(x)=x\iota(x)$. 
\footnote{This is true for any extension of commutative rings $A\sub B$ such that $B$ is free of rank $2$ as an $A$-module, and can be deduced from
the identity $\det(X+Y)-\det(X)-\det(Y)=\tr(X)\tr(Y)-\tr(XY)$ that holds for $2\times 2$-matrices over $A$.}
We also extend $\iota$ to an $E$-linear automorphism of $K$.

We define a symmetric $E$-bilinear pairing $K\times K\to E$ by
\begin{equation}\label{Nm-pairing-eq}
\lan a,b\ran:=\Tr(a\iota(b))=\Nm(a+b)-\Nm(a)-\Nm(b).
\end{equation}
The pairing $\lan\cdot,\cdot\ran$ induces a nondegenerate $\ov{E}$-bilinear pairing
$$NK\times \ov{K}\to NE\simeq \ov{E}.$$
Indeed, let $\eps$ be a generator of $N\OO_E$ as $\OO_{\ov{E}}$-module. Then for $\ov{a},\ov{b}\in \ov{K}$,
$$\lan \ov{a}\eps,\ov{b}\ran=\Tr_{\ov{K}/\ov{E}}(\ov{a}\iota(\ov{b}))\eps,$$
and the form $\Tr_{\ov{K}/\ov{E}}(\ov{a}\iota(\ov{b}))$ is nondegenerate.

The norm $\Nm:K\to E$ is zero on $NK$, and satisfies 
$$\Nm(x+x_1)=\Nm(x)+\lan \ov{x},x_1\ran,$$
for any $x\in K$, $x_1\in NK$, and $\ov{x}=x\mod NK$.

Let $\psi_E:E\to \C^*$ be an additive character of $E$, such that $\psi_E|_{NE}$ is nontrivial.


For $\a\in E$, we set
$|\a|_{E}=q^{-2v_{\ov{E}}(\ov{\a})}$, where $\ov{\a}=\a \mod(NE)\in \ov{E}$, and $q$ is the number elements in the residue field of $\OO_{\ov{E}}$.
Below in the definition of the Fourier transform on $\SS(K)$ we normalize the measure on $K$ so that $\vol(\OO_{K})=1$.

We denote by $K^*_1\sub K^*$ the subgroup of elements $\a$ with $\Nm(\a)=1$, and we set
$$NK^*_1:=(1+NK)\cap K^*_1.$$
In other words, $NK^*_1$ consists of $1+c$, where $c\in NK$ is such that $\Tr(c)=0$.


\begin{prop}\label{main-local-rep-prop}
There is a unique representation of $\SL_2(E)$ on $\SS(K)$ given by the formulas
$$r\left(\begin{matrix} a & 0\\ 0 & a^{-1}\end{matrix}\right)\Phi(x)=|a|_{E} \Phi(a x),$$
$$r\left(\begin{matrix} 1 & z\\ 0 & 1\end{matrix}\right)\Phi(x)=\psi_E(z\Nm_{K/E}(x)) \Phi(x),$$
$$r\left(\begin{matrix} 0 & 1\\ -1 & 0\end{matrix}\right)\Phi(x)=\widehat{\Phi}(\iota(x))=\int_{y\in K} \psi_E(\lan x,y\ran)\Phi(y)dy.$$
There is a commuting action of $K^*_1$ on $\SS(K)$ given by $(\a\cdot f)(x)=f(\a x)$ for $\a\in K^*_1$.
\end{prop}

Uniqueness follows from the fact that the corresponding matrices generate the group $\SL_2(E)$.
Existence can be proved directly. Below we will use a partial Fourier transform to give a different proof. 
Note that the fact that the corresponding local norms are always zero on $NK$ leads to the absence of the gamma-factors similar to $\om(\a)$ in the case of the finite field
(see Sec.\ \ref{local-setup-fin-sec}).

\begin{remark}\label{Nm-rescaling-rem}
Let $L$ be a free $K$-module of rank $1$, equipped with an isomorphism of $E$-modules,
$\Nm(L)\rTo{\sim} E$ (here we use the push-forward of torsors with respect to the norm map $\Nm:K^*\to E^*$). Then we get a norm map $\Nm:L\to E$, and the induced $E$-valued
bilinear pairing $\lan\cdot,\cdot\ran$ on $L$ given by \eqref{Nm-pairing-eq}.
Now the formulas of Proposition \ref{main-local-rep-prop} define a representation of $\SL_2(E)$ on $\SS(L)$.
Indeed, if we choose a trivialization of $L$ as $K$-module, we will get an action of $\SL_2(E)$, given
by the same formulas but with $\Nm$ and $\lan\cdot,\cdot\ran$ rescaled by an element $a\in E^*$, which is equivalent
to replacing $\psi_E$ with $x\mapsto \psi_E(a\cdot x)$. 
\end{remark}

\subsubsection{Partial Fourier transform and coinvariants}

We will obtain a simpler description of the representation of Proposition \ref{main-local-rep-prop}, by performing a partial Fourier transform.


Let us consider the subspace
$$\wt{\FF}\sub\C_{lc}(K^2)$$
consisting of functions such that
\begin{equation}\label{F'-def-eq}
f(x_1+n_1,x_2+n_2)=\psi_E(-\lan n_1,x_2\ran)f(x_1,x_2),
\end{equation}
such that the support of $f$ in $\ov{K}^2$ is compact.

We have the partial Fourier transform
$$\wt{S}_1:\SS(K)\rTo{\sim} \wt{\FF},$$
$$\wt{S}_1f(y_1,y_2)=\int_{n\in NK}\psi_E(\lan n,y_2\ran)f(y_1+n)dn.$$

It is convenient to use an additional twist (to make the action of $\SL_2(E)$ simpler). Namely, let
us denote by 
$$\FF\sub\C_{lc}(K^2)$$
the space of functions such that
\begin{equation}\label{F-def-eq}
f(x_1+n_1,x_2+n_2)=\psi_E(\frac{1}{2}\lan x_1,n_2\ran-\frac{1}{2}\lan n_1,x_2\ran)f(x_1,x_2),
\end{equation}
and such that the support of $f$ in $\ov{K}^2$ is compact.

We have an isomorphism 
$$\wt{\FF}\rTo{\sim} \FF:f\mapsto \psi_E(\frac{1}{2}\lan x_1,x_2\ran)\cdot f,$$
and we denote by $S_1:\SS(K)\rTo{\sim} \FF$ the composition of $\wt{S}_1$ with this isomorphism.


\begin{prop}\label{part-F-prop}
(i) We have a representation $r_\FF$ of $\SL_2(E)$ on the space $\FF$ given by
\begin{equation}\label{repr-after-part-FT}
r_\FF(g)(\phi)(v)=\phi(vg),
\end{equation}
where we think of $v\in K^2$ as row vectors.
Furthermore, for $g\in \SL_2(E)$ and $f\in\SS(K)$, one has
\begin{equation}\label{S1-r-eq}
S_1r(g)f=r_\FF(g)S_1f.
\end{equation}
Thus, $S_1$ gives an isomorphism between the representations $r$ and $r_\FF$.

\noindent
(ii) For an element $\a\in K^*_1$, one has 
$$S_1(\a f)(x_1,x_2)=S_1f(\a x_1,\a x_2).$$

\noindent
(iii) For $v=(x_1,x_2)\in K^2$ and $A=\left(\begin{matrix} a & b \\ c & -a\end{matrix}\right)\in \fg(E)$, one has
$$\det(\iota(v),vA)=\Nm(x_1)b-\Nm(x_2)c-\lan x_1,x_2\ran a,$$
which lies in $E$.
 
\noindent
(iv) Let $\phi\in \FF$. For $A\in \fg(NE)$ one has
\begin{equation}\label{A-g-NE-action-eq}
r_\FF(1+A)(\phi)(v)=\psi_E(\det(\iota(v),vA))\cdot \phi(v).
\end{equation}
For $c\in NK$ with $\Tr(c)=0$, one has
$$\phi((1+c)v)=\psi_E(c\cdot \det(\iota(v),v))\cdot \phi(v).$$
\end{prop}

\begin{proof}
(i) For $v=(x_1,x_2), v'=(x'_1,x'_2)\in K^2$, set
$$B(v,v')=\frac{1}{2}(\lan x_1,x'_2\ran-\lan x'_1,x_2\ran).$$
The fact that \eqref{repr-after-part-FT} defines a representation follows easily from
$\SL_2(E)$-invariance of $B(v,v')$.

It is enough to check \eqref{S1-r-eq} for $g$ upper-triangular, diagonal and for 
$g=\left(\begin{matrix} 0 & 1\\ -1 & 0\end{matrix}\right)$.
This boils down to checking the following formulas:
$$\wt{S}_1(\psi_E(z\Nm(x))\Phi(x))(y_1,y_2)=\psi_E(z\Nm(y_1))\wt{S}_1(\Phi)(y_1,zy_1+z_2),$$
$$\wt{S}_1(|a|\Phi(a x))(y_1,y_2)=\wt{S}_1(\Phi)(ay_1,a^{-1}y_2),$$
$$\wt{S}_1(\hat{\Phi}(\iota(x)))(y_1,y_2)=\psi_E(-\lan y_1,y_2\ran)\wt{S}_1(-y_2,y_1).$$

\noindent
(ii), (iii) This is straightforward.

\noindent
(iv) It is easy to check that
$$2B(v,v')=\det(\iota(v),v')+\det(v,\iota(v'))=\Tr(\det(\iota(v),v')).$$
Since $\det(\iota(v),vA)\in E$ (as follows from (iii)), it follows that
for $A\in \fg(NE)$, one has
\begin{equation}\label{B-det-identity}
B(v,vA)=\frac{1}{2}\Tr(\det(\iota(v),vA))=\det(\iota(v),vA),
\end{equation}
which implies the claimed formula for the action of $1+A$.

SImilarly, for $c\in NK$ with $\Tr(c)=0$, one has
$$B(v,cv)=\frac{1}{2}\Tr(\det(\iota(v),cv))=\frac{1}{2}\Tr(c\det(\iota(v),v))=c\det(\iota(v),v),$$
which implies the formula for the action of $1+c$.
\end{proof}

Now, let us fix $c_0\in \ov{K}$ with $\Tr(c_0)=0$, and
consider the $\chi_{c_0}$-twisted coinvariants of our representation of $K^*_1\times \SL_2(E)$ with respect to the action of 
the subgroup $NK^*_1\sub K^*_1$, where for $1+c\in NK^*_1$ (with $\Tr(c)=0$), one has 
$$\chi_{c_0}(1+c)=\psi_E(cc_0).$$ 
In other words,
$$\FF_{c_0}:=\FF/(\phi((1+c)v)-\psi_E(cc_0)\phi(v)).$$

Let us consider the $\SL_2(\ov{E})$-invariant affine quadric $\ov{Q}(c_0)\sub \ov{K}^2$ given by 
\begin{equation}\label{Q-c0-eq}
\det(\iota(\ov{v}),\ov{v})=c_0.
\end{equation}

\begin{cor} There is a natural identification of $\FF_{c_0}$ with
space of locally constant functions on the set of $v\in K^2$ such that $\ov{v}\in \ov{Q}(c_0)$, satisfying \eqref{F-def-eq} and with compact support in $\ov{Q}(c_0)$,
where the action of $\SL_2(E)\times K^*_1$ is given by \eqref{repr-after-part-FT} and by $\a\phi(v)=\phi(\a v)$ for $\a\in K^*_1$.
\end{cor}


\subsubsection{$\FF_{c_0}$ as an induced representation}

Let us define a quadratic map
$$\eta:\ov{Q}(c_0)\to \fg(\ov{E})$$
by the relation
$$\tr(\eta(\ov{v})\cdot A)=\det(\iota(\ov{v}),\ov{v}A)$$
for any $A\in \fg(\ov{E})$. It follows from Proposition \ref{part-F-prop}(iii) that
\begin{equation}\label{eta-formula-eq}
\eta(x_1,x_2)=\left(\begin{matrix} -\lan x_1,x_2\ran/2 & -\Nm(x_2)\\ \Nm(x_1) & \lan x_1,x_2\ran/2\end{matrix}\right).
\end{equation}
Note that 
$$\eta(\ov{v}g)=g^{-1}\eta(\ov{v})g$$
for $g\in \SL_2(\ov{E})$.

Similarly, for $c\in K$ such that $\Tr(c)=0$, let us define the affine quadric  $Q(c)\sub K^2$ by the equation 
$$\det(\iota(v),v)=c,$$
and consider the quadratic map $\eta:Q(c)\to \fg(E)$ given by \eqref{eta-formula-eq} and satisfying
\begin{equation}\label{eta-main-identity}
\tr(\eta(v)\cdot A)=\det(\iota(v),vA)
\end{equation}
for $A\in \fg(E)$.

\begin{cor}\label{eta-action-cor}
For $\phi\in \FF$ and $A\in\fg(NE)$ one has
$$r_\FF((1+A))\phi(v)=\psi_E(\tr(\eta(v)A))\phi(v).$$
\end{cor}

\begin{proof} Combine \eqref{A-g-NE-action-eq} with \eqref{eta-main-identity}.
\end{proof}

\begin{lemma}\label{v-c-lem}
Suppose $c\in K^*$ is such that $\Tr(c)=0$
(resp., $c_0\in \ov{K}^*$ is such that $\Tr(c_0)=0$). 

\noindent
(i) Then $v_c:=(-c/2,1)\in Q(c)$ (resp., $v_{c_0}:=(-c_0/2,1)\in Q(c_0)$), and
the right action of $\SL_2(E)$ on $Q(c)$ (resp., of $\SL_2(\ov{E})$ on $Q(c_0)$) is simply transitive.
In particular, $\eta(Q(c))$ is a single $\SL_2(E)$-orbit (resp., $\eta(Q(c_0))$ is a single $\SL_2(\ov{E})$-orbit),
namely, the orbit of
$$\eta_{c}:=\eta(v_c)=\left(\begin{matrix} 0 & -1\\ -c^2/4 & 0\end{matrix}\right)$$
(resp., of $\eta_{c_0}=\eta(v_{c_0})$).

\noindent
(ii) The fibers of $\eta$ in $Q(c)$ (resp., $Q(c_0)$) are exactly the orbits of the free $K^*_1$-action on $Q(c)$ (resp., free $\ov{K}^*_1$-action on $Q(c_0)$) given by $a(x,y)=(ax,ay)$
(where $\Nm(a)=1$). Hence, we get an identification of $K^*_1$ (resp., $\ov{K}^*_1$) with the stabilizer subgroup of $\eta_{c}$ in $\SL_2(E)$
(resp., $\eta_{c_0}$ in $\SL_2(\ov{E})$), 
$$a\mapsto i_c(a)=\left(\begin{matrix} \frac{a+\iota(a)}{2} & \frac{\iota(a)-a}{c} \\ \frac{(\iota(a)-a)c}{4} & \frac{a+\iota(a)}{2} \end{matrix}\right),$$
so that $av_c=v_c\cdot i_c(a)$ (resp., $av_{c_0}=v_{c_0}\cdot i(a)$).

\noindent
(iii) The map $\phi\mapsto (g\mapsto r(g)(\phi)(v_c)=\phi(v_c g))$ gives an isomorphism of $\SL_2(E)$-representations from $\FF_{c_0}$ to the representation, compactly induced  
from the character $A\mapsto \psi_E(\tr(\eta_{c_0}A))$ of $1+\fg(NE)$.
\end{lemma}

\begin{proof} 
(i) We will give a proof for $Q(c)$, where $c\in K^*$ with $\Tr(c)=0$. The case of $Q(c_0)$ is completely analogous.
Consider the $2\times 2$ matrix with coefficients in $K$,
$$X=\left(\begin{matrix} \iota(x) & \iota(y) \\ x & y\end{matrix}\right).$$
Then the condition $(x,y)\in Q(c)$ is equivalent to $\det(X)=c$. Thus,
$Q(c)$ is in bijection with the set of $X$ such that $\iota(X)=sX$ and $\det(X)=c$,
where $s$ is the nontrivial $2\times 2$ permutation matrix.
Furthermore, the right action of $\SL_2(E)$ on $Q(c)$ corresponds to the action on the set of such $X$ by
right multiplication.

Let $X_0$ denote the matrix corresponding to $(x,y)=(-c/2,1)$. Then
there is a unique element $g\in \SL_2(K)$ such that $X=X_0g$.
Furthermore,
$$\iota(X)=\iota(X_0)\iota(g)=sX_0\iota(g)$$
is equal to $sX=sX_0g$, hence $\iota(g)=g$, i.e., $g\in \SL_2(E)$.

\noindent
(ii) Since the action of $K^*_1$ commutes with the right action of $\SL_2(E)$,
it is enough to prove this for one fiber, say, the fiber over $\eta_{c}$. This fiber
consists of solutions of the equations
$$x\iota(y)=-c/2, \ \ y\iota(y)=1,$$
which are all of the form
$(x,y)=(-yc/2,y)$, with $y\iota(y)=1$.

\noindent
(iii)  For $A\in \fg(NE)$, we have by Corollary \ref{eta-action-cor},
$$r((1+A)g)(\phi)(v_c)=\psi_E(\tr(\eta_{c_0}A))\cdot r(g)(\phi)(v_c).$$
Thus, we can identify the representation of $\SL_2(E)$ compactly induced from the character $\psi_E(\tr(\eta_{c_0}A))$ of
$1+\fg(NE)$, with the space of functions $f$ on $Q(c)$ satisfying
$$f(v(1+A))=\psi_E(\tr(\eta_{c_0}A))\cdot f(v)$$ 
and with compact support modulo $NE$ (i.e., in $\ov{Q}(c_0)$).
Indeed, it is enough to check that this is equivalent to the condition
$f(v+n)=\psi_E(B(v,n))f(v)$ whenever both $v$ and $v+n$ belong to $Q(c)$.
Lemma \ref{v-c-lem}(i) easily implies that $v$ and $v+n$ both belong to $Q(c)$ if and only if $n=Av$ for some $A\in \fg(NE)$.
Now the equivalence follows from the identity \eqref{B-det-identity}.
\end{proof}


%

\subsection{Spherical vectors and Hecke operators}\label{modt2-sph-sec}

We still work with local representations defined above. Let $t\in \OO_E$ be a lifting of a uniformizer in $\OO_{\ov{E}}$.
We always assume that $\psi_E$ is trivial on $\OO_E$ and is non-trivial on $t^{-1}N\OO_E$. 
We are interested in the space of $G(\OO_E)$-invariant vectors (i.e., spherical vectors) in $\FF_{c_0}$.

\begin{lemma}\label{c0-int-lem}
Assume $\ov{K}/\ov{E}$ is non-split, i.e., $\ov{K}$ is a field. Then 
$\FF_{c_0}^{G(\OO_E)}=0$ unless $c_0\in \OO_{\ov{K}}$.
\end{lemma}

\begin{proof}
From $1+\fg(N\OO_E)$-invariance, we immediately see that a $G(\OO_E)$-invariant vector is supported on the set of $v$ such that 
$\ov{v}=(x_1,x_2)\in Q(c_0)$ satisfies $\eta(x_1,x_2)\in \fg(\OO_{\ov{E}})$. But in the case when $\ov{K}$ is a field, the condition $\Nm(x_i)\in \OO_{\ov{E}}$
implies that $x_i\in \OO_{\ov{K}}$ (for $i=1,2$). Hence, we necessarily should have $c_0=\det(\iota(\ov{v}),\ov{v})\in \OO_{\ov{E}}$.
\end{proof} 

Below we will consider the Hecke operator $T_1(t)$ corresponding to the double coset of $\diag(t,t^{-1})$, where $t\in E$ is a lifting of a uniformizer in $\ov{E}$.

\subsubsection{Non-split unramified case}

Assume the extension $\ov{K}/\ov{E}$ is non-split (i.e., $\ov{K}$ is a field) and unramified.

We are going to study the space of spherical vectors in the $\SL_2(E)$-representation $\FF_{c_0}$, where $c_0\in \ov{K}$ is an element with
$\Tr(c_0)=0$. 

Assume $c_0\in \OO_{\ov{K}}$, and let $\FF_{c_0,\OO}\sub \FF_{c_0}$ denote the subspace of $\phi$ such that $\phi(v)$=0 for
$\ov{v}\not\in Q(c_0)(\OO_{\ov{K}})$. Then every $\phi\in \FF_{c_0,\OO}$ is determined by its restriction to points with coordinates in $\OO_K$.
Due to relation \eqref{F-def-eq}, the latter restriction descends to a function on $Q(c_0)(\OO_{\ov{K}})$. Thus, we get an identification
$$\FF_{c_0,\OO}\simeq \SS(Q(c_0)(\OO_{\ov{K}})).$$


\begin{prop}\label{Hecke-non-split-0-prop}
(i) Assume that $v(c_0)=0$. Then the space $\FF_{c_0}^{G(\OO_E)}$ is $1$-dimensional and is spanned by the function
$$\de_{\OO}:=\de_{Q(c_0)(\OO_{\ov{K}})}\in \FF_{c_0,\OO}\sub\FF_{c_0}.$$ 

\noindent
(ii)  One has 
$$T_1(t)\de_{\OO}=0.$$
\end{prop}

\begin{proof} (i) Using the invariance under $\fg(N\OO_E)$, we see that any element of $\FF_{c_0}^{G(\OO_E)}$ is supported
on the set of $v\in K^2$ such that $\eta(\ov{v})\in \fg(\OO_{\ov{E}})$ (and $\ov{v}\in Q(c_0)$). It follows that $\ov{v}=(x,y)$ satisfies $\Nm(x)\in \OO_{\ov{E}}$,
$\Nm(y)\in \OO_{\ov{E}}$, hence $\ov{v}\in Q(c_0)(\OO_{\ov{K}})$. Thus, to prove the first assertion, we have to check that $\SL_2(\OO_{\ov{E}})$ acts
transitively on $Q(c_0)(\OO_{\ov{K}})$. This is proved similarly to Lemma \ref{v-c-lem}(i). Namely, points in $Q(c_0)(\OO_{\ov{K}})$ correspond to
$2\times 2$ matrices $X$ with entries in $\OO_{\ov{K}}$ such that $\iota(X)=sX$ and $\det(X)=c_0$. Since $c_0$ is a unit, such $X$ is necessarily in
$\GL_2(\OO_{\ov{E}})$. Let $X_0$ be the matrix corresponding to $(-c_0/2,1)$. Then $X=X_0g$, where $g\in \GL_2(\OO_{\ov{E}})$ and $\det(g)=1$,
hence $g\in \SL_2(\OO_{\ov{E}})$.

(ii) The function $\diag(t,t^{-1})\de_{\OO}$ is supported on the set of $v\in K^2$ such that $\ov{v}=(x,y)\in \ov{K}^2$ satisfies
$$x\in t^{-1}\OO_{\ov{E}}, \ \ y\in t\OO_{\ov{E}}, \ \ \iota(x)y-\iota(y)x=c_0.$$
Averaging with respect to the action of $\fg(N\OO_E)$ results in shrinking the support to the locus where $x\in \OO_{\ov{E}}$.
But this becomes incompatible with the conditions $y\in (t)$ and $\iota(x)y-\iota(y)x=c_0$, since it would imply that $c_0\in (t)$.
Hence, the averaging results in the zero function.
\end{proof}

Let $k\sub \ell$ be the quadratic extension of residue fields associated with $\ov{E}\sub \ov{E}$, and let $q=|k|$.

\begin{lemma}\label{Q-integer-points-lem}
For any units $x\in \OO_{\ov{K}}^*$ and $u\in \OO_{\ov{K}}^*$ such that $\iota(u)=-u$, there exists $y\in \OO_{\ov{K}}$ such that $\iota(x)y-\iota(y)x=u$.
\end{lemma}

\begin{proof}
For any unit $u\in \OO_{\ov{K}}^*$ such that $\iota(u)=-u$, one has $\OO_{\ov{K}}=\OO_{\ov{E}}+u\OO_{\ov{E}}$ (say, by Nakayama lemma).
Hence, we can write $x=-\frac{u}{2}a+c$ for $a,c\in \OO_{\ov{E}}$. Since $x$ is a unit, either $a$ or $c$ is a unit. Hence, we can find $b,c\in \OO_{\ov{E}}$,
such that $g=\left(\begin{matrix} a & b \\ c & d\end{matrix}\right)$ is in $\SL_2(\OO_{\ov{E}})$. Then $v_ug=(x,y)$ satisfies the required property.
\end{proof}

\begin{prop}\label{Hecke-non-split-1-prop}
Assume that $v(c_0)=1$. 

\noindent
(i) The $\SL_2(\OO_{\ov{E}})$-orbits on $Q(c_0)(\OO_{\ov{K}})$ are numbered by $\la\in \ell^*$ such that $\la^{q+1}=1$. The orbit $O(c_0,\la)$ corresponding to $\la$
consists of $v=(x,y)\in Q(c_0)(\OO_{\ov{K}})$ such that $i(x)\equiv \la x \mod(t)$, $i(y)\equiv \la y \mod(t)$.

\noindent
(ii) The space $\FF_{c_0}^{G(\OO_E)}$ is $(q+1)$-dimensional, with the basis 
$$\de_{O(c_0,\la)}\in \FF_{c_0,\OO}\sub\FF_{c_0},$$
numbered by $\la\in \ell^*$ with $\la^{q+1}=1$.
One has
$$T_1(t)\de_{O(c_0,\la)}=q^2\sum_{\mu\neq \la}\de_{O(c_0,\mu)}.$$
In particular, 
$$T_1(t)\de_{\OO}=q^3\de_{\OO}.$$
\end{prop}

\begin{proof}
(i) As before, we view $Q(c_0)(\OO_{\ov{K}})$ as $2\times 2$-matrices $X$ with entries in $\OO_{\ov{K}}$ such that $\iota(X)=sX$ and $\det(X)=c_0$.
Let $L(X)\sub \OO_{\ov{K}}^2$ denote the lattice given by the span of the columns of $X$. Then $L(X)\supset t\OO_{\ov{K}}^2$, 
$\OO_{\ov{K}}^2/L(X)\simeq \OO_{\ov{K}}/(t)$, and $L(Xg)=L(X)$ for $g\in\SL_2(\OO_{\ov{E}})$. Furthermore, $L(X)$ is stable under the involution $(x,y)\mapsto (\iota(y),\iota(x))$. Conversely, if
$L(X)=L(X')$ then $X'=Xg$ where $g\in \SL_2(\ov{E})$. Thus, the orbits are in bijection with a subset of lines in $\ell^2$ preserved by the involution 
$(x_1,x_2)\mapsto (x_2^q,x_1^q)$.
These are exactly lines of the form $x_1=\la x_2$ where $\la^{q+1}=1$. It remains to check that each such line can be realized. In fact, we claim there exists
a representative $(x,y)$ with $y\in t\OO_{\ov{K}}$.
We look for $y$ in the form $y=ty'$. Then we need $\iota(x)y'-\iota(y')x=c_0/t$ (where $c_0/t$ is a $\iota$-antiinvariant unit). By Lemma \ref{Q-integer-points-lem}
such $y'$ exists.

\noindent
(ii) The first assertion follows from (i). 

The function $\diag(t,t^{-1})\de_{O(c_0,\la)}$ is supported on the set of $v\in K^2$ such $\ov{v}=(x,y)\in \ov{K}^2$ satisfies 
$$x\in t^{-1}\OO_{\ov{E}}, \ \ y\in t\OO_{\ov{E}}, \ \ \iota(x)y-\iota(y)x=c_0, \ \ \iota(x)\equiv \la x \mod \OO_{\ov{E}}, \ \ 
\iota(y)\equiv \la y \mod (t^2).$$
Averaging with respect to the action of $\fg(N\OO_E)$ results in shrinking the support to the locus where in addition $x\in \OO_{\ov{E}}$.
which makes the congruence condition on $x$ vacuous. Let $S(c_0,\la)\sub Q(c_0)(\OO_{\ov{E}})$ denote the subset consisting of $(x,y)$ such that
$y\in (t)$ and $\iota(y)\equiv \la y \mod(t^2)$. Let $A$ denote the averaging operator with respect to $\SL_2(\OO_{\ov{E}})$ on $\FF_{c_0,\OO}$. Then we have
$$T_1(t)\de_{O(c_0,\la)}=(q^4+q^3)A\de_{S(c_0,\la)}=(q^4+q^3)\sum_{\mu} \vol(S(c_0,\la)\cap O(c_0,\mu))\cdot \de_{O(c_0,\mu)},$$
where the volume is taken with respect to the natural measure on $O(c_0,\mu)\simeq \SL_2(\OO_{\ov{E}})$.
Note that if $\iota(x)y-\iota(y)x=c_0$ and $y\in (t)$, $\iota(y)\equiv \la y \mod(t^2)$, then
$$(\iota(x)-\la x)y\equiv c_0 \mod (t^2),$$
which implies that $\iota(x)\not\equiv \la x \mod(t)$ (since $v(c_0)=1$). Therefore, we have $S(c_0,\la)\cap O(c_0,\la)=\emptyset$.

Next, let us calculate $\vol(S(c_0,\la)\cap O(c_0,\mu))$ for $\mu\neq \la$. Let $v_\la=(x_\mu,y_\mu)$ be a representative of $O(c_0,\mu)$ with $y_\mu\in (t)$
(we checked in part (i) that such a representative exists). For an element $g=\left(\begin{matrix} a & b \\ c & d\end{matrix}\right)\in\SL_2(\OO_{\ov{E}})$,
we have $v_{\la}g\in S(c_0,\la)$ if and only if 
$$bx_\mu+dy_\mu\equiv 0 \mod (t), \ \ b[\iota(x_\mu)-\la x_\mu]+d[\iota(y_\mu)-\la y_\mu]\equiv 0\mod (t^2).$$
Since $y_\mu\in (t)$, the first condition means that $b\in (t)$. Since $\la\neq \mu$, we see that $\iota(x_\mu)-\la x_\mu$ is invertible.
Hence, the second condition implies the first and it fixes $[b:d]$ as a point in the projective line $\P^1(\OO_{\ov{E}}/(t^2))$.
Therefore, we get
$$\vol(S(c_0,\la)\cap O(c_0,\mu))=\frac{1}{|\P^1(\OO_{\ov{E}}/(t^2))|}=\frac{1}{q^2+q}.$$
\end{proof}

\subsubsection{Non-split ramified case}

Next, we assume that the extension $\ov{K}/\ov{E}$ is ramified. In this case the residue field extension is trivial, so 
for any $x\in \OO_{\ov{K}}$ one has $\iota(x)\equiv x \mod \fm_{\ov{K}}$. Hence, if $c_0\in \OO_{\ov{K}}$ satisfies $\iota(c_0)=-c_0$,
it necessarily lies in the maximal ideal. 


\begin{prop}\label{Hecke-ram-0-prop} 
Assume $v_{\ov{E}}(\Nm(c_0))=1$, or equivalently, $v_{\ov{K}}(c_0)=1$.  

\noindent
(i) The $\SL_2(\OO_{\ov{E}})$-action on $Q(c_0)(\OO_{\ov{K}})$ is transitive.

\noindent
(ii) The space $\FF_{c_0}^{G(\OO_E)}$ is $1$-dimensional and is spanned by the function 
$\de_{\OO}=\de_{Q(c_0)(\OO_{\ov{K}})}\in \FF_{c_0,\OO}\sub \FF_{c_0}$.
One has $T_1(t)\de_{\OO}=0$.
\end{prop}

\begin{proof}
(i) As before, we view $Q(c_0)(\OO_{\ov{K}})$ as $2\times 2$-matrices $X$ with entries in $\OO_{\ov{K}}$ such that $\iota(X)=sX$ and $\det(X)=c_0$.
Let $L(X)\sub \OO_{\ov{K}}^2$ denote the span of the columns of $X$. Let $z$ denote the uniformizer of $\ov{K}$. Then $L(X)\supset z\OO_{\ov{K}}^2$
and $\OO_{\ov{K}}^2/L(X)\simeq \OO_{\ov{K}}/(z)$. It is enough to check that $L(X)$ does not depend on $X$, or equivalently, the line
$$\ell(X):=L(X)/z\OO_{\ov{K}}^2\sub (\OO_{\ov{K}}/(z))^2=k^2$$
does not depend on $X$. But we know that $\iota(x)\equiv x \mod (z)$ for any $x\in \OO_{\ov{K}}$. This implies that $\ell(X)$ coincides with the line spanned by 
$\left[\begin{matrix} 1 \\ 1\end{matrix}\right]\in k^2$.

\noindent
(ii) The first assertion follows from (i). The second is proved in exactly the same way as in Proposition \ref{Hecke-non-split-0-prop}(ii), using the fact that $v_{\ov{K}}(t)=2$.
\end{proof}

\subsubsection{Split case}

In the case $\ov{K}=\ov{E}^2$,
we can identify $Q(c_0)$, for $c_0=(-b_0,b_0)\in \ov{K}=\ov{E}^2$, with the set $M(b_0)$ of pairs of vectors $x,y\in \ov{E}^2$ such that $\det(x,y)=b_0$
(where we think of $x$ and $y$ as column vectors).
In other words, $M(b_0)$ is just the set of $2\times 2$-matrices with determinant $b_0$.

Recall that we have an action of $\ov{K}^*_1\simeq \ov{E}^*$ on $\FF_{c_0}$. In particular, we will use the action of the elements $(t^n,t^{-n})\in \ov{K}^*_1$. 

Let $X(b_0)\sub M(b_0)(\ov{E})$ denote the subset consisting of $(x,y)$ such that $\eta(x,y)\in \fg(\OO_{\ov{E}})$, i.e.,
$x_1x_2\in \OO_{\ov{E}}$, $y_1y_2\in \OO_{\ov{E}}$ and $x_1y_2+x_2y_1\in \OO_{\ov{E}}$.

\begin{lemma}\label{X-b0-lem} 
Assume $b_0\neq 0$. Then 
$$X(b_0)=\cup_{n\in\Z} (t^n,t^{-n})M(b_0)(\OO_{\ov{E}}).$$
In particular, if $b_0\not\in \OO_{\ov{E}}$ then $X(b_0)=\emptyset$.
\end{lemma}

\begin{proof}
Let $(x_1,x_2,y_1,y_2)\in X(b_0)$. 
Since $b_0\neq 0$, we have either $x_1\neq 0$ or $y_1\neq 0$.
Without loss of generality we can assume that $x_1\neq 0$.
Applying some $(t^n,t^{-n})$ we can assume that $x_1\in \OO^*_{\ov{E}}$. Then the condition $x_1x_2\in \OO_{\ov{E}}$ implies that $x_2\in \OO_{\ov{E}}$. 
If $y_1\in \OO_{\ov{E}}$ then $x_2y_1\in \OO_{\ov{E}}$, so we get that $x_1y_2\in \OO_{\ov{E}}$, which implies that $y_2\in \OO_{\ov{E}}$, and we are done.
Otherwise, let $n>0$ be minimal such that $t^ny_1\in \OO_{\ov{E}}$. Then from the condition $y_1y_2\in \OO_{\ov{E}}$ we get $y_2\in t^n\OO_{\ov{E}}$. Hence, 
$x_1y_2\in \OO_{\ov{E}}$, so $x_2y_1\in \OO_{\ov{E}}$, which implies that $x_2\in t^n\OO_{\ov{E}}$. But then
$(t^n,t^{-n})(x,y)=(t^nx_1,t^{-n}x_2,t^ny_1,t^{-n}y_2)$ has integer coefficients.
\end{proof}

\begin{lemma}\label{c0-int-lem-bis}
For $\ov{K}/\ov{E}$ split, we have $\FF_{c_0}^{G(\OO_E)}=0$ unless $c_0\in \OO_{\ov{E}}$.
\end{lemma}

\begin{proof} This is proved as in Lemma \ref{c0-int-lem-bis}, using the fact that $\eta(x,y)\in \fg(\OO_{\ov{E}})$ implies that $\det(x,y)\in \OO_{\ov{E}}$,
by Lemma \ref{X-b0-lem}.
\end{proof}

As before, assuming that $b_0\in \OO_{\ov{E}}$, we denote by $\FF_{c_0,\OO}\sub \FF_{c_0}$ the subspace of $\phi$ such that $\phi(x,y)=0$ for
$(\ov{x},\ov{y})\not\in\OO_{\ov{E}}^4$, and use the natural identification of $\FF_{c_0,\OO}$ with $\SS(M(b_0)(\OO_{\ov{E}}))$.

\begin{prop}\label{Hecke-split-0-prop}
Assume $v(b_0)=0$.
Consider the function $\de_{\OO}=\de_{M(b_0)(\OO_{\ov{E}})}\in \FF_{c_0,\OO}\sub \FF_{c_0}$.
Then the space $\FF_{c_0}^{G(\OO_E)}$ has as a basis $((t^n,t^{-n})\de_{\OO})$, where $n\in\Z$.
One has 
$$T_1(t)\de_{\OO}=q^2[(t,t^{-1})\de_{\OO}+(t^{-1},t)\de_{\OO}].
$$
\end{prop}

\begin{proof} First, we observe that any matrix in $M(b_0)(\OO_{\ov{E}})$ is in $\GL_2(\OO_{\ov{E}})$ (since its determinant is invertible). 
Hence, the action of $\SL_2(\OO_{\ov{E}})$ on $M(b_0)(\OO_{\ov{E}})$ is transitive. 
Therefore, any $\SL_2(\OO_{\ov{E}})$-invariant function in $\FF_{c_0,\OO}$, supported on $M(b_0)(\OO_{\ov{E}})$ is proportional to $\de_\OO$.
Since any $\phi\in \FF_{c_0}^{G(\OO_E)}$ is supported on the set of $v$ such that $\ov{v}$ is in $X(b_0)$, the first assertion follows from Lemma \ref{X-b0-lem}.

The function $\diag(t,t^{-1})\de_{\OO}$ is supported on the set of $v\in K^2$ such $\ov{v}=(x_1,x_2,y_1,y_2)\in \ov{E}^2$ satisfies
$$x_i\in t^{-1}\OO_{\ov{E}}, \ \ y_i\in t\OO_{\ov{E}}, \ \ x_1y_2-x_2y_1=b_0.$$
Averaging with respect to the action of $\fg(N\OO_E)$ results in shrinking the support to the locus where $x_1x_2\in \OO_{\ov{E}}$.
Note that we cannot have $x_1,x_2\in \OO_{\ov{E}}$ since this would imply $b_0\in t\OO_{\ov{E}}$.
Hence, our locus is the union of two subsets: 
$$S_1: (x,y)\in M(b_0): x_1\in t^{-1}\OO^*_{\ov{E}}, \ \ x_2\in t\OO_{\ov{E}}, \ \ y_i\in t\OO_{\ov{E}}$$
$$S_2: (x,y)\in M(b_0): x_2\in t^{-1}\OO^*_{\ov{E}}, \ \ x_1\in t\OO_{\ov{E}}, \ \ y_i\in t\OO_{\ov{E}}.$$
We have $(t,t^{-1})S_1\in M(b_0)(\OO_{\ov{E}})$, $(t,t^{-1})S_2\in M(b_0)(\OO_{\ov{E}})$. Hence,
$$T_1(t)\de_{\OO}=(q^4+q^3)\cdot[\vol(S_1)(t^{-1},t)\de_{\OO}+\vol(S_2)(t,t^{-1})\de_\OO].$$
Finally, we have 
$$\vol(S_1)=\vol(S_2)=\vol\{g=\left(\begin{matrix} a & b \\ c & d\end{matrix}\right)\in\SL_2(\OO_{\ov{E}}):\ b\in (t^2)\}=\frac{1}{q^2+q}.$$
\end{proof}

Given a character $\chi:K^*_1\to \C^*$ such that $\chi|_{NK^*_1}(1+c)=\psi_E(cc_0)$, 
let $\FF_\chi$ denote the corresponding space of twisted coinvariants of $\FF$ (the maximal quotient on which $K^*_1$ acts through $\chi$).
Note that $\FF_\chi$ is the quotient of $\FF_{c_0}$.

\begin{cor}\label{Hecke-split-0-cor}
Assume $v(b_0)=0$.
Then the space $\FF_\chi^{G(\OO_E)}$ is spanned by the image of $\de_{\OO}$ and
$$T_1(t)\de_{\OO}\equiv q^2(\chi(t,t^{-1})+\chi(t,t^{-1})^{-1})\de_{\OO}$$
in $\FF_{\chi}^{G(\OO_E)}$.
\end{cor}

\begin{prop}\label{Hecke-split-1-prop}
Assume that $v(b_0)=1$. 

\noindent
(i) The $\SL_2(\OO_{\ov{E}})$-orbits on $X(b_0)$ are in bijection with $\P^1(k)\times \Z$. Namely, for $\la\in \P^1(k)\setminus \{\infty\}$ we have 
the orbit $O(b_0,\la)\sub M(b_0)(\OO_{\ov{E}})\sub X(b_0)$ consisting
of integer $(x_1,x_2,y_1,y_2)$ such that $x_1\equiv \la x_2 (t)$, $y_1\equiv \la y_2$, while the orbit $O(b_0,\infty)$ is given by $x_2\equiv y_2\equiv 0 (t)$ for $i=1,2$.
All the orbits in $X(b_0)$ are obtained by additionally
applying the action of $(t^n,t^{-n})$ for $n\in\Z$.

\noindent
(ii) The space $\FF_{c_0}^{G(\OO_E)}$ has as a basis $((t^n,t^{-n})\de_{O(b_0,\la)})$, where $n\in \Z$, $\la\in \P^1(k)$. 
One has
$$T_1(t)\de_{O(b_0,\la)}=
q^2\cdot\sum_{\mu\in \P^1(k)\setminus \la}\de_{O(b_0,\mu)},  \ \la\neq 0,\infty,
$$
$$T_1(t)\de_{O(b_0,0)}=
q^2\cdot [\sum_{\mu\in \P^1(k)\setminus 0}\de_{O(b_0,\mu)}+\sum_{\mu\in\P^1(k)\setminus \infty}(t,t^{-1})\de_{O(b_0,\mu)}],
$$
$$T_1(t)\de_{O(b_0,\infty)}=
q^2\cdot [\sum_{\mu\in \P^1(k)\setminus \infty}\de_{O(b_0,\mu)}+
\sum_{\mu\in\P^1(k)\setminus 0}(t^{-1},t)\de_{O(b_0,\mu)}].
$$
\end{prop}

\begin{proof}
(i) For $v=(x,y)\in M(b_0)(\OO_{\ov{E}})$, the column vectors $x\mod (t)$ and $y\mod (t)$ are proportional and not both zero.
Hence, they span a line $L=L_\la\sub k^2$. This line determines the sublattice in $\OO_{\ov{E}}^2$ spanned by $x$ and $y$. Hence,
$\SL_2(\OO_{\ov{E}})$ acts transitively on $O(b_0,\la)$.

Each orbit $O(b_0,\la)$ is nonempty since we have the following representatives $v_\la\in O(b_0,\la)$:
$$v_{\la}:=\left(\begin{matrix} \wt{\la} & -b_0\\ 1 & 0\end{matrix}\right), \ \la\neq\infty,$$
$$v_\infty:=\left(\begin{matrix} 1 & 0\\ 0 & b_0\end{matrix}\right),$$
where $\wt{\la}\in \OO_{\ov{E}}$ is a representative of $\la$.


Now the assertion follows from Lemma \ref{X-b0-lem}.

\noindent
(ii) The first assertion follows from (i). 

Assume first that $\la\in k$. Then $T_1(t)\de_{O(b_0,\la)}$ is supported on the set of $v\in M(b_0)$ such that $\ov{v}=(x_1,x_2,y_1,y_2)$ satisfies
$$x_1y_2-x_2y_1=b_0t, \ \ x_i\in t^{-1}\OO_{\ov{E}}, \ \ y_i\in t\OO_{\ov{E}}, \ \ x_1x_2\in \OO_{\ov{E}}, \ \ x_1\equiv \la x_2 (\OO), \ \ y_1\equiv \la y_2 (t^2\OO).$$
This set is the disjoint union of two subsets $S_0(\la)$ and $S_1(\la)$, where for $S_0(\la)$ (resp., $S_1(\la)$) we additionally have $x_1,x_2\in \OO_{\ov{E}}$
(resp., $x_2\not\in \OO_{\ov{E}}$, $x_1\in t\OO_{\ov{E}}$).

For $\mu\neq \infty$, the intersection $S_0(\la)\cap O(b_0,\mu)$ consists of $(x,y)$ such that 
$$x_1\equiv \mu x_2 (t), \ \ y_1\equiv y_2\equiv 0 (t), \ \ y_1\equiv \la y_2 (t^2).$$
In other words, $v_\mu\cdot \left(\begin{matrix} a & b \\ c & d\end{matrix}\right)\in S_0(\la)\cap O(b_0,\mu)$ if and only if $b\equiv 0 (t)$ and
$$(\mu-\la)b\equiv b_0d (t^2).$$
Since $d\in \OO^*_{\ov{E}}$, this is possible only if $\mu\neq \la$, and the volume of this intersection is $(q^2+q)^{-1}$ (since $|\P^1(\OO_{\ov{E}}/(t^2))|=q^2+q$).
Similarly, the intersection $S_0(\la)\cap O(b_0,\infty)$ consists of $(x,y)$ such that 
$$x_2\equiv 0, \ \ y_1\equiv y_2\equiv 0 (t), \ \ y_1\equiv \la y_2 (t^2).$$
Hence, $v_\infty\cdot \left(\begin{matrix} a & b \\ c & d\end{matrix}\right)\in S_0(\la)\cap O(b_0,\infty)$ if and only if $b\equiv 0 (t)$ and
$$b\equiv \la b_0 d (t^2).$$
This intersection again has volume $(q^2+q)^{-1}$.


Next, $(t^{-1},t)S_1(\la)$ is nonempty only for $\la=0$ and consists of $(x',y')\in M(b_0)(\OO_{\ov{E}})$ such that 
$$y'_1\equiv 0 (t), \ \ y'_2\equiv 0 (t^2).$$
Hence, for $\mu\neq \infty$, one has $v_\mu\cdot \left(\begin{matrix} a & b \\ c & d\end{matrix}\right)\in (t^{-1},t)S_1(0)\cap O(b_0,\mu)$ if and only if $b\equiv 0 (t^2)$,
while the intersection $(t^{-1},t)S_1(0)\cap O(b_0,\infty)$ is empty.
These calculations give the formula for $T_1(t)\de_{O(b_0,\la)}$ for $\la\neq \infty$.

Next, we consider the case $\la=\infty$. The function $T_1(t)\de_{O(b_0,\infty)}$ is supported on the set of $v\in M(b_0)$ such that $\ov{v}=(x,y)$ satisfies
$$x_1y_2-x_2y_1=b_0t, \ \ x_1\in t^{-1}\OO_{\ov{E}}, \ \ x_2\in \OO,  \ \ x_1x_2\in \OO_{\ov{E}}, \ \ y_1\in t\OO_{\ov{E}}, \ \ y_2\in t^2\OO_{\ov{E}}.$$
This set is the disjoint union of two subsets $S_0$ and $S_1$, where for $S_0$ (resp., $S_1$) we additionally have $x_1\in \OO_{\ov{E}}$
(resp., $x_1\not\in \OO_{\ov{E}}$, $x_2\in t\OO_{\ov{E}}$).

For $\mu\neq\infty$, we have
 $v_\mu\cdot \left(\begin{matrix} a & b \\ c & d\end{matrix}\right)\in S_0\cap O(b_0,\mu)$ if and only if $b\equiv 0 (t^2)$.
On the other hand, the intersection $S_0\cap O(b_0,\infty)$ is empty.

Finally, $(t,t^{-1})S_1$ consists of $(x',y')\in M(b_0)(\OO_{\ov{E}})$ such that
$$y'_1\equiv 0 (t^2), \ \ y'_2\equiv 0 (t).$$ 
Hence, for $\mu\neq \infty$, one has $v_\mu\cdot \left(\begin{matrix} a & b \\ c & d\end{matrix}\right)\in (t,t^{-1})S_1\cap O(b_0,\mu)$ if and only if
$b\equiv 0(t)$ and $\mu b\equiv b_0d (t^2)$. This is possible only for $\mu\neq 0$ and the corresponding subset has measure $(q^2+q)^{-1}$.
On the other hand, $v_\infty\cdot \left(\begin{matrix} a & b \\ c & d\end{matrix}\right)\in S_1\cap O(b_0,\infty)$ if and only if $b\equiv 0 (t^2)$.
This implies the formula for $T_1(t)\de_{O(b_0,\infty)}$.
\end{proof}

For $z\in \C^*$ with $|z|=1$, set
$$a=a(z,q):=z+z^{-1}+q-2\in \R.$$
Let $\la_1,\la_2$ be two distinct real roots of the equation 
$$\la^2-a\la-q-a=0.$$

Let $\chi:K_1^*\to U(1)$ be a character such that $\chi|_{NK^*_1}(1+c)=\psi_E(cc_0)$, 
In the corollary below we diagonalize the action of the Hecke operator $T_1$ on the space
$\FF_{\chi}^{G(\OO_E)}$ 

\begin{cor}
Assume that $v(b_0)=1$. 
Consider the operator $\wt{T}_1:=q^{-2}T_1(t)$ (which does not depend on $t$) 
acting on the space $\FF_{\chi}^{G(\OO_E)}$. Let $z:=\chi(t,t^{-1})$, and let $\la_1,\la_2$ be the real numbers defined above.
Set 
$$f_0=\sum_{\la\neq \infty} \de_{O(b_0,\la)}, \ \ f_\infty=\sum_{\la\neq 0} \de_{O(b_0,\la)},$$
$$v_0=(z-1)\de_\OO-zf_0+f_\infty, \ \ v_i=-\de_\OO+(1-\frac{\la_i}{z})^{-1}f_0+(1-\la_i z)^{-1}f_\infty, \ i=1,2,$$
$$\wt{\de}_{O(b_0,\la)}=\de_{O(b_0,\la)}-\frac{1}{q-1}\sum_{\mu\neq 0,\infty}\de_{O(b_0,\la)}, \ \ \la\neq 0,\infty.$$
Then the space $\FF_{\chi}^{G(\OO_E)}$ is spanned by the vectors $v_0$, $v_1$, $v_2$, $(\wt{\de}_{O(b_0,\la)})_{\la\in\P^1(k)\setminus \{0,\infty\}}$,
subject to the linear relation $\sum_{\la\neq 0,\infty} \wt{\de}_{O(b_0,\la)}=0$, and  $\wt{T}_1$ acts by
$$\wt{T}_1v_0=0, \ \ \wt{T}_1v_i=\la_i v_i, \ i=1,2,$$
$$\wt{T}_1\wt{\de}_{O(b_0,\la)}=-\wt{\de}_{O(b_0,\la)}, \ \ \la\neq 0,\infty.$$
\end{cor}

\subsection{Global picture}\label{nilp-global-sec}

\subsubsection{Global setup and global representations}\label{global-setup-sec}

Let $C$ be a smooth proper curve over $A=k[\eps]/(\eps^2)$, where $k$ is a finite field. We assume that the characteristic of $k$ is $\neq 2$. 
We denote by $F$ the stalk of the structure sheaf at the general point of $C$ (this is an $A$-algebra, which is a nilpotent extension of a global field
$F_0=k(C_0)$, where $C_0=C\times \Spec(k)$).
For every closed point $p\in C_0$, the completion $\hat{\OO}_{C,p}$ of $\OO_{C,p}$ is a square-zero extension of $\hat{\OO}_{C_0,p}$ of the kind
considered in Sec. \ref{local-setup-sec}. We denote by $F_p$ the total ring of fractions of $\hat{\OO}_{C,p}$.
We have natural maps $F\to F_p$. 

Given $\eta\in \om_{C/A}(F_p)=\om_{C/A}(F)\ot_F F_p$, we have a natural residue $\Res_p(\eta)\in A$, so that the analog of the residue theorem holds: 
for $\eta\in \om_{C/A}(F)$, the sum of residues of $\eta$ over all $p\in C_0$ is zero (see \cite[Sec.\ VII.1]{Hart-RD}).

Let us fix a nontrivial additive character $\psi$ of $k$.
Let $\tau:A\to k$ be the map $a+b\eps\mapsto b$, and let us set
$$\psi_A:=\psi\circ \tau:A\to \C^*.$$ 

For each $p$ we denote by $\psi_{F_p}$ the additive character of $\om_{C/A}(F_p)$ (trivial on $\om_{C/A}\ot \hat{\OO}_{C,p}$) given by
$$\psi_{F_p}(\eta)=\psi_A(\Res_p(\eta)).$$
By the residue theorem, we get a well defined character
$$\psi_C: \om_{C/A}(\A_C)/\om_{C/A}(F)\to \C^*: (\eta_p)\mapsto \prod_p\psi_{F_p}(\eta_p)=\psi_A(\sum_p \Res_p(\eta_p)).$$
Note that for $(\eta_{0p})\in \om_{C_0}(\A_{C_0})$, one has 
$$\psi_{C_0}((\eta_{0p})):=\psi_C((\eps\eta_{0p}))=\psi(\sum_p \Res_p(\eta_{0p})).$$


Now we fix a double covering 
$$\pi:\wt{C}\to C$$ 
of smooth curves over $A$ (so $\pi$ is a flat finite morphism such that 
$\pi_*\OO_{\wt{C}}$ is locally free of rank $2$ on $C$). We denote by $\wt{C}_0\to C_0$ the corresponding map of curves over $k$.
We denote by $\iota$ the corresponding involution on $\wt{C}$, which also acts on functions and on differentials on $\wt{C}$, so that
$\iota(f)=\Tr(f)-f$.
Then the canonical trace morphism $\pi_*\om_{\wt{C}/A}\to \om_{C/A}$ induces an isomorphism 
$$\tau:(\pi_*\om_{\wt{C}/A})^+\rTo{\sim} \om_{C/A},$$
where in the left-hand side we consider $\iota$-invariant part.
Note that if $u$ is a formal parameter on $\wt{C}$ near a ramification point $p\in\wt{C}$, such that $u^2$ is a parameter on $C$,
then 
$$\tau(f(u^2)udu)=f(u^2)d(u^2).$$

We let $F\sub K$ (resp., $F_0\sub K_0$) be the corresponding extension of $k[\eps]/(\eps^2)$-algebras of stalks of $C$ and $\wt{C}$ at general points
(resp., of fields of rational functions on $C_0$ and $\wt{C}_0$).

Let us also fix a pair $(\LL,\a)$, where $\LL$ is a line bundle on $\wt{C}$ and 
$$\a:\Nm(\LL)|_U\rTo{\sim} \om_C|_U$$
is a generic isomorphism of line bundles on $C$ (i.e., isomorphism on some dense open subset $U\sub C$; we do not consider $U$ as part of the data).
For each point $p\in C$, we have the corresponding local representation of $\SL_2(F_p)$ on $\SS(\LL_p\ot_{\OO_{C,p}}K_p))$ 
given by formulas of Proposition \ref{main-local-rep-prop},
where we use the character $\psi_{F_p}$ and the norm map $\a\Nm:\LL(K_p)\to \om_C(F_p)$, and the $\om_C(F_p)$-valued pairing 
$\lan x,y\ran=\Nm(x+y)-\Nm(x)-\Nm(y)$ on $\LL(K_p)$ (see Remark \ref{Nm-rescaling-rem}).
Here and below we often omit $\a$ from the notation.

Furthermore, the same formulas as in the local case define a representation $r$ of $\SL_2(\A_C)$ on $\SS(\LL(\A_{\wt{C}}))$ and a commuting action of $\A_{\wt{C},1}^*$ (the
kernel of the norm homomorphism). 

\begin{lemma} We have a well defined $\SL_2(F)\times K^*_1$-invariant functional $\th$ on $\SS(\LL(\A_{\wt{C}}))$, given by
$$\th(\Phi)=\sum_{x\in \LL(K)}\Phi(x).$$
The map $\Phi\mapsto (g\mapsto \th(g\cdot \Phi))$ gives a well defined map of $\SL_2(\A_C)$-representations
\begin{equation}\label{kappa-L-eq}
\kappa_{\LL}:\SS(\LL(\A_{\wt{C}}))\to \SS(\SL_2(F)\backslash\SL_2(\A_C)).
\end{equation}
\end{lemma}

Our goal is to understand (most of) the image of $\kappa_\LL$, and in case of $G(\hat{\OO})$-invariant functions $\Phi$, relate $\kappa_{\LL}(\Phi)$ to functions
on Hitchin fibers.

Similarly to the local case, we have a global analog of the partial Fourier transform,
$$S_1:\SS(\LL(\A_{\wt{C}}))\rTo{\sim} \FF_{\LL},$$
where $\FF_{\LL}\sub \C_{lc}(\LL(\A_{\wt{C}})^2)$ consists of $f$ satisfying 
\begin{equation}\label{F-global-def-eq}
f(v+\eps v'_0)=\psi_{C_0}(B(v \mod (\eps),v'_0))f(v)
\end{equation}
(the global version of \eqref{F-def-eq}) and with the compact support in $\LL(\A_{\wt{C}_0})^2$.

One can easily check that $\th=\th_1\circ S_1$, where for $f\in \FF_{\LL}$, 
$$\th_1(f):=\sum_{x\in \LL(K_0)^2} f(x),$$
where we use the fact that $f|_{\LL(K)^2}$ descends to $f|_{\LL(K_0)^2}$ (due to triviality of $\psi_{C_0}$ on $\om_{C_0}(F_0)$).

So we should study the map $\kappa'_{\LL}=\kappa_{\LL}\circ S_1^{-1}$ sending $f$ to $g\mapsto \th_1(g\cdot f)$.

Let us denote by $\A_{\wt{C}_0}^-\sub \A_{\wt{C}_0}$ (resp., $K_0^-\sub K_0)$ the subgroup of $\iota$-antiinvariant elements, or equivalently the kernel of the trace map.
We have an inclusions of groups 
$$1+\eps \A_{\wt{C}_0}^-\sub \A^*_{\wt{C}_0,1}=\ker(\Nm:K^*\to F^*), \ \ 1+\eps K_0^-\sub K^*_1=\ker(\Nm:K^*\to F^*).$$
Note that the map $\kappa_{\LL}$ factors through the space of 
$(1+\eps K_0^-)$-coinvariants  $\SS(\LL(\A_{\wt{C}}))_{(1+\eps K_0^-)}\simeq (\FF_\LL)_{(1+\eps K_0^-)}$.
The latter space is equipped with an action of a compact group $(1+\eps \A_{\wt{C}_0}^-)/(1+\eps K_0^-)$ 
Thus, it is enough to study the restriction of $\kappa_{\LL}$ to isotypic subspaces 
$$\FF_{\LL,c_0}\sub (\FF_\LL)_{(1+\eps K_0^-)}$$ 
associated with elements 
$c_0\in \om_{\wt{C}_0}(K_0)$ such that $\iota^*c_0=-c_0$, where 
$\FF_{\LL,c_0}$ consists of $\phi$ such that $(1+\eps a_0)\phi\equiv \psi_{C_0}(a_0\cdot c_0)\phi$ for $a_0\in \A_{\wt{C}_0}^-$. 

Similarly to the local case, for $c_0\neq 0$ we consider the quadric $Q(\LL,c_0)\sub \LL(\A_{\wt{C}_0})^2$ given by \eqref{Q-c0-eq}.
Let us denote by $\wt{Q}(\LL,c_0)\sub \LL(\A_{\wt{C}})^2$ the preimage of $Q(c_0)$ under the reduction.
The natural restriction map identifies $\FF_{\LL,c_0}$ with the space of functions on $\wt{Q}(\LL,c_0)$ satisfying \eqref{F-def-eq}
and with compact support in $Q(\LL,c_0)$.

We denote by $\de_{\OO}\in \FF_{\LL,c_0}$ the unique function such that $\de_{\OO}|_{\wt{Q}(\LL,c_0)(\hat{\OO})}\equiv 1$ and
$\de_{\OO}(v)=0$ if the reduction of $v$ is not in $Q(\LL,c_0)(\hat{\OO})$.

The $\SL_2(F)$-invariant functional $\th_1$ on $(\FF_\LL)_{(1+eK_0^-)}$ restricts to the $\SL_2(F)$-invariant functional on $\FF_{\LL,c_0}$,
$$\th_{c_0}:\phi\mapsto \sum_{x\in Q(\LL,c_0)(K_0)}\phi(x),$$
where we use the fact that the restriction of $\phi\in \FF_{\LL}$ (resp., $\phi\in \FF_{\LL,c_0}$) to $\LL(K)^2$ (resp., $\wt{Q}(\LL,c_0)(K)$)
descends to a function on $\LL(K_0)^2$ (resp., $Q(\LL,c_0)(K_0)$).

Thus, the image of $\kappa_{\LL}$ is the sum over $c_0$ of the images of the maps
\begin{equation}\label{kappa-L-c0-map}
\kappa_{\LL,c_0}:\FF_{\LL,c_0}\to \SS(\SL_2(F)\backslash\SL_2(\A_C)): \phi\mapsto (g\mapsto \th_{c_0}(g\cdot\phi)).
\end{equation}
We will study these images for $c_0\neq 0$ (see Remark \ref{nilp-rem} regarding the case $c_0=0$).

The analog of Lemma \ref{v-c-lem} holds in the global setting with some changes. Namely, for $c\in \om_{\wt{C}}(K)$ such that $\iota^*c=-c$
(resp., $c_0\in \om_{\wt{C}_0}(K_0)$ such that $\iota^*c_0=-c_0$), we still have the
quadric $Q(\LL,c)\sub \LL(\A_{\wt{C}})^2$, and for $c_0\neq 0$,
the right action of $\SL_2(\A_C)$ on $Q(\LL,c)$ (resp., of $\SL_2(\A_{C_0})$ on $Q(\LL,c_0)$) is simply transitive. The analogs of elements $v_c$ (resp., $v_{c_0}$) are constructed as
follows. Fix an element $s\in \LL(K)\setminus 0$ (resp., $s_0\in \LL(K_0)\setminus 0$), and set
$$v_{c,s}:=(-\frac{c}{2\iota(s)},s)\in \LL(K)^2, \ \ v_{c_0,s_0}:=(-\frac{c_0}{2\iota(s_0)},s_0)\in \LL(K_0)^2.$$
Then $v_{c,s}\in Q(\LL,c)(K)$ (resp., $v_{c_0,s_0}\in Q(\LL,c_0)(K_0)$), and we have identification of the fibers of $\eta$ with $\A_{\wt{C},1}^*$-orbits.

Similarly to Lemma \ref{v-c-lem}(ii), (still for $c_0\neq 0$) we have identifications
\begin{equation}\label{ics-isom}
i_{c,s}:\A^*_{\wt{C},1}\rTo{\sim} \St_{\eta_{c,s}}, \ \  i_{c_0,s_0}:\A^*_{\wt{C}_0,1}\rTo{\sim} \St_{\eta_{c_0,s_0}}
\end{equation}
such that $av_{c,s}=v_{c,s}\iota_{c,s}(a)$ (resp., $av_{c_0,s_0}=v_{c_0,s_0}\iota_{c_0,s_0}(a)$), where $\eta_{c,s}=\eta(v_{c,s})$ (resp.,
$\eta_{c_0,s_0}=\eta(v_{c_0,s_0})$).

Furthermore, similarly to the local case, the map
$$\phi\mapsto r(g)(\phi)(v_{c,s})=\phi(v_{c,s}g)$$
gives an isomorphism from $\FF_{\LL,c_0}$ to the representation of $\SL_2(\A_C)$, compactly induced from the character 
$1+\eps A\mapsto \psi_{C_0}(\tr(\eta_{c_0,s_0}A))$ of $1+\eps\fg(A_{C_0})$.


\subsubsection{Correspondence with adjoint orbits and induced representations}

Let $\LL_0$ be the line bundle on $\wt{C}_0$ induced by $\LL$, and let $\a_0$ be the generic isomorphism
from $\Nm(\LL_0)$ to $\om_{C_0}$, induced by $\a$.

For nonzero $c_0\in \om_{\wt{C}_0}(K_0)$ such that $\iota^*(c_0)=-c_0$, let 
$$\Om_{\LL_0,\a_0,c_0}:=\eta(Q(\LL_0,c_0)(K_0))$$ 
be the corresponding $\SL_2(F_0)$-orbit in $\fg\ot\om_{C_0}(F_0)$. Let us describe the obtained correspondence between our geometric data and 
$\SL_2(F_0)$-orbits in $\fg\ot\om_{C_0}(F_0)$.

\begin{prop}\label{orbits-correspondence-prop}
(i) Let $c_0\in \om_{\wt{C}_0}(K_0)\setminus 0$ be such that $\iota^*(c_0)=-c_0$, and let $s_0$ be a nonzero element of $\LL_0(K_0)$.
Then $\Om_{\LL_0,\a_0,c_0}$ is the orbit of
$$\eta_{c_0,\a_0,s_0}:=\left(\begin{matrix} 0 & -\a_0\Nm(s_0)\\ -\frac{c_0^2}{4\a_0\Nm(s_0)} & 0\end{matrix}\right).$$
This orbit is regular (semisimple) elliptic.

\noindent
(ii) Consider the groupoid $\GG$ whose objects are data $(\pi:\wt{C}_0\to C_0, \LL_0,\a_0,c_0)$, where $\pi:\wt{C}_0\to C_0$ is a double covering with $\wt{C}_0$ smooth irreducible, $\LL_0$ is a line bundle on $\wt{C}_0$,
$\a_0$ is a generic isomorphism of $\Nm(\LL_0)$ with $\om_{C_0}$, $c_0$ is a nonzero antiinvariant rational differential on $\wt{C}_0$. Morphisms
are given by isomorphisms of double coverings and by generic isomorphisms of line bundles $\LL_0$ (compatible with $\a_0,c_0$).
Then the map 
\begin{equation}\label{main-orbit-map}
(\wt{C}_0\to C_0,\LL_0,\a_0,c_0)\mapsto \Om_{\LL_0,\a_0,c_0}
\end{equation}
induces a bijection from isomorphism classes of objects in $\GG$ to the set of regular elliptic $\SL_2(F_0)$-orbits in $\fg\ot\om_{C_0}(F_0)$.
\end{prop}

\begin{proof}
(i) Note that $-\det(\eta_{c_0,s_0})=c_0^2/4$, which is not a square in $\om_{C_0}(F_0)$. Indeed, if $c_0^2=x^2$ then $c_0=\pm x$ which is impossible since
$\iota^*(c_0)=-c_0$. This implies that the orbit is regular elliptic.

\noindent
(ii) Let us constuct the inverse map. Regular elliptic $\SL_2(F_0)$-orbits in $\fg\ot F_0$ 
are represented by matrices of the form $A(f_1,f_2):=\left(\begin{matrix} 0 & f_1 \\ f_2 & 0\end{matrix}\right)$, where $f_1,f_2$ are nonzero, such that
$-\det A(f_1,f_2)=f_1f_2$ is not
a square in $F_0$. The matrix $A(f'_1,f'_2)$ represents the same orbit as $A(f_1,f_2)$ if and only if $f'_1=\Nm(\phi)f_1$, $f'_2=\Nm(\phi)^{-1}f_2$ for $\phi\in F_0(\sqrt{f_1f_2})$.
The same is true for elliptic $\SL_2(F_0)$-orbits in $\fg\ot \om_{C_0}(F_0)$ with $f_1,f_2$ replaced by nonzero elements $\om_1,\om_2\in \om_{C_0}(F_0)$. Then 
$-\det A(\om_1,\om_2)$ is in $\om_{C_0}^2(F_0)$ and we require that it is not a square of an element in $\om_{C_0}(F_0)$, or equivalently $\om_2/\om_1$ is not a square in $F_0$. 
We use $F_0(\sqrt{\om_2/\om_1})$ as the corresponding quadratic extension.

Starting with the orbit of $A(\om_1,\om_2)$ (where $\om_1\om_2$ is not a square), we define the smooth curve $\wt{C}_0$ which is a double covering of $C_0$, by setting 
$$k(\wt{C}_0)=F_0(\sqrt{\om_2/\om_1}).$$
We take $\LL_0$ to be trivial, and we set 
$$\a_0:=-\om_1,  \ \ c_0:=2\sqrt{\om_2/\om_1}\cdot \om_1,$$
where in the last formula we view $\om_1$ as an invariant differential on $\wt{C}_0$, while $\sqrt{\om_2/\om_1}$ is an antiinvariant rational function on $\wt{C}_0$.
With this definition, we have $\eta_{c_0,\a_0,1}=A(\om_1,\om_2)$, which shows that this is the right inverse to \eqref{main-orbit-map}.

Conversely, if we start with $(\wt{C}_0\to C_0,\LL_0,\a_0,c_0)$, and a generic section $s_0$ of $\LL_0$, then we get 
$$\om_1=-\a_0\Nm(s_0), \ \ \om_2=-\frac{c_0^2}{4\a_0\Nm(s_0)},$$
then $\om_2/\om_1=(c_0/2\a_0\Nm(s_0))^2$, so we get an isomorphism 
$$F_0(\sqrt{\om_2/\om_1})\rTo{\sim} k(\wt{C}_0)$$
sending $\sqrt{\om_2/\om_1}$ to the antiinvariant rational function $c_0/2\a_0\Nm(s_0)$ on $\wt{C}_0$. 

Now $s_0$ induces an isomorphism of the data $(\LL_0,\a_0)$ with
$(\OO_{\wt{C}_0},-\a_0\Nm(s_0)=\om_1)$, so our maps are indeed mutually inverse.
\end{proof}

Recall that for $\eta_0\in \fg\ot \om_{C_0}(F_0)$ we have an $\SL_2(\A_C)$-representation $\wt{\SS}_{\eta_0}$ with the $\SL_2(F)$-invariant functional $\Theta_{\eta_0}$,
which induces an isomorphism $\wt{\SS}_{\eta_0}\simeq \SS(\SL_2(F)\backslash G(\A_C))_{\Om_{c_0}}$ (see Sec.\ \ref{recoll-sec}).

\begin{theorem}\label{two-reps-two-funct-thm} 
Fix $\wt{C}\to C$, $\LL$, $\a$, $c_0$ as before (in particular, $c_0\neq 0$).

\noindent
(i) For $\eta_0\in \Om_{\LL_0,\a_0,c_0}$, one has a map of $\SL_2(\A_C)\times \A^*_{\wt{C},1}$-representations
\begin{equation}\label{FLc0-to-S-eta-map}
\wt{\kappa}_{\LL,c_0}:\FF_{\LL,c_0}\to \wt{\SS}_{\eta_0}:\phi\mapsto (g\mapsto \th_{\eta_0,c_0}(r(g)\cdot\phi)),
\end{equation}
where 
$$\th_{\eta_0,c_0}(\phi)=\sum_{v\in Q(\LL,c_0)(K_0): \eta(v)=\eta_0} \phi(v),$$
and the action of $\A^*_{\wt{C},1}$ on $\wt{\SS}_{\eta_0}$ comes from the identification \eqref{ics-isom}
(for some choice of $(c,s)$).

Furthermore, we have equality of $\SL_2(F)$-invariant functionals on $\FF_{\LL,c_0}$,
\begin{equation}\label{th-c0-Theta-eq}
\th_{c_0}=\Theta_{\eta_0}\circ \wt{\kappa}_{\LL,c_0}.
\end{equation}

\noindent
(ii) The map of $\SL_2(\A_C)$-representations $\kappa_{\LL,c_0}$ (see \eqref{kappa-L-c0-map}) factors as the composition
$$\FF_{\LL,c_0}\rTo{\wt{\kappa}_{\LL,c_0}}\wt{\SS}_{\eta_0}\rTo{\kappa_{\eta_0}} \SS(\SL_2(F)\backslash\SL_2(\A_C)),$$
where $\kappa_{\eta_0}$ is given by \eqref{kappa-eta-def-eq}.
\end{theorem}

\begin{proof}
(i) First, we observe that the summation in the definition of $\th_{\eta_0,c_0}$ is finite since the support of $\phi$ in $Q(\LL,c_0)(\A_{\wt{C}_0})$
is compact. Furthermore, the fact that the function $g\mapsto\th_{\eta_0,c_0}(r(g)\phi)$ belongs to $\wt{\SS}_{\eta_0}$ boils down to
$\St_{\eta_0}(K)$-invariance of $\th_{\eta_0,c_0}$ and the identity of Corollary \ref{eta-action-cor}.
The compatibility with the  $\SL_2(\A_C)$-action is clear.
To check compatibility with the $\A^*_{\wt{C},1}$-action we use the fact that it acts transitively on the set of 
$v\in Q(\LL,c_0)$ such that $\eta(v)=\eta_0$ (see Lemma \ref{v-c-lem}), and hence, for any such $v$,
$$av=vi_{c,s}(a).$$

The equality \eqref{th-c0-Theta-eq} follows directly from the definitions.

\noindent
(ii) Using the definitions this immediately boils down to \eqref{th-c0-Theta-eq}.
\end{proof}

\begin{remark}\label{nilp-rem}
In the case $c_0=0$, the set $Q(\LL,0)$ consists of vectors $a\cdot v$, where $a\in \A^*_{\wt{C}_0}$, $v\in \A_{C_0}^2$, and $\eta(Q(\LL,0))$ is a union of
nilpotent orbits of $\SL_2(F_0)$. One can show that the image of $\kappa_{\LL,0}$ contains no nonzero cuspidal functions. 
\end{remark}

The map \eqref{FLc0-to-S-eta-map} has a natural interpretation in terms of induced representations. Suppose we have 
\begin{itemize}
\item an $l$-group $G$ (i.e., a topological group with a basis of open compact subgroups);
\item a closed normal subgroup $N\sub G$; 
\item a continuous character $\chi$ of $N$;
\item a discrete subgroup $\Ga\sub G$ such that $\chi|_{\Ga\cap N}\equiv 1$, $\Ga/\Ga\cap N$ is discrete in $G/N$, and the adjoint action of $\Ga$ preserves $\chi$.
\end{itemize}

Consider compactly induced representations 
$\Ind_N^G\chi$ and $\Ind_{\Ga N}^G\chi'$, where $\chi'$ is the unique character of $\Ga N$, trivial on $\Ga$, extending $\chi$.
Then we have a natural map of $G$-representations 
\begin{equation}\label{induced-repr-map}
\Theta:\Ind_N^G\chi\to \Ind_{\Ga N}^G \chi': \phi\mapsto (g\mapsto \sum_{\ga\in \Ga/\Ga\cap N}\phi(\ga g)). 
\end{equation}

\begin{lemma}\label{ind-rep-surj-lem}
In the above situation the map $\Theta$ is surjective.
\end{lemma}

\begin{proof}
Let $\ov{U}$ be a small compact open in $G/N$, so that $\ov{U}\cap \ga \ov{U}=\emptyset$ for any $\ga\in \Ga\setminus \Ga\cap N$,
and let $U\sub G$ be the corresponding open neighborhood of $N$ (the preimage of $\ov{U}$ under the projection $G\to G/N$).
It is enough to prove that any $f\in \Ind_{\Ga N}^G \chi'$, supported on $\Ga U$, is in the image of $\Theta$. Let $(\ga_i)$ be representatives
of $\Ga\cap N$-cosets in $\Ga$, with $\ga_1=1$. Then $\Ga U=\sqcup_i \ga_i U$. Define the function $\phi$ on $\Ga U$ by
$\phi|_U=f|_U$ and $\phi|_{\ga_i U}=0$ for $\ga_i\neq 1$. Then $\phi\in \Ind_N^G\chi$ and $\Theta(\phi)=f$.
\end{proof}

\begin{theorem}\label{main-thm1}
Assume characteristic of $k$ is $\neq 2$.
For every regullar elliptic orbit $\Om$ in $\fg\ot \om_{C_0}(F_0)$, there exists data $(\wt{C}\to C,\LL,\a,c_0)$ such that the map
$\kappa_{\LL,c_0}:\FF_{c_0}\to \SS(\SL_2(F)\backslash\SL_2(\A_C))$ (see \eqref{kappa-L-c0-map})
maps surjectively onto $\SS(\SL_2(F)\backslash\SL_2(\A_C))_\Om$.
\end{theorem}

\begin{proof}
By Proposition \ref{orbits-correspondence-prop}, we have a correspondence between elliptic orbits and data $(\wt{C}_0\to C_0,\LL_0,\a_0,c_0)$.
Given a smooth curve $C$ over $A$ and a double covering $\pi_0:\wt{C}_0\to C_0$, where $\wt{C}_0$ is smooth, 
we claim that exists a double covering $\pi:\wt{C}\to C$ of smooth curves over $A$ inducing $\pi_0$.
Indeed, we have $\pi_{0*}\OO_{\wt{C}_0}\simeq \OO_{C_0}\oplus M_0$, and the algebra structure on $\OO_{C_0}\oplus M_0$ is determined by the
ramification divisor, $D_0\sub C_0$, where $\OO_{C_0}(D_0)\simeq M_0^2$. We can extend $D_0$ to an effective Cartier divisor $D\sub C$.
We claim that there exists a  line bundle $M$ on $C$ extending $M_0$, such that $\OO_C(D)\simeq M^2$.
Indeed, this follows from the exact sequence 
$$H^1(C_0,\OO)\to \Pic(C)\to \Pic(C_0)\to 0$$
together with the fact that the group $H^1(C_0,\OO)$ is $2$-divisible.
Using $D$ we can define an algebra structure on $\OO_C\oplus M$, thus, defining the required double covering $\pi:\wt{C}\to C$ (the fact that $\wt{C}$ is smooth over $A$
follows from the smoothness of $\wt{C}_0$ and flatness of $\pi$).
Finally, we can extend $(\LL_0,\a_0)$ to a similar data $(\LL,\a)$ over $A$.

Since $\SS(\SL_2(F)\backslash\SL_2(\A_C))_\Om$ is the image of $\kappa_{\eta_0}$, by Theorem \ref{two-reps-two-funct-thm}, it remains to show that
the map $\wt{\kappa}_{\LL,c_0}$ is surjective. But this follows from Lemma \ref{ind-rep-surj-lem} due to the identification of 
$\FF_{\LL,c_0}$ with the compactly induced representation of $\SL_2(\A_C)$ from the character $1+\eps A\mapsto \psi_{C_0}(\tr(\eta_0A))$ of the subgroup $1+\eps\fg(\A_{C_0})$
(where $\eta_0=\eta_{c_0,s_0}$),
while $\wt{\SS}_{\eta_0}$ is compactly induced from the corresponding character of $(1+\eps\fg(\A_{C_0}))\cdot \bSt_{\eta_0}(F)$.
\end{proof}

\subsubsection{Relation to functions on moduli of Higgs $\SL_2$-bundles}

From now on we will be interested in $G(\hat{\OO})$-invariant vectors. 
We assume that $\a:\Nm(\LL)\to \om_C$ is an isomorphism.

\begin{lemma}
The space $\FF_{\LL,c_0}^{G(\hat{\OO})}$ vanishes unless the differential $c_0$ is regular on $\wt{C}_0$ and has a double zero at every ramification point of $\pi:\wt{C}_0\to C_0$.
\end{lemma}

\begin{proof}
This follows from Lemmas \ref{c0-int-lem} and \ref{c0-int-lem-bis}. 
More precisely, we use the fact that the corresponding local element 
$c_{0,p}$ is given by $c_0/du$ at a point where $\pi$ is unramified, and by $c_0/udu$ at a ramfication point of $\pi$ (where $u$ is a local parameter).  
Furthermore, at the ramification point $c_0/udu$ is still antiinvariant, so it vanishes at this point, which implies that $c_0$ has a double zero.
\end{proof}

Let us explain the relation between our data $(\wt{C}_0,\LL_0,c_0)$ and Higgs $\SL_2$-bundles (under some genericity assumptions).

\begin{lemma}\label{Higgs-corr-lem}
For a fixed smooth curve $C_0$ over $k$,
there is a natural equivalence of the following two groupoids:

\noindent
(G1) $(\pi:\wt{C}_0\to C_0,\LL_0,\a_0,c_0)$, where $\wt{C}_0$ is a smooth curve, $\pi$ is a double covering, $\LL_0$ a line bundle on $\wt{C}_0$ together with an isomorphism
$\a_0:\Nm(\LL_0)\rTo{\sim} \om_{C_0}$, and $c_0\in \om_{\wt{C}_0}$ is a nonzero regular differential, antiinvariant under the involution $\iota$ associated with the double covering,  such that
the divisor of zeros of $c_0$ is twice the ramification divisor.

\noindent
(G2) $(V,\phi)$ a Higgs $\SL_2$-bundle on $C_0$ such that $\de:=-\det(\phi)$ has only simple zeros.

Under this correspondence $V=\pi_*\LL_0$ and $\phi$ is induced by the norm map $\Nm:S^2V\to \Nm(\LL_0)\simeq \om_{C_0}$,
using the isomorphism 
\begin{equation}\label{End-S2V-isom}
\und{\End}_0(V)\rTo{\sim}S^2V^\vee:\phi\mapsto (v_1v_2\mapsto \frac{1}{2}(\det(\phi(v_1),v_2)+\det(\phi(v_2),v_1)))
\end{equation}
(where we use the trivialization of $\det(V)$). 
\end{lemma}

\begin{proof} To go from (G1) to (G2) we set $V:=\pi_*\LL_0$. 
This is a module over $\pi_*\OO_{\wt{C}_0}=\OO_{C_0}\oplus M^{-1}$, where
$\wt{C}_0$ is given by $x^2=F$ for some $F\in H^0(M^2)$, and the module structure is given by a map $\phi:V\to V\ot M$ such that
$\tr(\phi)=0$ and $-\det(\phi)=F$. 

By Grothendieck-Serre duality, the invariant part of $\pi_*\om_{\wt{C}_0}$ is identified with $\om_{C_0}$.
On the other hand, the fact that the divisor of zeros of $c_0$ is exactly $\pi^*\div(F)$, means that we have a unique $\iota$-antiinvariant isomorphism
$$\pi^*M^2\rTo{\sim} \om_{\wt{C}_0},$$
sending $\pi^*F$ to $c_0$. Applying $\pi_*$ to this isomorphism we deduce an isomorphism 
$$M\simeq (\pi_*\pi^*M^2)^-\simeq (\pi_*\om_{\wt{C}_0})^+\simeq \om_{C_0}.$$
Hence, we can view $\phi$ as a Higgs field on $V$ (and $F$ as a quadratic differential).
The isomorphism $\Nm(\LL_0)\simeq \om_C$ corresponds to an isomorphism $\det V\simeq M^{-1}\om_{C_0}\simeq \OO_{C_0}$.

Conversely, starting from $(V,\phi)$ such that $\de=-\det(\phi)$ has only simple zeros, we construct $\pi:\wt{C}_0\to C_0$ as the divisor $u^2=p^*\de$
in $T^*C_0$ (where $p$ is the projection to $C_0$), where $u$ is the tautological section of $p^*\om_{C_0}$. Then the restriction of $u$ defines an
antiinvariant section $c_0$ of $\pi^*\om_{C_0}$ such that $c_0^2=\pi^*\de$. Equivalently, $c_0$ is the canonical antiinvariant section of $\pi^*\om_{C_0}$
corresponding to the identification of the antiinvariant part of $\pi_*\pi^*\om_{C_0}$ with $\OO_{C_0}$.
We can also view $c_0$ as an antiinvariant differential on $\wt{C}_0$.

Let us define the $\pi_*\OO_{\wt{C}_0}$-module structure on $V$, so that the action of the summand $\om_{C_0}\sub \pi_*\OO_{\wt{C}_0}$ 
is given by the Higgs field $\phi$. This structure gives a coherent sheaf $\LL_0$ on $\wt{C}_0$ such that $\pi_*\LL_0=V$. Since $\wt{C}_0$ is smooth, 
$\LL_0$ is automatically a line bundle (since $V$ is a rank $2$ bundle on $C_0$).

To prove the equality of $\phi$ with the norm map $\Nm:S^2V\to\om_{C_0}$, we can argue over the general point of $C_0$. 
Consider a basis $(1, u)$ of $\pi_*\OO$, where $u$ is a section of $\om_{C_0}^{-1}$.
By definition, the canonical isomorphism 
\begin{equation}\label{norm-det-isom}
\om_{C_0}^{-1}\ot \Nm(\LL_0)\simeq \det(\pi_*\OO)\ot \Nm(\LL_0)\to \det(\pi_*\LL_0)
\end{equation}
takes $(u\we 1)\ot \Nm(s)$ to $us\we s$, where $s\in \LL_0$.
Now the assertion follows from the fact that by the definition of the isomorphism \eqref{End-S2V-isom},
$\phi$ corresponds to the quadratic map 
$$\om_{C_0}^{-1}\ot S^2\pi_*(\LL_0)\to \det(\pi_*\LL_0): u\ot s^2\mapsto us\we s.$$
\end{proof}


We conclude with an explicit computation of the functions on the moduli spaces of Higgs bundles corresponding to our cuspidal
functions on $\Bun_{\SL_2}(C)$ constructed via theta correspondence, i.e., the image of spherical vectors unders the maps
$\kappa_{\LL,c_0}$ (see \eqref{kappa-L-c0-map}).
Recall that we can identify the relevant cuspidal piece $\wt{\SS}_{\eta_0}^{G(\hat{\OO})}$ with sections of a $\C^*$-torsor $\rL_\psi$ on 
$$\MM_{\eta_0}^{Higgs}(C_0)\simeq \bSt_{\eta_0}(F_0)\backslash \{g\in \SL_2(\A_{C_0}) \ |\ g^{-1}\eta_0g\in \fg(\hat{\OO})\}/\SL_2(\hat{\OO}).$$

Let us denote by $\de_{\hat{\OO}}$ the function in $\FF_{\LL,c_0}^{G(\hat{\OO})}$ corresponding to integer points (see Sec.\ \ref{global-setup-sec} and 
Sec.\ \ref{modt2-sph-sec} in the local case).

\begin{theorem}\label{main-thm2} 
Assume characteristic of $k$ is $\neq 2$.
Let $\pi:\wt{C}\to C$ be a double covering of smooth curves over $A$, such that the corresponding double covering
$\pi_0:\wt{C}_0\to C_0$ is associated with a quadratic differential $\de$ with simple zeros on $C_0$. We denote by 
$c_0$ be the canonical antiinvariant differential on $\wt{C}_0$ (see the proof of Lemma \ref{Higgs-corr-lem}).
Let also $\LL$ be a line bundle on $\wt{C}$
equipped with an isomorphism $\a:\Nm(\LL)\rTo{\sim}\om_{C/A}$.

\noindent
(i) Let $f^{Higgs}_{\LL,c_0}$ be the element of $\SS(\MM_{\eta_0}^{Higgs}(C_0),\rL_\psi)$ associated with $\wt{\kappa}_{\LL,c_0}(\de_{\OO})$ (where $\eta_0=\eta_{c_0,\a_0,s_0}$
for some $s_0\in \LL_0(K_0)$), and let
$(V_0=\pi_{0,*}\LL_0,\phi_0)$ denote the Higgs $\SL_2$-bundle associated with $(\LL_0,\a_0,c_0)$ via the correspondence of Lemma \ref{Higgs-corr-lem}.
Then $f^{Higgs}_{\LL,c_0}$ is supported at the point $(V_0,\phi_0)$, and if we use $V:=\pi_*\LL$ to trivialize $\rL_\psi$ at this point, then 
$$f^{Higgs}_{\LL,c_0}=2\de_{(V_0,\phi_0)}.$$

\noindent
(ii) Let $f_{\LL,c_0}$ be the function on $\Bun_{\SL_2}(C)$ associated with $\kappa_{\LL,c_0}(\de_\OO)$, and let $(V_0,\phi_0)$ be the corresponding Higgs $\SL_2$-bundle
(see part (i)). Set $V:=\pi_*\LL$; this is an $\SL_2$-bundle on $C$ whose reduction is $V_0$.
Then for $V'\in \Bun_{\SL_2(C)}$, one has $f_{\LL,c_0}(V')=0$ unless the reduction of $V'$ is isomorphic to $V_0$, and
$$f_{\LL,c_0}(V+y)=2\sum_{\phi\in H^0(\und{\End}(V_0)\ot\om_{C_0}): \det(\phi)=-\de}\psi(y(\phi)),$$
where $y\in H^1(\und{\End}(V_0))\simeq H^0(\und{\End}_0(V_0)\ot\om_{C_0})^*$, and we use the standard action of this group on the fiber of
$\Bun_{\SL_2}(C)\to \Bun_{\SL_2}(C_0)$.

\noindent
(iii) Assume now that $C=C_0\times_{\Spec(k)} \Spec(A)$, $\wt{C}=\wt{C}_0\times_{\Spec(k)}\Spec(A)$, and $(\LL,\a)$ is obtained by extension of scalars to $A=k[\eps]/(\eps^2)$ from the data $(\LL_0,\a_0)$ over $k$.
Let $\chi$ be a character of $\PP_0(\wt{C}_0/C_0)$ (see Sec.\ \ref{fin-cusp-sec}), and let $f_{\LL_0,c_0}$ denote the corresponding function on $\Bun_{\SL_2}(C)$.
Then the function 
$$\Phi_{\LL_0,\chi}:=\sum_{\xi\in \PP_0(\wt{C}_0/C_0)} \chi(\xi) f_{\xi\ot \LL,c_0}$$
on $\Bun_{\SL_2}(C)$ is nonzero and is an eigenfunction for all Hecke operators $T_p$ associated with simple divisors $p\sub C$ (i.e., $p$ is a $A$-flat subscheme whose reduction
from $A$ to $k$ is a closed point $p_0\sub C_0$).
More precisely, such $p$ corresponds to a tangent vector $v$ at $p_0$ and 
$$T_p\Phi_{\LL_0,\chi}=\begin{cases} q^2(\psi(\lan c_0,v\ran)\chi(\OO(p_1-p_2))+\psi(-\lan c_0,v\ran)\chi(\OO(p_2-p_1)))\Phi_{\LL_0,\chi}, & \pi^{-1}(p)=\{p_1,p_2\}, p_1\neq p_2,\\
0, & \text{otherwise}.\end{cases}$$
\end{theorem}

\begin{proof}
(i) Recall that the rank $2$ bundle $V(g)$ associated with an adele $g=(g_p)\in \SL_2(\A_{C_0})$ is trivialized at the general point, and for each $p\in C_0$,
the $\OO_{C_0,p}$-submodule $V(g)_p\sub F_0^2$ is given by $V(g)_p=F_0^2\cap g_p\hat{\OO}_p^2$. If in addition $g^{-1}\eta_0g\in \fg\ot \om(\hat{\OO})$,
then $\eta_0$ defines a Higgs field on $V(g)$. If we think of $F_0^2$ as column space, then we can identify the fiber of the dual bundle $V(g)^\vee$ 
with the row space $F_0^{2,rows}$, so that $V(g)^\vee_p=F_0^{2,rows}\cap \hat{\OO}_p^{2,rows}g_p^{-1}$.

By definition, if we lift $g$ to $\wt{g}\in \SL_2(\A_C)$ then 
$$\wt{\kappa}_{\LL,c_0}(\de_{\OO})(\wt{g})=\sum_{v\in Q(\LL,c_0)(K_0):\eta(v)=\eta_0}\de_{\OO}(\wt{v}\wt{g}),$$
where $\wt{v}\in \wt{Q}(\LL,c_0)(K)$ is any lifting of $v$ (the value does not depend on a choice of $\wt{v}$).

By definition, $\de_{\OO}(\wt{v}\wt{g})$ is nonzero if and only if $vg\in \LL_0(\hat{\OO})^2$, which means that $v$ can be viewed as
a global section of $\pi^*V(g)^\vee\ot \LL_0$. Furthermore, if $\wt{v}\wt{g}\in \LL(\hat{\OO})^2$, i.e., $\wt{v}$ corresponds to a global section of 
$V(\wt{g})^\vee \ot \LL$ lifting $v$, then $\de_{\OO}(\wt{v}\wt{g})=1$. 

Next, let us analize the condition $\eta(v)=\eta_0$, viewing $v$ as a section of 
$$H^0(\wt{C}_0,\pi^*V(g)^\vee\ot \LL_0)\simeq \Hom(\pi^*V(g),\LL_0)\simeq \Hom(V(g),V_0).$$
Recall that under the identification $\und{\End}_0(V_0)\ot\om_{C_0}\simeq S^2V_0^\vee\ot\om_{C_0}$,
$\phi_0$ corresponds to the norm map 
$$\Nm_{\LL_0}:S^2V_0\to \Nm(\LL_0)\simeq \om_{C_0}.$$
Now we claim that 
\begin{equation}\label{eta-Nm-eq}
\eta(v)=\Nm_{\LL_0}\circ S^2(v)\in \Hom(S^2V(g),\om_{C_0})\simeq H^0(\und{\End}_0(V(g))\ot\om_{C_0}),
\end{equation}
where we use the natural quadratic map
$$S^2:\Hom(V(g),V_0)\to \Hom(S^2V(g),S^2V_0).$$

We will prove \eqref{eta-Nm-eq} in two steps. First, we claim that
$$\pi^*\eta(v)=\iota(v)\cdot v\in H^0(\wt{C}_0,S^2(\pi^*V(g)^\vee)\ot \iota^*(\LL_0)\ot\LL_0)\simeq H^0(\wt{C}_0,\pi^*(\und{\End}_0(V(g))\ot\Nm(\LL_0))).$$  
Indeed, this follows directly from the identity \eqref{eta-main-identity}.

Now to prove \eqref{eta-Nm-eq}, it remains to check the equality
$$\pi^*(\Nm_{\LL_0}\circ S^2(v))=\iota(v)\cdot v\in S^2(\pi^*V(g)^\vee)\ot \iota^*(\LL_0)\ot\LL_0,$$
where on the left we view $S^2(v)$ as a morphism $V(g)\to \pi_*\LL_0$. We can rewrite the left-hand side as the composition
$$S^2(\pi^*V(g))\to S^2(\pi^*\pi_*\LL_0)\to S^2(\iota^*(\LL_0)\oplus \LL_0)\to \iota^*(\LL_0)\ot \LL_0,$$
where the last arrow is the natural projection. It remains to observe that the map $\pi^*V(g)\to \iota^*(\LL_0)\oplus \LL_0$ obtained from $v$ has components
$\iota^*(v)$ and $v$.

Thus, the condition $\eta(v)=\eta_0$ becomes
\begin{equation}\label{S2-fields-condition}
\phi_0\circ S^2(v)=\eta_0\in \Hom(S^2V(g),\om_{C_0}).
\end{equation}

Next, let us analyze the condition $v\in Q(\LL(c_0))$, i.e., $\det(\iota(v),v)=c_0$. We claim that if we view $v$ as a map of 
vector bundles $V(g)\to V_0$ then it is equivalent to $\det(v)=1$.
Indeed, $\det(\iota(v),v)$ is the determinant of the map of vector bundles $\pi^*V(g)\to \iota^*(\LL_0)\oplus \LL_0$,
where we use the isomorphism $\iota^*(\LL_0)\ot \LL_0\simeq \pi^*\om_{C_0}$. 
Equivalently, it is given as $F\circ\pi^*\det(v)$ where we view $v$ as the map of vector bundles $V(g)\to \pi_*\LL_0$,
and $F$ is the map
$$\det \pi^*(\pi_*\LL_0)\to \det (\iota^*(\LL_0)\oplus \LL_0)\simeq \iota^*(\LL_0)\ot\LL_0\simeq \pi^*\om_{C_0}.$$
It remains to check that under the natural trivialization of $\det \pi_*\LL_0$, $F$ coincides with $c_0$. We can argue over the general point of $C_0$.
Then as in the proof of Lemma \ref{Higgs-corr-lem}, we use the fact that the natural isomorphism \eqref{norm-det-isom} sends
$(u\we 1)\ot \Nm(s)$ to $us\we s$. Since $F(us\we s)=\Nm(s)$, the assertion now follows from the fact that $c_0$ is induced by $u$ (see the proof of Lemma \ref{Higgs-corr-lem}).



Thus, in our summation, $v$ has to be an isomorphism $V(g)\rTo{\sim}V_0$ with determinant $1$, and the condition 
\eqref{S2-fields-condition} simply says that it is compatible with the Higgs fields, so we get an isomorphism of $\SL_2$-Higgs bundles
$(V(g),\eta_0)\rTo{\sim} (V_0,\phi_0)$.
Note that the Higgs bundle $(V_0,\phi_0)$ has no invariant line subbundles (since $-\det(\phi_0)=\de$ is not a square), 
so it has only $\pm 1$ as automorphisms of determinant $1$. This leads to the claimed formula.

\noindent
(ii) This follows easily from (i) and from Theorem \ref{two-reps-two-funct-thm} and Proposition \ref{Fourier-prop}.

\noindent
(iii) The function $\Phi_{\LL_0,\chi}$ is obtained by the Fourier transform of Proposition \ref{Fourier-prop} from the function 
$$\sum_{\xi\in \PP_0(\wt{C}_0/C_0)} \chi(\xi) f^{Higgs}_{\xi\ot \LL,c_0}.$$
The fact that this function is nonzero follows from (i) and the fact that the Higgs $\SL_2$-bundles associated with $(\xi\ot \LL,c_0)$
are all non-isomorphic (by Lemma \ref{Higgs-corr-lem}).

For every point $p\in C_0$, the corresponding antiinvariant element of the extension of local fields, $c_{0,p}$, is given by $c_0/du$, if 
$\pi:\wt{C}_0\to C_0$ is unramified at $p$, and by $c_0/udu$ if $p$ is a ramification point 
(where $u$ is a local parameter on $\wt{C}_0$). It follows that $v(c_{0,p})=0$ away from the ramification locus and $v(c_{0,p})=1$ if $p$ is
a ramification point. Now the assertion follows from the calculation of Hecke operators in these cases given in Propositions \ref{Hecke-non-split-0-prop},
\ref{Hecke-ram-0-prop} and \ref{Hecke-split-0-prop}.
\end{proof}


\end{document}